\documentclass[a4paper,11pt]{article}

\usepackage[utf8]{inputenc}
\usepackage[T1]{fontenc}
\usepackage{textcomp}
\usepackage{lmodern}
\usepackage{amsmath,amsthm,amssymb,mathtools}
\usepackage{graphicx}
\usepackage[textwidth=135mm, textheight=195mm]{geometry}
\usepackage[english]{babel}

\newtheorem{lemma}{Lemma}
\newtheorem{prop}{Proposition}
\newtheorem{remark}{Remark}
\newtheorem{theorem}{Theorem}
\newtheorem{corollary}{Corollary}

\newcommand{\D}{{\mathcal{D}}}
\newcommand{\R}{{\mathbb{R}}}
\newcommand{\e}{\eps}
\newcommand{\ds}{\displaystyle}
\newcommand{\eps}{\epsilon}

\newcommand{\dsigma}{\,d\sigma}
\newcommand{\dx}{\,dx}
\newcommand{\1}{1\hspace{-.55ex}\mbox{\rm l}}

\newcommand{\abs}[1]{\lvert#1\rvert}

\newcommand{\set}[1]{\{#1\}}

\renewcommand{\div}{\operatorname{div}}

\DeclareMathOperator*{\argmin}{arg\,min}
\newcommand{\Set}[2]{\left\lbrace#1\ ;\ #2\right\rbrace}
\newcommand{\Scal}[2]{\left\langle#1\,,\,#2\right\rangle}

\newcommand{\diffp}[3][]{\frac{\partial^{#1} #2}{{\partial #3}^{#1}}}
\renewcommand{\H}{\mathcal{H}}
\renewcommand{\S}{\mathbb{S}}
\newcommand{\F}{\mathcal{F}}
\graphicspath{{pdf/}}
\renewcommand{\div}{\textup{div}}

\graphicspath{{img/}}

\title{Consistency result for a non monotone scheme for
        anisotropic mean curvature flow}
\author{Eric BONNETIER \\
        {\footnotesize LJK, Universit\'e de Joseph Fourier, B.P. 53, 38041 Grenoble Cedex 9, France,} \\
           {\footnotesize \it Eric.Bonnetier@imag.fr } \\
         ~\\
         Elie BRETIN \\
       {\footnotesize CMAP, Ecole Polytechnique, CNRS, 91128 Palaiseau, France, }\\
         {\footnotesize \it bretin@polytechnique.fr } \\
         ~\\ 
         Antonin CHAMBOLLE \\
        {\footnotesize CMAP, Ecole Polytechnique, CNRS, 91128 Palaiseau, France } \\
         {\footnotesize \it antonin.chambolle@polytechnique.fr } }
\date{Juin 2010}

\begin{document}

\maketitle

\begin{abstract}

In this paper, we propose a new scheme for anisotropic motion by mean curvature in $\R^d$.
The scheme consists of a phase-field approximation of the motion, where the nonlinear diffusive
terms in the corresponding anisotropic Allen-Cahn equation are linearized in the Fourier space.
In real space, this corresponds to the convolution with a kernel of the form
\[
K_{\phi,t}(x) = \F^{-1}\left[ e^{-4\pi^2 t \phi^o(\xi)} \right](x).
\]
We analyse the resulting scheme,
following the work of Ishii-Pires-Souganidis on the convergence of the
Bence-Merriman-Osher algorithm for isotropic motion by mean curvature.
The main difficulty here, is that the kernel $K_{\phi,t}$ is not positive
and that its moments of order 2 are not in $L^1(\R^d)$. Still, we can show 
that in one sense the scheme is consistent with the anisotropic mean curvature flow.
\end{abstract}

%\tableofcontents

\section{Introduction and motivation}
%===============================================================================

In the last decades, a lot of attention has been devoted to the motion of interfaces, 
and particularly to motion by mean curvature. 
Applications concern image processing (denoising, segmentation), material sciences 
(motion of grain boundaries in alloys, crystal growth), biology (modelling of vesicles 
and blood cells). 
 
\subsection{Motion by isotropic mean curvature}

The simplest case of motion by isotropic mean curvature concerns
the evolution of a set $\Omega_t \subset \R^d$ with a boundary $\partial \Omega_t$ of
codimension 1, whose normal velocity $V_n$ is proportional to its mean curvature $\kappa$ 
\begin{equation} \label{eq:motion_mc}
   V_n(x) = \kappa(x),
	\quad \textrm{a.e.}\; x \in \Gamma_t,
\end{equation}
with the convention that $\kappa$ is negative    
if $\Omega_t$ is a convex set.
It at $t=0$ the initial set $\Omega_0$ is smooth, then the evolution is well-defined 
until some time $T > 0$ when singularities may develop~\cite{Ambrosio2000}.
\vspace*{5mm}

Viscosity solutions provide a more general framework, that defines evolution past singularities,
or evolution from non-smooth initial sets.
If $g$ is a level set function of $\Omega_0$, i.e.,
$$\Omega_0 = \Set{x \in \R^d}{g(x) \leq 0}, \quad \partial \Omega_0 = \Set{x \in \R^d}{g(x) = 0},$$
and if $u$ denotes the viscosity solution to the Hamilton-Jacobi equation
$$  \begin{cases}
     u_t = \div \left( \frac{\nabla u}{|\nabla u| } \right) |\nabla u| \\
     u(0,x) = g(x),
    \end{cases}
$$ 
then the generalized mean curvature flow $\Omega_t$ starting from $\Omega_0$ is 
defined by the 0-level set of $u$~\cite{Evans1992,OsherSethian,Chen_giga,Evans_spruck}   
$$\Omega_t = \Set{x \in \R^d}{u(t,x) \leq 0}, \quad \partial \Omega(t) = \Set{x \in \R^d}{u(t,x) = 0}.$$
\vspace*{5mm}

Alternatively, one can define the motion by mean curvature as the limit of diffuse interface
approximations obtained by solving the Allen-Cahn equation
\begin{equation}
\label{eq:ac}
  \diffp{u}{t}
  = \Delta u - \frac{1}{\eps^2} W'(u) ,
\end{equation} 
where $\eps$ is a small parameter
(that determines the width of the diffuse interface) and where $W(s) = \frac{s^2(1-s)^2}{2}$
is a double well potential.
This equation can be viewed as a gradient flow for the energy
\[
  J_\eps(u)
  = \int_{\R^d} \left( \frac{\eps}{2} \abs{\nabla u}^2 + \frac{1}{\eps} W(u) \right) \dx.
\]
Modica and Mortola~\cite{Modica1977, Modica1977a} have shown that 
$J_\e$ approximates (in the sense of $\Gamma$- convergence) the surface energy $c_W \, J$
where
\[
J(\Omega) = \ds\int_{\partial \Omega} 1\, d\sigma
\quad \textrm{and}\quad
c_W = \ds\int_0^1 \sqrt{2 W(s)} \,ds.
\]
Existence, uniqueness, and a comparison principle have been established for~\eqref{eq:ac} 
(see for example chapters 14 and 15 in \cite{Ambrosio2000} and the references therein).

Let $u_\eps$ solve \eqref{eq:ac} with the initial condition
\[
  u_\eps(x,0) = q\left(\frac{d(x, \Omega_0)}{\eps}\right),
\]
where $d(x,\Omega)$ denotes the signed distance of a point $x$ to the set $\Omega$
and where the profile $q$ is defined by

\begin{eqnarray*}
 q &=& \argmin \left\{ \int_\R \left( \frac{1}{2} {\gamma'}^2(s)
    + W(\gamma(s)) \right)   ds ;\gamma \in H^1_{loc}(\R),\ \gamma(-\infty) = +1, \right. \\
  &~& \quad  \quad \quad \quad  \quad  \quad \quad \quad  \quad \quad \quad  \quad  \quad \quad  \quad \quad  \quad \quad \quad   \left.  \gamma(+\infty) = -1,\ \gamma(0) = \frac{1}{2} \right\}.
\end{eqnarray*}

Then, for smooth motion by mean curvature~\cite{Chen1992, Bellettini1995},
or for generalized motion by mean curvature without fattening~\cite{Barles1994, Evans1992},
the set
\[
  \Omega_\eps(t) =  \Set{x \in \R^d}{u_\eps(x,t) \ge \frac{1}{2}},
\]
approximates $\Omega(t)$ at the rate of convergence $O(\eps^2 \abs{\log \eps}^2)$. 
\vspace*{5mm}

The Bence-Merriman-Osher algorithm~\cite{BenceMerrimanOsher} is yet another approximation to motion by mean curvature. Given a closed set $E \subset \R^d$, and denoting $\chi_E$ its characteristic function,
one defines
$$ T_hE = \Set{x \in \R^d}{ u(x,h) \geq \frac{1}{2}}, $$   
where $u$ solves the heat equation
$$ \begin{cases}
   \label{eq:heat}
    \frac{\partial u}{\partial t}(x,t) = \Delta u(x,t), & t > 0 \quad x \in \R^d \\
    u(x,0) = \chi_{E}(x).
   \end{cases}
$$
Setting $E_h(t) = T^{[t/h]}E$, where $[t/h]$ is the integer part of $t/h$,
Evans~\cite{Evans_BMO}, and Barles and Georgelin~\cite{Georgelin_Barles} have shown that 
$E_h(t)$ converges to $E_t$, the evolution by mean curvature from $E$.

\subsection{Motion by anisotropic mean curvature}
  
We use the framework of the Finsler geometry as described in~\cite{bellettini_1996}. 
Let $\phi : \R^d \to [0,+\infty[$  denote a strictly convex function in 
$C^2(\R^d \setminus \set{0}))$, which is 1-homogeneous and bounded, i.e.,
\[
\left\{ \begin{array}{ll}
\phi(t \xi) = |t| \phi(\xi) & \xi \in \R^d, \; t \in \R, 
\\
\lambda |\xi| \leq  \phi(\xi) \leq \Lambda |\xi| & \xi \in \R^d,
\end{array} \right.
\]
for two positive constants $0< \lambda \leq \Lambda < + \infty$. 
We assume that its dual function $\phi^o : \R^N \to [0,+\infty[$, defined by 
 $$ \phi^o(\xi^*) = \sup\Set{\xi^*.\xi}{ \phi(\xi) \leq 1}$$
is also in $C^2(\R^N \setminus \set{0}))$. 
Given a smooth set $E$ and a smooth function $u : \R^d \to \R$ such that 
$\partial E = \Set{x \in \R^d}{ u(x) = 0}$, we define 
 \begin{itemize}
  \item  
  the Cahn-Hoffman vector field $n_{\phi} =  \phi^o_{\xi}(\nabla u)$.
  \item  
  the $\phi$-curvature $\kappa_{\phi} = \div(n_{\phi})$.   
 \end{itemize}
We say that $E(t)$ is the evolution from $E$ by $\phi$-curvature, if at each time $t$,
the normal velocity $V_n$ is given by
$$ V_n = - \kappa_{\phi} n_{\phi}.$$
%\begin{verbatim}
%add a remark about existence of smooth flows ?
%add sthg on the Wulff shape and frank diagrams and how they are related to 
%\phi and \phi^o
%\end{verbatim}
%\vspace*{5mm}

As in the case of isotropic flows, one can define motion by $\phi$-curvature using
a level set formulation, i.e., following the level lines of the solution to the
anisotropic Hamilton-Jacobi equation
\begin{equation} \label{eq:hamilton_jacobi_phi}
  u_t =  \phi^o(\nabla u )~ \phi^o_{\xi\xi}( \nabla u ) : \nabla^2 u.
\end{equation}
Existence, uniqueness and a comparison principle have been etablished in  \cite{Crandall_Ishii_Lions,chen_giga_goto,MR1115933,MR1205984}. 
% or later...
\vspace*{5mm}

The anisotropic surface energy 
\[
  J(\Omega) = \int_{\partial \Omega} \phi^o(n) \dsigma.
\]
can be approximated by the Ginzburg-Landau-like energy
\[
  J_{\eps,\phi}(u)
  = \int_{\R^d} \left( \frac{\eps}{2} \phi^o(\nabla u)^2
    + \frac{1}{\eps} W(u)\right) \dx,
\]
and its gradient flow leads to the anisotropic Allen--Cahn equation \cite{Allen1979}
\begin{equation}
\label{eq:ac_ani}
  \diffp{u}{t}
  = \Delta_{\phi} u - \frac{1}{\eps^2} W'(u).
\end{equation} 
The operator $\Delta_{\phi} := \div \left( \phi^o_{\xi}(\nabla u ) \phi^o(\nabla u) \right)$
is called the anisotropic Laplacian.
\vspace*{5mm}
 
The Bence-Merriman-Osher algorithm has also been extended to anisotropic motion by mean curvature.
One generalization was proposed by Chambolle and Novaga~\cite{Chambolle_novaga2006} as follows:
Given a closed set $E$, let $T_h(E) = \Set{x \in \R^d}{u(x,h) \geq \frac{1}{2}}$, 
where $u(x,t)$ is the solution to  
$$ \begin{cases}
   \label{eq:anisotropic heat}
    \dfrac{\partial u}{\partial t}(x,t) = \Delta_{\phi} u(x,t), & t > 0 \quad x \in \R^d \\
    u(x,0) = \chi_{E}(x).  
   \end{cases}
 $$
Define then $E_h(t) = T_h^{[t/h]}E$.
The convergence of $E_h(t)$ to the generalized anisotropic mean curvature flow from $E$
is established in~\cite{Chambolle_novaga2006}. The result holds for very general anisotropic
surface tensions and even in the cristalline case.
However, because of the strongly nonlinear character of $\Delta_\phi$, the numerical
resolution of~(\ref{eq:anisotropic heat}) is much harder than in the isotropic case. 
 
Another generalization of the Bence-Merriman-Osher algorithm has been studied by Ishii, Pires and Souganidis ~ 
\cite{Ishhi_pires_souganidis}. 
The main idea is to represent the solution $u$ of~(\ref{eq:anisotropic heat})
as the convolution of $\chi_E$ with a geometric kernel. 
More precisely, Let $f : \R^d \to \R$ be a function which satisfies the following conditions
  \begin{itemize}
  \item[$(A_1)$] 
  Positivity and symmetry : 
  $$ f(x) \geq 0, \quad f(-x)=f(x), \quad \text{and } \int_{\R^d} f(x) dx = 1$$
  \item[$(A_2)$] 
  Boundedness of the moments :
  $$ \int_{\R^d} |x|^2 f(x) dx < + \infty,$$
   $$ 0 < \int_{p^{\perp}} (1 + |x|^2) f(x) d \H^{d-1} < \infty, \quad \text{ for all  } p \in \S^{d-1}.$$
  \item[$(A_3)$] 
  Smoothness :
   $$ p \to  \int_{p^{\perp}} f(x) d \H^{d-1}  \text{ and }  
   p \to  \int_{p^{\perp}} x_i x_j f(x) d \H^{d-1}
   \quad \text{are continous  on}\;\S^{d-1}.
   $$
  \end{itemize}
Given $E \subset \R^d$, define $T_hE = \Set{x \in \R^d}{u(x,h) \geq \frac{1}{2}}$,
where 
$$ u(x,h) = \int_{\R^d} \tilde{K}_h(y) \,\chi_E(y-x) \,dy,$$
with the kernel  
$$ \tilde{K}_t(x) = \frac{1}{t^{d/2}} f( \sqrt{t} x), \quad x \in \R^d. $$
They showed  \cite{Ishhi_pires_souganidis}  that $T_h^{[t/h]}E$ converges to the set $E(t)$
obtained from $E$ as the generalized motion by anisotropic mean curvature via
the Hamilton Jacobi equation
 $$ u_t = F(D^2 u, D u)   $$
where 
$$ F(X,p) = \left( \int_{p^{\perp}} f(x) d \H^{d-1}(x)  \right)^{-1} \left( - \frac{1}{2} \int_{p^{\perp}} \Scal{Xx}{x} f(x) d \H^{d-1}(x)  \right). $$

This result raises a natural question: Given an anisotropy $\phi^o$, can one find a
kernel $f$, so that the generalized fromt $\partial E(t)$ defined by the associated 
Hamilton Jacobi equation evolves by $\phi$-mean curvature ? 
This problem has been addressed by Ruuth and Merriman~\cite{Ruuth_merriman1} in dimension~2:
They propose a class of kernels and study the corresponding numerical schemes, which
prove very efficient. However, their appraoch cannot be generalized to higher dimensions.
In contrast, our algorithm is not specific to the dimension~2.
\vspace*{5mm}
%%----------------------------------------------------------------------------------------

\subsubsection{A new algorithm for motion by anisotropic mean curvature}

In this work, our objective is to extend Ishii-Pires-Souganidis' analysis to study 
the following algorithm.
Starting from a bounded closed set $E \subset \R^d$, we define an operator $T_h E$ by 
\begin{eqnarray} \label{def_Th} 
T_h E = \Set{x \in \R^d}{ u(x,h) \geq \frac{1}{2}}, 
\end{eqnarray}
where $u$ solves the following parabolic equation:
$$ 
(2)~  \begin{cases}
    \frac{\partial u}{\partial t}(x,t) = \tilde{\Delta}_{\phi} u(x,t), & t > 0 \quad x \in \R^d \\
    u(x,0) = \chi_{E}(x).  
   \end{cases}
$$
Denoting by $\F(u)$ the Fourier transform of a function $u$,
\begin{eqnarray*}
\F(u)(\xi) &=& \ds\int_{\R^d} u(x) e^{-2\pi i x \cdot \xi} \, dx,
\end{eqnarray*}
the operator $\tilde{\Delta}_{\phi}$ is defined by
$$ 
\tilde{\Delta}_{\phi} u = \F^{-1} \left(  -4 \pi^2 \phi^o(\xi)^2 \F(u)(\xi) \right).
$$
and can be seen as a linearization of $\Delta_{\phi}$ in the Fourier space. 
The solution  $u$ of $(2)$ can be expressed  as a convolution product of 
the characteristic function of $E$ and of the anisitropic  kernel
$$
K_{\phi,t}(x) =  \F^{-1}\left( e^{-4\pi^2 t \phi^o(\xi)^2} \right)(x). 
$$
However, this kernel (more precisely $K_{\phi,t=1}$) does not satisfy the hypotheses $(A_1)$
and $(A_2)$ above: $K_{\phi,1}$ is not positive and 
$x \to  \int_{\R^d} |x|^2 K_{\phi}(x)$ is not in $L^1(\R)$.
In section~2, we etablish some properties of the anisotropic heat kernel $ K_{\phi}$. 
Precisely, we prove that the associated Hamiltonian flow is

\begin{eqnarray*}
 F(X,p) &=& \left(  \int_{p^{\perp}} K_{\phi} d\H^{d-1}   \right)^{-1} \left( \frac{1}{2} \int_{p^{\perp}} <Xx,x> K_{\phi}(x) d\H^{d-1}  \right) \\
&=& \phi^o(p) \phi^o_{\xi \xi}(p):X,
\end{eqnarray*}

which establishes a link between $K_{\phi}$ and  $\phi$-anisotropic mean curvature flow.

In section~3, we establish the consistency of a Bence-Merriman-Osher 
scheme based on~(\ref{def_Th}). We have however not been able to prove the convergence of the algorithm
to $\phi$-anisotropic mean convergence in the general setting of uniformly bounded
and continuous functions. The main difficulty in trying to extend the
argument of \cite{Ishhi_pires_souganidis}. is the thresholding and the lack of monotonicity of our sheme that may not preserve the continuity of the front. 

Therefore, in the last section, we present numerical evidence of the convergence
of a modified scheme. In this scheme, the thresholding is obtained via a reaction
term, in the spirit of phase-field approximation. Computationally, the scheme
proves very efficient and very fast, even when the anisotropy  $\phi^o$  is not smooth.

%==========================================================================================

\section{The operator $\tilde{\Delta}_{\phi}$ and properties of the anisotropic kernel $K_{\phi}$}

Let $\phi = \phi(\xi)$ denote a strictly convex smooth Finsler metric 
and let $\phi^o$ denote its dual~(see \cite{bellettini_1996}). 
We assume that that $\phi^o$ is a $1$-homogenous, symmetric function 
in~$C^{\infty}(\R^d \setminus \set{0})$ that satisfies 
\begin{eqnarray} \label{bounds_phi^o}
\lambda |\xi| \leq \phi^o(\xi) \leq \Lambda |\xi|.
\end{eqnarray}
In particular, it follows that for any $\xi \in \R^d$ and $t \in \R$,
$$ \begin{cases}
      \phi^o(t \xi) = |t| \phi^o(\xi) \\
      \phi^o_{\xi}(t \xi) = \frac{t}{|t|} \phi^o_{\xi}(\xi) \\
      \phi^o_{\xi}(\xi).\xi = \phi^o(\xi). 
     \end{cases}
$$
The associated anisotropic mean curvature is defined as the anisotropic Laplacian operator
$$ \triangle_{\phi} u =  \div \left( \phi^o(\nabla u) \phi^o_{\xi}(\nabla u) \right), 
\quad \forall u \in H^2(\Omega)
$$ 

%%------------------------------------------------------------------------------------------
A direct computation shows that for any $\xi \in \R^d$, 
$$ \begin{cases} \label{prop:vecteur_propre}
      \triangle_{\phi} \left[ \cos(2 \pi \xi.x ) \right]  = -4 \pi^2 \phi^o(\xi)^2 \cos(2 \pi \xi.x ) \\
      \triangle_{\phi} \left[ \sin(2 \pi \xi.x ) \right]  = -4 \pi^2 \phi^o(\xi)^2 \sin(2 \pi \xi.x ), \\
    \end{cases}
$$
i.e., that plane waves are eigenfunctions of the anisotropic Laplacian (albeit nonlinear).
We define $\tilde{\triangle}_{\phi} : H^2(\R^d) \rightarrow L^2(\R^d)$ by 
$$ 
\tilde{\triangle}_{\phi} u = \F^{-1}\left[-4\pi^2\phi^o(\xi)^2 \F[u](\xi)  \right], 
$$
Given an initial condition $u_0 \in L^2(\R^d)$, we study the solution $u$ of, 
$$ \begin{cases}
    u_t(t,x) = \tilde{\triangle}_{\phi} u(t,x), \\
    u(0,x) =  u_0  
   \end{cases}
$$
The function $u$ can also be  expressed as the convolution product $u = K_{\phi,t}*u_0$, where the anisotropic heat kernel $K_{\phi,t}$  is defined by
$$ K_{\phi,t} = \F^{-1} \left[ e^{-4 \pi^2 t \phi^o(\xi)^2} \right].$$
We also set $K_\phi = K_{\phi,1}$.
In the rest of this section, we establish some properties of this operator.
\vspace*{5mm}
%%-------------------------------------------------------------------------------------------
 
\begin{prop}[Regularity of $\hat{K}_{\phi}$] \label{prop_regularite}~\\
The function  
$\hat{K_{\phi}}:\xi \to e^{-4 \pi^2 \phi^o(\xi)^2}$
is in $W^{d+1,1}(\R^d)$, and  the distribution $ D^{d+2} \hat{K_{\phi}} $ is a regular function.
\end{prop}
%%-------------------------------------------------------------------------------------------

\begin{proof}
First, we claim that the Hessian of $\hat{K_{\phi}}$ is a regular distribution since
$$ D \hat{K_{\phi}}(\xi) =  -8 \pi^2 \phi_{\xi}^o(\xi) \phi^o(\xi) e^{-4 \pi^2 \phi^o(\xi)^2},$$
and 
\begin{eqnarray*}
  D^2 \hat{K_{\phi}}(\xi) &=&  64 \pi^4  \phi^o(\xi)^2 \left(  \phi_{\xi}^o(\xi)\otimes  \phi_{\xi}^o(\xi) \right) e^{-4 \pi^2 \phi^o(\xi)^2} \\
 &~& \quad \quad- 8 \pi^2  \left(  \phi^o(\xi) \phi_{\xi \xi}^o(\xi)  +  \phi_{\xi}^o(\xi)\otimes  \phi_{\xi}^o(\xi) \right) e^{-4 \pi^2 \phi^o(\xi)^2}.
\end{eqnarray*}
We note that $\phi^o_\xi$ is discontinuous at $\xi = 0$.
Nevertheless, we next prove that the $d-1^{th}$ derivative of  $D^2 \hat{K_{\phi}}$ is a regular distribution,  without Dirac mass at $\xi=0$. 
Assume that $f = D^{n+2}  \hat{K_{\phi}}$ is an integrable function on $\R^d$ 
for some integer $n<d$. 
The homogeneity of $\phi^o$ shows the existence of a constant  $C_n$ such that 
$$ | D^{n+2}  \hat{K_{\phi}}| \leq C_n  \frac{1}{|\xi|^n} e^{- \lambda |\xi|^2}, \quad \text{for all } \xi \in \R^d \setminus \set{0}.$$
Since $f$ is smooth away from $\xi=0$,
the distributional  derivative  of $f$  is the sum of a regular function and of possibly
a Dirac mass at $\xi = 0 $ : 
$$ 
D f   =    \left\{  \nabla f \right\}  + c~\delta,
$$
where $c$ is a constant and $\nabla f$ denotes the pointwise derivative of $f$. 
Let  $\varphi \in \D(\R^d)^{d^{n+2}}$  and let $\epsilon > 0$. Then 
\begin{eqnarray*}
 \Scal{Df}{ \varphi} &=& -\Scal{f}{ \div \varphi } = - \int_{\R^d} f. \div \varphi dx \\
  &=& - \int_{\R^d \setminus B(0, \epsilon)}   f.\div \varphi dx -  \int_{B(0, \epsilon)}   f.\div \varphi dx \\
  &=&  \int_{\R^d \setminus B(0, \epsilon)}   \nabla f.\varphi dx - \int_{\partial  B(0, \epsilon) }  f.(\varphi.\vec{n}) d\sigma -   \int_{B(0, \epsilon)}   f.\div \varphi dx,
\end{eqnarray*}
Since we assumed that $f \in L^1(\R^d)^{d^{n+2}}$, the last integral above tends to $0$,
as $\e \rightarrow 0$.
Moreover  as  $n < d$, we have 
\begin{eqnarray*}
 \left| \int_{\partial  B(0, \epsilon) }  f ~ \varphi.\vec{n} d\sigma \right| &\leq&    \| \varphi \|_{L^{\infty}} \int_{\partial B(0, \epsilon)} C_n  \frac{1}{|\xi|^n} e^{- \lambda |\xi|^2} d\sigma  \\ &\leq&   \| \varphi \|_{L^{\infty}} C_n  \int_{\partial B(0, \epsilon)} \epsilon^{-n} d\sigma \leq C_n \| \varphi \|_{L^{\infty}} \eps^{d-1-n},  
 \end{eqnarray*}
 so that 
 $$ \lim_{\epsilon \to 0}  \left| \int_{\partial  B(0, \epsilon) }  f ~ \varphi.\vec{n} d\sigma  \right| = 0.$$
 It follows that $c=0$, which concludes the proof.
\end{proof}

%%-------------------------------------------------------------------------------------------

\begin{prop}[Decay properties of $K_{\phi}$]~\\ \label{prop_decrease}
Let $s \in [0,1[$. 
There exists a constant $C_{\phi^o,s}$, which only depends on the anisotropy $\phi^o$ and on $s$,  
such that
\begin{eqnarray} \label{estim_Kphi}
|K_{\phi}(x)| &\leq& \frac{C_{\phi^o,s}}{1 + |x|^{d+1+s}}, \quad \forall x \in \R^d. 
\end{eqnarray}
\end{prop}
%%-------------------------------------------------------------------------------------------

\begin{remark} The case $s = 0$ is easy: According to proposition~\ref{prop_regularite}, 
the function $\triangle^{\frac{d+1}{2}}\hat{K}_{\phi}(\xi)$ is in $L^1(\R^d)$. 
The continuity of the Fourier transform from $L^1$ to $L^{\infty}$ shows that  
$$ 
\| (1 + |x|^{d+1}) K_{\phi} \|_{L^{\infty}} \leq  
C \| \hat{K}_{\phi}(\xi) + \triangle^{\frac{d+1}{2}}\hat{K}_{\phi}(\xi) \|_{L^1(\R^d)},
$$
and since $\hat{K}_\phi(\xi) = e^{-4\pi^2 \phi^o(\xi)^2}$, 
$$ 
 |K_{\phi}(x)| \leq \frac{C_{\phi^o,0}}{1 + |x|^{d+1}}, \quad \forall x \in \R^d. 
$$
\end{remark}
%%-------------------------------------------------------------------------------------------
%%\begin{remark}
%%The above result for $s>0$ garantees that the integral 
%%$\int_{p^{\perp}} x \otimes x K_{\phi} dx$ 
%%are well-defined, for any $p \in \mathbf{S}^d$. 
%%\end{remark}
%%-------------------------------------------------------------------------------------------
 
The proof  uses properties of interpolation spaces~\cite{BerghLofstrom}. 
Consider $X$, $Y$ two Banach spaces, and for $u \in X + Y$ and $t \in \R^+$, let  
$$
k(t,u) = \inf_{u = u_0 + u_1} \left\{ \| u_0\|_{X} + t \|u_1\|_{Y}  \right\}.
$$
For $s \in [0,1]$ and $p \geq 1$, the interpolation space $[X,Y]_{s,p}$ beetween $X$ and $Y$ is
defined  by
$$
[X,Y]_{s,p} = \Set{u \in X + Y}{ t^{-s} K(t,u) \in L^p \left ( \frac{1}{t} \right)}.
$$
In particular, given a strictly positive function $h : \R^d \to \R$, consider the weighted space
$L^{\infty}_{h}$ defined by
$$
L^{\infty}_{h}(\R^d) = \Set{ u \in L^{\infty}(\R^d) }{ \sup_{x \in \R^d} \set{ h(x) u(x)} < \infty }.  $$
One can interpolate between $L^{\infty}$ and $L^{\infty}_h$ according to the following lemma.
%%-------------------------------------------------------------------------------------------
\begin{lemma} \label{lemma:interpolation_L_infty}
Let $h$ be a strictly positive function $\R^d \to \R$, and let $s \in ]0,1[$.
Then
$$  
[L^{\infty}(\R^d),L_{h}^{\infty}(\R^d)]_{s,\infty} = L_{h^{s}}^{\infty}(\R^d) 
$$
\end{lemma}
%%-------------------------------------------------------------------------------------------
\begin{proof} 1) 
%%{\it $ \supseteq$ inclusion : }\\
Assume that $u \in L_{h^{s}}^{\infty}(\R^d)$. 
There exists  a constant $C$ such that for a.e. $x \in \R^d$,   
\begin{eqnarray} \label{estim_u}
|u(x)| \leq  \frac{C}{ h(x)^s }. 
\end{eqnarray}
To estimate 
$k(t,u) =  \inf_{u = u_0 + u_1} \left\{ \| u_0\|_{L^{\infty}} + t \|u_1\|_{L^{\infty}_{h}}  \right\}$, we note that
\begin{itemize}
\item 
If $t \geq 1$, the choice $u_0 = u$ and $u_1=0$, shows that 
$K(t,u) \leq \| u\|_{L^{\infty}}$.
\item 
If $t<1$, we consider the set $A = \Set{x \in \R^d}{ |u(x)|h(x) \leq t^{s-1} }$, 
and we choose $u_0 = \chi_{A^{c}}~u$ and $u_1 = \chi_{A}~u $, so that 
$\| u_1 \|_{L^{\infty}_h} \leq t^{s-1}$. 
Moreover, we remark that for all $x \in A^{c}$,
$ |u(x)| h(x) \geq t^{s-1}$ so that, in view of~(\ref{estim_u})
\begin{eqnarray*}
|u_0(x)| &\leq& C h(x)^{-s} \;\leq\; C |u_0(x)|^s t^{s(1-s)},
\end{eqnarray*} 
and thus $k(t,u) \leq (C+1)t^{s}$. 
\end{itemize}
In summary, these estimates show that  
$$
K(t,u) \leq \begin{cases}
             \| u\|_{L^{\infty}} & \text{ if } t \geq 1 \\
             (C + 1)t^{s}  &   \text{ if } t < 1, 
            \end{cases}
$$
which proves that $u \in [L^{\infty},L_{h}^{\infty}]_{s,\infty}$.
\\
2) Conversely, we consider $u \in [L^{\infty},L_{h}^{\infty}]_{s,\infty}$. 
For all $t>0$, there exists a decomposition  $u = u_{0,t} + u_{1,t}$ such that 
$$
|u_{0,t}|_{L^{\infty}} + t| u_{1,t}|_{L^{\infty}_{h}} \leq C t^{s}.  
$$ 
It follows that for all $t>0$, we have  
\begin{eqnarray*}
h(x)^s |u(x)| &\leq& 
\left|h(x)^s \, |u_{0,t}(x) + u_{1,t}(x) \right|  
\;\leq\;  C \left( h(x)^s  t^{s} +  h(x)^{s-1} t^{s-1} \right).  
\end{eqnarray*}
Choosing $t = h(x)^{-1}$ in the above inequality shows that for all $x \in \R^d$, \\
$h(x)^s \, |u(x)| \leq 2C$, which concludes the proof.
\end{proof}
\vspace*{5mm}

%%-------------------------------------------------------------------------------------------

We use the following properties of interpolation spaces:
\begin{itemize}
 \item[$(P_1)$]   if $T$ is continous from $X \to \tilde{X}$ and from $Y \to \tilde{Y}$,  then $T$ is continous from $[X,Y]_{s,p}$  to $[\tilde{X},\tilde{Y}]_{s,p}$.
 \item[$(P_2)$] if $p<p'$, then  $[X,Y]_{s,p} \subset [X,Y]_{s,p'}$ \quad  for any $0< s< 1$ and  $p\geq 1 $.
 \item[$(P_3)$] $[L^{\infty}(\R^d),L^{\infty}_{(1+|x|)}(\R^d)]_{s,\infty} =  L^{\infty}_{(1+|x|)^s}(\R^d)$\quad  for any $0< s< 1$.
\end{itemize}
In the following, we consider the case where  $T$ is the Fourier transform, $X = L^1(\R^d)$, 
$Y = L^{\infty}(\R^d)$, $\tilde{X} = W^{1,1}(\R^d)$ and $\tilde{Y} =L^{\infty}_{(1+|x|)}(\R^d)$. 
\vspace*{5mm}

%%-------------------------------------------------------------------------------------------
\begin{proof}[Proof of Proposition~\ref{prop_decrease}]
We claim that it suffices to show that for any  $0<s<1$
\begin{eqnarray} \label{claim_u}
u(\xi) &:=& \triangle^{\frac{d+1}{2}}\hat{K}_{\phi}(\xi) \;\in [X,Y]_{s,1}.
\end{eqnarray}
Indeed, the inclusion  $[X,Y]_{s,1} \subset [X,Y]_{s,\infty}$ implies then that
$u \in [X,Y]_{s,\infty}$, so that in view of $(P_1)$ and $(P_3)$ we obtain
$$ 
\hat{u} \in [\tilde{X},\tilde{Y}]_{s,\infty} 
= [L^{\infty}(\R^d),L^{\infty}_{(1+|x|)}(\R^d)]_{s,\infty} =  L^{\infty}_{(1+|x|)^s}(\R^d),
$$
and consequently
$$
|(1 + |x|^s) \hat{u}(x)| = |(1 + |x|^{d+1}) K_{\phi}(x) (1+|x|)^s| \leq C_{\phi^o,s}, 
\quad  \text{ for all } x \in \R^d.
$$
It follows that for some constant $C_{\phi^o,s}$
$$ |K_{\phi}(x)| \leq \frac{C_{\phi^o,s}}{1 + |x|^{d+1+s}}, \quad  \text{for all } x \in \R^d.$$
\medskip

We now prove~(\ref{claim_u}). 
The homogeneity  of $\phi^o$ shows that for some $c_1>0$ and $c_2>0$, 
and for $\xi \in \R^d \setminus \set{0}$,
\begin{eqnarray*}
| u(\xi)| \leq  \frac{c_1}{|\xi|^{d-1}} e^{-\lambda|\xi|^2} 
&\quad \textrm{and} \quad&
|\nabla u(\xi)| \leq  \frac{c_2}{|\xi|^{d}} e^{-\lambda |\xi|^2},
\end{eqnarray*}
which shows that $u \in X = L^1(\R^d)$. 
However, $u$ may not belong to $Y = L^\infty(R^d)$.  
We now estimate $k(u,t)$, for $t \in \R^+$. 
If $t \geq 1 $,  we set $u_0 =u$, $u_1=0$, so that 
\begin{eqnarray} \label{estimation_K_phi_t_1}
k(t,u) \leq \| u \|_X, \quad \forall t \geq 1.
\end{eqnarray}
If $t < 1$, consider the function $\rho_t(\xi)$ defined by 
$$ \rho_t(\xi) = 
\begin{cases} 
0 & \text{if} \quad |x| \leq t \\
1 & \text{if} \quad |x| > 2 t \\
\sin\left( \frac{\pi}{2} \frac{ |\xi| - t }{t} \right) & \text{otherwise}. 
\end{cases}
$$ 
We choose $ u_0 = (1 - \rho_t) u$ and $u_1 =  \rho_t u$ and check that
\begin{eqnarray*}
|u_0|_{L^1(\R^d)}  &\leq& 
\int_{B(0,2t)} |u(\xi)|  d \xi \leq  \int_{B(0,2t)}\frac{C}{|\xi|^{d-1}} d\xi 
\leq 2 C  |\mathbf{S}^d|  t.
\end{eqnarray*}
Moreover, 
\begin{eqnarray*}
\| \nabla u_1 \|_{L^1(\R^d)} 
&\leq&   
\| \nabla \rho_t  u +   \rho_t  \nabla  u \|_{L^1(\R^d)} 
\\
&\leq& 
\int_{R^d \setminus B(0,t)}  |\nabla \rho_t| u(\xi) d\xi 
+ \int_{R^d \setminus B(0,t)}  |\nabla u(\xi)| d\xi 
\\
&\leq&  
\frac{\pi}{2 t } \int_{  B(0,2t)\setminus B(0,t)}  
\frac{C}{|\xi|^{d-1}} e^{-\lambda|\xi|^2}  d\xi 
+ \int_{R^d \setminus B(0,t)}  \frac{C}{|\xi|^{d}}e^{-\lambda|\xi|^2} d\xi.
\end{eqnarray*}
First, we have
\begin{eqnarray*}
\frac{\pi}{2 t }  \int_{ B(0,2t)\setminus B(0,t)}   \frac{C}{|\xi|^{d-1}} e^{-\lambda|\xi|^2} d\xi 
&\leq& 
\frac{ C \pi}{2 t } |\mathbf{S}^d| \int_{t}^{2t}  dr \leq  \frac{ |\mathbf{S}^d| C \pi}{2 }.
\end{eqnarray*}
Second,
\begin{eqnarray*}
\int_{R^d \setminus B(0,t)}  \frac{C}{|\xi|^{d}}e^{-\lambda|\xi|^2} d\xi 
&\leq&   
\int_{B(0,1) \setminus B(0,t)} 
\frac{C}{|\xi|^{d}}e^{-\lambda|\xi|^2} d\xi + 
\int_{\R^d \setminus B(0,1)}  
\frac{C}{|\xi|^{d}}e^{-\lambda|\xi|^2} d\xi 
\\
&\leq&  
C  |\mathbf{S}^d| \int_{t}^{1} \frac{1}{r} dr 
+ C  |\mathbf{S}^d| \int_{1}^{\infty} e^{-\lambda r^2 } dr 
\\
&\leq&  
C  |\mathbf{S}^d| \left( |\ln(t)| +   \frac{1}{\sqrt{\lambda}} \frac{\sqrt{\pi}}{2} \right),
\end{eqnarray*}
so that
\begin{eqnarray*}
\|  u_1 \|_{Y} 
&\leq& C \left[
|\mathbf{S}^d| \left( \frac{\pi}{2 } 
+ \frac{1}{\sqrt{\lambda}} \frac{\sqrt{\pi}}{2} + |\ln(t)| \right)
\right].
\end{eqnarray*}
Consequently, this decomposition of $u$ shows that 
\begin{eqnarray} \label{estimation_K_phi_t_2}
k(u,t) \leq C(1 +  |\ln(t)| ) t , \quad \forall t< 1,
\end{eqnarray}
for some constant $C > 0$.
In summary, 
$$
k(u,t) \leq 
\begin{cases}
\| u \|_X & \text{if} \quad  t \geq 1 \\
C( 1 +  |\ln(t)| )  t  &  \text{if} \quad t < 1, 
\end{cases}
$$
and therefore we obtain
\begin{eqnarray*}
\|t^{-s} \, k(t,u)\|^1_{L^1(1/t)} 
&=& 
\int_{\R^+} |k(t,u) t^{-s}|  \frac{1}{t}  dt 
\\
&\leq& 
\int_{0}^1 \frac{(C_0 + C_1 |\ln(t)|  )}{t^{ s } } dt 
+ \int_{1}^\infty  \frac{\| u \|^1_X}{t^{1 + s}}dt  < + \infty,
\end{eqnarray*}
which proves that $u \in [X,Y]_{s,1}$ as claimed.
\end{proof} 
\vspace*{5mm}

%%-------------------------------------------------------------------------------------------
\begin{corollary} \label{cor_1}
For any $s \in [0,1[$ and $p \in \mathbf{S}^d$,  
\begin{eqnarray*}
|x|^{1+s} K_{\phi} \in L^1(\R^d),
&\quad
\left(  K_{\phi} \right)_{| p^{\perp}} \in  L^1(\R^{d-1}),
\quad &
\left( x \otimes x K_{\phi} \right)_{| p^{\perp}} \in  L^1(\R^{d-1}). 
\end{eqnarray*}
\end{corollary}
\vspace*{5mm}

%%-------------------------------------------------------------------------------------------
\begin{prop}[Decay of averages of $K_{\phi}$ on spheres]~\\ 
\label{prop_mean_decrease}
The integral 
$$I(R) = \int_{\partial B(0,R)} K_{\phi} d \H^{d-1},$$
is stricly positive, and decays rapidly as  
$$  
\frac{R^{d-1} |S^{d-1}|  }{ (4\pi)^{d/2} \Lambda^d} e^{- \frac{R^2}{ 4 \Lambda^2 }}  
\leq I(R) \leq 
\frac{R^{d-1} |S^{d-1}|  }{ (4\pi)^{d/2} \lambda^d} e^{- \frac{ R^2}{4 \lambda^2 }}
$$
where $\lambda$ and $\Lambda$ are bounds for $\phi^o$ as in~(\ref{bounds_phi^o}).
\end{prop}
%%-------------------------------------------------------------------------------------------

\begin{proof}

Since the measure $\mu := \delta_{\partial B(0,R)}$ has finite mass, its Fourier transform
is the continuous and bounded function
\begin{eqnarray*}
\hat{\mu}(\xi) &=& 
 \ds\int_{\R^d} e^{- 2\pi i x \cdot \xi} d\mu
\;=\;  \ds\int_{\partial B(0,R)} e^{- 2 \pi i x \cdot \xi}.
\end{eqnarray*}
As $\mu$ is radially symmetric, $\hat{\mu}$ can be expressed in the form
\begin{eqnarray*} 
\hat{\mu}(\xi) &=& R^{d-1} J(R |\xi|),
\end{eqnarray*}
where $J$ is a function $\R^+ \to \R$.
It follows that    
\begin{eqnarray} \label{expression_IR}
I(R) &=&  \Scal{\delta_{\partial B(0,R)}}{ K_{\phi}} 
=  \Scal{  R^{d-1} J(R|\xi|) }{ e^{-4\pi^2 \phi^o(\xi)^2}} 
\nonumber \\
&=&  
R^{d-1} \int_{S^{d-1}} \int_0^{+\infty} 
r^{d-1} J(R r) e^{-4\pi^2 \phi^o(\theta)^2 r^2} dr d \H^{d-1}.
\end{eqnarray}
We use the particular case when $\phi^o(\xi)$ is isotropic, i.e., 
$\phi^o(\xi) = |\xi|$ to estimate the previous integral. 
In this case, $K_\phi = \frac{1}{(4\pi)^{d/2}} e^{- \frac{x^2}{4}}$ is the heat kernel,
and by a direct calculation we see that the corresponding integral is
$I(R) =  < \delta_{B(0,R)}, K_\phi> =
\frac{R^{d-1} |S^{d-1}|}{ (4\pi)^{d/2}} e^{- \frac{R^2}{4}}$. 
Comparing this expression to~(\ref{expression_IR}) and using 
the radial symmetry of $K_{\phi}$ shows that  
\begin{eqnarray*}
\int_{0}^{+\infty} r^{d-1}  J(R r) e^{-4\pi^2  r^2} dr 
&=& 
\frac{1}{(4\pi)^{d/2}} e^{- \frac{R^2}{4}},
\end{eqnarray*}
or, after a change of variable, that
\begin{eqnarray} \label{intJ}
 \int_{0}^{+\infty} r  J(Rr) e^{-4\pi^2 \phi^o(\theta)^2 r^2} dr = \frac{1}{ (4\pi)^{d/2} \phi^o(\theta)^d } e^{- \frac{R^2}{4 \phi^o(\theta)^2}}.
\end{eqnarray}

Returning to a general kernel $K_\phi$, 
we deduce from~(\ref{expression_IR}) and~(\ref{intJ}) that
\begin{eqnarray*}
I(R) &=& \frac{R^{d-1}}{(4\pi)^{d/2} } \int_{S^{d-1}} 
\frac{1}{\phi^o(\theta)^d }  e^{- \frac{R^2}{4 \phi^o(\theta)^2}} d\H^{d-1},  
\end{eqnarray*}
which  in view of~(\ref{bounds_phi^o}) concludes the proof.
\end{proof}
\vspace*{5mm}

%%-------------------------------------------------------------------------------------------
\begin{prop}[Positivity on hyperplanes]~\\ 
\label{proposition:positivity}
For all $p \in \S^{d}$, the integral  $ \int_{p^{\perp}} K_{\phi}  d\H^{d-1}$ 
is well defined, and satisfies
$$ 
\int_{p^{\perp}} K_{\phi}  d\H^{d-1}  
=  \frac{1}{2 \sqrt{\pi} \phi^o(p)} .
$$
In particular, we have
$$
\frac{1}{2 \sqrt{\pi} \Lambda}  \leq  
\int_{p^{\perp}} K_{\phi}  d\H^{d-1} \leq \frac{1}{2 \sqrt{\pi} \lambda}.
$$
\end{prop}
%%-------------------------------------------------------------------------------------------
\begin{proof}
Let $p \in \S^{d}$. We already know from Corollary~\ref{cor_1} that 
$ \int_{p^{\perp}} K_{\phi}  d\H^{d-1}$ is well defined. 
Consider for $\mu > 0$, the approximating functions $f_{\mu}$, defined by 
$$
\begin{cases}
f_{\mu}(x) = K_{\phi}(x) e^{-\pi|x|^{2}/\mu^2},
\\
\hat{f}_{\mu}(\xi) = e^{-4 \pi^2 \phi^o(\xi)^2}*\frac{1}{\mu^2}e^{-\pi \mu^2|\xi|^2}.
\end{cases}
$$
The function $ f_{\mu}$ belongs to the Schwartz space ${\cal S}(\R^d)$.
Moreover, $\hat{f}_{\mu} \to \hat{K_{\phi}}$ in $W^{d-1,1}(\R^d)$, and
the trace trace  theorem~\cite{Grisvard} shows that one also has
\begin{eqnarray} \label{eqn:lim1}
\lim_{\mu \to \infty} \int_{\R} \hat{f}_\mu(s p) ds &=& \int_{\R} \hat{K_{\phi}}(s p) ds.
\end{eqnarray} 
On the other hand, it follows from the Lebesgue dominated convergence theorem 
and from~(\ref{estim_Kphi}) that 
\begin{eqnarray} \label{eqn:lim2}
\lim_{\mu \to \infty} \int_{p^{\perp}} f_{\mu} d\H^{d-1} &=& \int_{p^{\perp}} K_{\phi} d\H^{d-1}.
\end{eqnarray}
As  $f_{\mu} \in {\cal S}(\R^d)$, we infer that 
\begin{eqnarray*}
\int_{p^{\perp}} f_{\mu} d\H^{d-1} = \Scal{\delta_{p^{\perp}}}{f_{\mu}} 
= \Scal{\delta_{p}}{\F\left[ f_{\mu} \right]} = \int_{\R}  \hat{f}_{\mu}(sp) ds. 
\end{eqnarray*}
so that  (\ref{eqn:lim1}) and (\ref{eqn:lim2}) yield
\begin{eqnarray*}
\int_{p^{\perp}} K_{\phi}d\H^{d-1}  
&=& \int_{\R} \hat{K_{\phi}}(sp)  \, ds  
=  \int_{\R}  e^{-4\pi^2 s^2 \phi^o(p)^2} \, ds 
\\ 
&=&  
\int_{\R} e^{- \pi \left( 2 \sqrt{\pi}  \phi^o(p) s \right)^2 } \, ds  
= \frac{1}{ 2 \sqrt{\pi} \phi^o(p)},
\end{eqnarray*}
which concludes the proof.
\end{proof}
\vspace*{5mm}
%%-------------------------------------------------------------------------------------------

\begin{prop}[Moments of order $2$]~\\ 
\label{prop:moment_order_2}
Let $p \in \S^{d}$. Then $ \frac{1}{2}  \int_{p^{\perp}}x \otimes x K_{\phi}  d\H^{d-1}$ 
is well defined and satisfies
$$  
\frac{1}{2}  \int_{p^{\perp}} x \otimes x K_{\phi}  d\H^{d-1} 
=   \phi^o_{\xi \xi}(p)  \frac{1}{2 \sqrt{\pi}}. 
$$
\end{prop}
%%-------------------------------------------------------------------------------------------

\begin{proof}
Corollary~\ref{cor_1} states that the integral  
$\int_{p^{\perp}} |x|^{2} K_{\phi}d\H^{d-1}$ is well defined. 
Recalling the sequence $f_{\mu}$ used in the previous proposition, we observe  that 
$ D^2 \hat{f}_{\mu} \to D^2 \hat{K_{\phi}}$ in $W^{d-1,1}(\R^d)$, so that the trace theorem 
implies  
\begin{eqnarray} \label{eqn:lim3}
\lim_{\mu \to \infty} \int_{\R} D^2 \hat{f}_\mu(s p) ds  
&=& \int_{\R} D^2 \hat{K_{\phi}}(s p) ds.
\end{eqnarray}
From proposition \ref{prop_decrease}  and the Lebesque dominated convergence, we obtain
\begin{eqnarray} \label{eqn:lim4}
\lim_{\mu \to \infty}  \int_{p^{\perp}} x \otimes x~f_\mu(x)~ d\H^{d-1}  
&\to&  \int_{p^{\perp}} x \otimes x ~K_{\phi}(x)~  d\H^{d-1}. 
\end{eqnarray}
Moreover, we have 
\begin{eqnarray*}
\int_{p^{\perp}} x \otimes x f_{\mu}(x) d\H^{d-1} 
&=&  
\Scal{ \delta_{p^\perp} }{ x \otimes x f_{\mu}}
=  - \frac{1}{4 \pi^2} \Scal{\delta_{p}}{D^2 \hat{f}_{\mu}} \\
&=& - \frac{1}{4 \pi^2} \int_{\R} D^2 \hat{f}_{\mu}(s p) ds,
\end{eqnarray*}
so that in view of~(\ref{eqn:lim3})
$$  
\int_{p^{\perp}} x \otimes x K_{\phi}(x) d\H^{d-1}  
= - \frac{1}{4 \pi^2}  \int_{\R} D^2 \hat{K_{\phi}}(s p) ds.
$$

We next estimate the above right-hand side by a direct calculation:
\begin{eqnarray*}
- \frac{1}{4 \pi^2} \int_{\R} D^2 \hat{K_{\phi}}(sp) \,ds  
&=&  \left[ 
2 \phi^o (p) \phi^o_{\xi \xi} (p) +  2 \phi_{\xi}^o (p) \otimes \phi^o_{\xi} (p)  
\right] 
\int_{\R}  e^{-4\pi^2 s^2 \phi^o(p)^2} \,ds 
\\
&&  
- \left[ 2 \phi_{\xi}^o (p) \otimes \phi^o_{\xi} (p) \right]  
\int_{\R} 8 \pi^2 s^2 \phi^o(p)^2  e^{-4\pi^2 s^2 \phi^o(p)^2} \,ds. 
\end{eqnarray*}
Further, we see by integration by parts that 
\begin{eqnarray*}
\int_{\R} 8 \pi^2 s^2  \phi^o(p)^2  e^{-4\pi^2 s^2 \phi^o(p)^2}  \,ds 
&=&  
\int_{\R} \left\{  4 \pi^2 2 s  \phi^o(p)^2 e^{-4\pi^2 s^2 \phi^o(p)^2} \right\} \{ s \}  \,ds
\\
&=& 
\int_{\R} e^{-4\pi^2 s^2 \phi^o(p)^2} d s 
\;=\; \ds\frac{1}{2 \sqrt{\pi} \phi^o(p)},
\end{eqnarray*}
and we conclude that
$$ 
\frac{1}{2} \int_{p^{\perp}} x \otimes x K_{\phi}(x) d\H^{d-1}  
= \phi^o_{\xi \xi}(p)  \frac{1}{2 \sqrt{\pi}}.
$$
\end{proof}
\vspace*{5mm}

%%-------------------------------------------------------------------------------------------

\begin{corollary}[The operator $F(X,p)$]~\\
Given $X \in \R^{d \times d}$ and $p \in \S^{d}$, let 
\begin{eqnarray} \label{def_F}
F(X,p) &=& \left(  \int_{p^{\perp}} K_{\phi}(x)  d\H^{d-1} \right)^{-1} 
\left( \frac{1}{2} \int_{p^{\perp}} <Xx,x> K_{\phi}(x) d\H^{d-1}  \right).
\end{eqnarray}
This operator is elliptic and satisfies  
\begin{eqnarray} \label{form_F} 
F(X,p) &=& \phi^o(p) \phi^o_{\xi \xi}(p):X.
\end{eqnarray}
\end{corollary}

\begin{proof} 
Equation~(\ref{form_F}) is a direct consequence of  propositions~\ref{proposition:positivity} 
and~\ref{prop:moment_order_2}, while the ellipticity of $F$ follows from the convexity of $\phi^o$.
\end{proof}

%%-------------------------------------------------------------------------------------------
\begin{remark}
In the next section, we introduce an algorithm for motion by anisotropic mean
curvature, and show its consistency with an evolution equation of the
form $u_t = - F(D^2u,\ds\frac{\nabla u}{|\nabla u|})$ where $F$ is defined by~(\ref{def_F}).
The expression~(\ref{form_F}) shows that this operator is precisely 
the one corresponding to motion by anisotropic mean curvature
(see~\cite{bellettini_1996}).
\end{remark}
%%-------------------------------------------------------------------------------------------

\vspace*{5mm}
%%-------------------------------------------------------------------------------------------

\begin{prop}[Positivity of order moment $s$] \label{prop:moment_order_s}
Let $V$ be a subspace of $\R^d$ of dimension $1 \leq m \leq d$, and let $0 < s < 2$.
Then
$$ 
\int_{V} |x|^s K_{\phi} d\H^{m} > 0.
$$
\end{prop}
%%-------------------------------------------------------------------------------------------

\begin{proof} 
We first consider the case $m=d$ and $V = \R^d$. 
we consider the finite part $Pf\left( \frac{1}{|x|^{d+s}} \right)$ as a temperate
distribution, defined for $\varphi \in {\cal S}(\R^d)$ by
$$ 
\Scal{Pf\left( \frac{1}{|x|^{d+s}} \right)}{\varphi} 
= \lim_{\eps \to 0} 
\left\{  
\int_{\R^d\setminus B(0,\eps)} \frac{ \varphi(x) - \varphi(0) }{ |x|^{d+s}} dx 
\right\}. 
$$
This function happens to be the Fourier transform of the distribution $|x|^s$.
More precisely, 
\begin{eqnarray} \label{def_Pf}
\F\left[ |x|^s \right] = C_{s,d} Pf\left( \frac{1}{|2 \pi \xi|^{d+s}} \right),
& \quad \textrm{with}\;&
C_{s,d} = 2^{s + d} \pi^{d/2} \frac{\Gamma((s+d)/2)}{\Gamma(-s/2)},
\end{eqnarray}  
(see for instance~\cite{GelfandShilov}, $\Gamma$ denotes the Gamma function).
We can thus write
\begin{eqnarray}
\ds\int_{\R^d} |x|^s K_\phi \,dx
&=& \Scal{|x|^s}{K_\phi} 
\;=\; 
\Scal{C_{s,d} Pf \left( \frac{1}{|2 \pi \xi|^{d+s}} \right) }{e^{-4\pi^2 \phi^o(\xi)^2}}
\\
&=&
C_{s,d} \, \lim_{\e \to 0}
\ds\int_{\R^d \setminus B(0,\e)}
\ds\frac{e^{-4\pi^2 \phi^o(\xi)^2}-1}{|2 \pi \xi|^{d+s}}
\;>\;0,
\end{eqnarray}
a stricly positive quantity, in view of the sign of $C_{s,d}$.
\medskip

Suppose now that $m < d$ and consider the subspace $V = \textup{Vect}\{e_1, \dots, e_m\}$.
We write $x= (x^\prime, x^{\prime\prime})$, $\xi = (\xi^\prime, \xi^{prime\prime})$,
with $x^\prime, \xi^\prime \in V$. 
A straightforward computation shows that
\begin{eqnarray*}
\ds\int_V |x^\prime|^s K_\phi \, d\H^{m} 
&=&
\Scal{|x^\prime|^s}{K_\phi(x^\prime,0)}_{{\cal D}^\prime(\R^m),{\cal D}(\R^m)}
\\
&=&
\Scal{\H^{d-m}_{\llcorner \{ \xi^{\prime\prime}=0 \}} \otimes |x^\prime|^s}
{K_\phi(x^\prime,x^{\prime\prime})}_{{\cal D}^\prime(\R^d),{\cal D}(\R^d)}
\\
&=&
\Scal{C_{s,m} Pf \left( \ds\frac{1}{|2 \pi \xi^\prime|^{m+s}} \right)}{h(\xi^\prime)}
_{{\cal D}^\prime(\R^m),{\cal D}(\R^m)},
\end{eqnarray*}
where the function  $h : \R^{m} \to \R$ is defined by  
$$ 
h(\xi') = \int_{ \R^{d-m}} e^{-4\pi^2 \phi^o((\xi^{\prime}, \xi^{\prime\prime}))^2} 
d\xi^{\prime\prime}.
$$
The next lemma states that $h$ is $C^1$ and maximal at $\xi^\prime = 0$,
which in view of~(\ref{def_Pf}) and of the sign of $C_{s,m}$ concludes the proof.
\end{proof}
\vspace*{5mm}

%%-------------------------------------------------------------------------------------------

\begin{lemma} The function  $h : \R^{m} \to \R$,  defined by   
$$ 
h(\xi') = \int_{ \R^{d-m}} 
e^{-4\pi^2 \phi^o((\xi', \xi^{\prime\prime})^2} d\xi^{\prime\prime}
$$
is $C^1$, with fast decay as $|\xi^\prime| \to \infty$, and is maximal at $\xi'=0$.
\end{lemma}
%%-------------------------------------------------------------------------------------------
\begin{proof}
recalling~\eqref{bounds_phi^o}, we first remark that 
\[
e^{-4\pi^2 \phi^o(\xi', \xi_{\prime\prime})^2}
\leq e^{-4\pi^2 \lambda^2 |\xi|^2}
\leq e^{-4\pi^2 \lambda^2 |\xi^\prime|^2},
\]
so that the functions $\xi^\prime \to e^{-4\pi^2 \phi^o(\xi', \xi_{\prime\prime})^2}$
and their derivatives are uniformly bounded in $L^1(\R^{d-m})$. 
The $C^1$ regularity of $h$ is thus a consequence
of the Lebesgue theorem. The above estimate also shows that
\begin{eqnarray*}
 |h(\xi')| \leq \int_{\R^{d-m}} e^{-4\pi^2 \lambda^2 (\xi_1^2 + \xi_2^2 + ... + \xi_d^2)} 
d\xi_{m+1} ... d\xi_d \leq  \frac{1}{2 \lambda^m \sqrt{\pi}^{m}} e^{-4 \pi^2 \lambda^2 \xi'^2 }. 
\end{eqnarray*}
\medskip

To determine the maximal value of $h$, we consider 
the sets  $A_{\xi^\prime,t}$, defined for all  $\xi^\prime \in \R^{m}$ and  $t \in ]0,1[$ by 
 \begin{eqnarray*}
  A_{\xi^\prime,t} &=& 
\Set{ \xi^{\prime\prime} \in \R^{d-m}}{ e^{- 4 \pi^2 \phi^o((\xi^\prime,\xi^{\prime\prime}))^2} \geq t}
 \end{eqnarray*}
 
\begin{figure}[htbp]
\centering
\includegraphics[width=4cm]{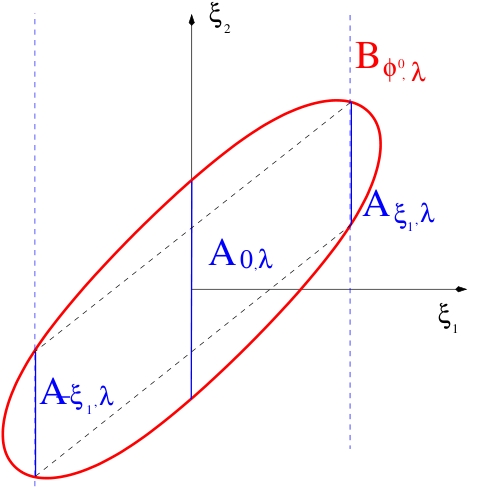}
\caption{ }
\label{fig:anisotropie_2}
\end{figure}

Fix $\xi^\prime_0 \in \R^{m}$. The set  $A_{\xi^\prime_0,t}$ can be defined as 
the intersection of the hyperplane  $\Set{ \xi \in \R^d }{\xi^\prime = \xi_0^\prime}$ 
with the Frank shape  
  $$B_{\phi^o,t} = \Set{\xi \in \R^d}{ \phi^o(\xi) \leq \frac{1}{2\pi}\sqrt{-ln(t)}}.$$
The set $B_{\phi^o,t}$  is convex since $\phi^o$ is convex. 
Moreover, from the symmetry of $\phi^o$, ($\phi^o(\xi) = \phi^o(-\xi)$), we have  
$$ 
|A_{\xi^\prime_0,t}| = | A_{-\xi^\prime_0,t}|.
$$
Next, let
\begin{eqnarray*}
 \tilde{A}_{\xi^\prime_0,t} &=& 
\frac{1}{2} \left( A_{\xi^\prime_0,t} + A_{\xi^\prime_0,t} \right) \\
&=& \Set{\xi^{\prime\prime} \in \R^{d-m}}
{ \exists (\xi_1^{\prime\prime},\xi_2^{\prime\prime}) 
\in  A_{\xi^\prime_0,t} \times A_{-\xi^\prime_0,t}, 
\quad \xi^{\prime\prime} = \frac{1}{2} \left(\xi_1^{\prime\prime} + \xi_2^{\prime\prime} \right) }.
\end{eqnarray*}

We remark that the convexity of $\phi^o$ implies  that 
$\tilde{A}_{\xi^\prime_0,t} \subset A_{0,t}$. 
Indeed, let  $\xi^{\prime\prime} \in \tilde{A}_{\xi^\prime_0,t}$, 
\begin{eqnarray*}
\phi^o\left( (0,\xi^{\prime\prime}) \right) 
&=& \phi^o\left( \frac{1}{2} 
\left( (\xi_0^\prime,\xi_1^{\prime\prime}) + (-\xi^\prime_0,\xi_2^{\prime\prime})  
\right) \right)
\\
&\leq& 
\frac{1}{2} \left( \phi^o\left( (\xi_0^\prime,\xi_1^{\prime\prime})\right) +  
\phi^o\left( (-\xi_0^\prime,\xi_2^{\prime\prime})\right) \right) 
\leq \frac{1}{2\pi} \sqrt{-ln(t)},
\end{eqnarray*}
so that $e^{-4 \pi^2 \phi^o\left( (0,\xi'') \right)^2} \geq t$, 
i.e.  $\xi^{\prime\prime} \in A_{0,t} $. 
Invoking the Brunn-Minkowski inequality, we obtain
\begin{eqnarray}
 |\tilde{A}_{\xi^\prime_0,t}|^{1/(d-m)} &=&  
\frac{1}{2} |A_{\xi^\prime_0,t} +  A_{-\xi^\prime_0,t}|^{1/(d-m)} 
\\
& \geq& 
\frac{1}{2} \left( |A_{\xi^\prime_0,t}|^{1/(d-m)} + |A_{-\xi^\prime_0,t}|^{1/(d-m)}  \right) 
\geq  |A_{\xi^\prime_0,t}|^{1/(d-m)},
\end{eqnarray}
and finally that, 
$$ |A_{0,t} | \geq |\tilde{A}_{\xi^\prime_0,t}| \geq  |A_{\xi^\prime_0,t}|.$$
As this equality holds for any $\xi^\prime_0 \in \R^{m}$, it follows that $h$
is maximal at $\xi^\prime = 0$. 
  \end{proof}
%---------------------------------------------------------------

%%%%%%%%%%%%%%%%%%%%%%%%%%%%%%%%%%%%%%%%%%%%%%%%%%%%%%%%%%%%%%%%%%%%%%%%%%%%%%%%%%%%%%%%%%%%%

\section{The Bence-Merriman-Osher-like algorithm}

Barles and Souganidis~\cite{MR1115933} have studied the convergence of a general
approximation scheme to viscosity solutions of nonlinear second-order parabolic PDE's of the type
\begin{eqnarray} \label{eq:EDP}
    u_t + F(D^2 u, D u) = 0.
\end{eqnarray}
The main assumption on the function $F$ is its ellipticity, i.e., $F$
satisfies
\begin{eqnarray} \label{def_ellipticity}
\forall\; p \in \R^{d}\setminus \{0\},
\forall X, Y \in {\mathbf M}^{d \times d}_s, 
\quad X \leq Y &\Leftarrow& 
F(X,p) \leq F(Y,p).
\end{eqnarray}

Barles and Souganidis study a family of operators $G_h : BUC(\R^d) \to BUC(\R^d) $ for $h >0$, 
which satisfy, for all $ u,v \in BUC(\R^d)$  
\begin{itemize} 
 \item  {\it Continuity  }  
 \begin{eqnarray}\label{continuity}
  \forall\; c \in \R, \quad G_h(u + c) = G_h u, 
 \end{eqnarray}
 \item  {\it Monotonicity  }
  \begin{eqnarray} \label{Monotonicity}
   u \leq v \;\Leftarrow\; G_h u \leq  G_h v + o(h) 
  \end{eqnarray}
  (see remark 2.1 in~\cite{MR1115933})
 \item  {\it Consistency } 
  \begin{eqnarray} \label{Consistent}
  \forall\; \varphi \in {\cal C}^\infty(\R^d), \quad
 \begin{cases}
  \lim_{h \to 0} h^{-1} (G_h(\varphi) - \varphi)(x) & \leq - F_{*}(D^2 \varphi(x), D \varphi(x)) \\
  \lim_{h \to 0} h^{-1} (G_h(\varphi) - \varphi)(x) &\geq - F^{*}(D^2 \varphi(x), D \varphi(x))
 \end{cases}.
  \end{eqnarray}
 \end{itemize}
For all $T > 0$ and for all partitions $P = \{  O = t_0 < ... < t_n = T \}$  of $[0,T]$, 
one can then define a sequence of fonctions $u_P: \R^d \times [0,T] \to \R$ by
\begin{eqnarray} \label{eqn_def_scheme}
u_P(.,t) =
\begin{cases}
 G_{t-t_i}(u_P(.,t_i)) & \text{if} \quad t \in (t_i,t_{i+1}], \\
 g & \text{if} \quad  t = 0,
\end{cases}
\end{eqnarray}
If additionnally the following condition holds,
\begin{itemize}
 \item {\it Stability}
 \begin{eqnarray} \label{Stability}
  \begin{cases}
   \text{there exists } \omega \in C([0,\infty],[0,\infty]), \text{ independent of } P \text{ and depending } \\
  \text{ on } g \text{ only through the modulus of continuity of } g, \\
  \text{such that } \omega(0) = 0 \text{ and for all } t \in [0,t], \\ 
  \| u_P(.,t) - g \|_{L^{\infty}} \leq \omega(t),
  \end{cases}
 \end{eqnarray}
\end{itemize}
then the following theorem holds~\cite{MR1115933} : 
%%%%%%%%%%%%%%%%%%%%%%%%%%%%%%%%%%%%%%%%%%%%%%%%%%%%%%%%%%%%%%%%%%%%%%%
\begin{theorem} \label{thm_BarlesSouganidis}
Assume that $G_h: BUC(\R^d) \to BUC(\R^d)$ satisfies \eqref{continuity}, \eqref{Monotonicity}, \eqref{Consistent}, and \eqref{Stability} for all $T>0$, $g \in BUC(\R^d)$ and all partitions $P$ of $[0,T]$. Then, $u_P$ defined in~(\ref{eqn_def_scheme}) converges uniformly in $\R \times[0,T]$
to the viscosity solution of \eqref{eq:EDP}.
\end{theorem}
%%%%%%%%%%%%%%%%%%%%%%%%%%%%%%%%%%%%%%%%%%%%%%%%%%%%%%%%%%%%%%%%%%%%%%%

This result was used by H. Ishii, G. Pires and P.E. Souganidis in~\cite{Ishhi_pires_souganidis} 
to study anisotropic mean curvature flow.
These authors introduce a kernel $f$, which satisfies:  
 \begin{itemize}
 \item[($H_1$)]   $f(x) \geq 0, \quad f(-x) = f(x) \quad \text{for all } x \in \R^d, \quad \text{and} \quad \int_{\R^d} f(x) dx = 1$
 \item[($H_2$)]  $\int_{p^{\perp}} (1 + |x|^2 ) |f(x)| d \H^{d-1} < \infty \quad \text{for all} \quad p \in S^{d}$
 \item[($H_3$)]  $\begin{cases}
                 \text{the functions } p \to  \int_{p^{\perp}} f(x) d\H^{d-1} \quad p  \to \int_{p^{\perp}} x_i x_j f(x) d\H^{d-1},\\
                  1 \leq i,j \leq d, \quad \text{are continuous on } {\mathbf S}^{d}
                 \end{cases}
$
\item[($H_4$)]  $  \int_{\R^d} |x|^{2} |f(x)| dx < \infty  $
\item[($H_5$)] For all collections $\set{R(\rho)}_{0<\rho<1} \subset \R$ such that 
 $R(\rho) \to \infty \quad \text{and} \quad \rho R(\rho)^{2} \to 0 \quad \text{as} \quad  \rho \to 0$, and for all 
 functions $g: \R^{d-1} \to \R$ of the form
$ g(\xi) = a + \Scal{A \xi}{\xi}$ with  $a \in \R$ and $A \in \S^{d-1}$, 
$$ \lim_{\rho \to 0 } \sup_{U \in {\mathbf O}(d)} \sup_{0<r<\rho} \left| \int_{B(0,R(\rho))} f_{U}(\xi,rg(\xi)) g(\xi) d\xi -  \int_{\R^{d-1}}f_{U}(\xi,0) g(\xi) d\xi   \right| = 0,$$ where ${\mathbf O}(n)$ denotes the group of $d \times d$ orthogonal matrices, and where $f_U: \R^d \to \R$ is defined for all $U \in {\mathbf O}(d) $ by $ f_{U}(x) = f(U^{*}x) $. 
\end{itemize}
Theorem~\ref{thm_BarlesSouganidis} has been applied to schemes for anisotropic mean curvature motion
(see theorem 3.3 in~\cite{Ishhi_pires_souganidis})
with $G_h$  defined by 
 \begin{eqnarray}
  G_h \Psi(x) &=& \sup\Set{\lambda \in \R}{S_h \1_{\Psi \geq \lambda}(x) \geq \theta_h}  \\
              &=&   \inf\Set{\lambda \in \R}{S_h \1_{\Psi \geq \lambda}(x) <\theta_h} 
 \end{eqnarray}
where
$$ S_h g(x) = h^{-d/2} f(./\sqrt{h}) * g (x)
= h^{-d/2} \ds\int_{\R^d} f(y/\sqrt{h})  g(x-y)\,dy, 
\quad  \theta_h = \frac{1}{2} + c \sqrt{h},$$
and where $F(X,p)$ is given by
$$ F(X,p) = -  \left( \int_{p^{\perp}} f(x) d \H^{d-1}(x)  \right)^{-1} \left(\frac{1}{2} \int_{p^{\perp}} \Scal{Xx}{x} f(x) d \H^{d-1}(x) + c|p|\right),$$ 
(the last term in this integral models a forcing term).

In this section, we follow the proof in \cite{Ishhi_pires_souganidis} to show a consistency
result in our case when f is a non positive kernel and does not have moments of
order two ( ie.  $x \to |x|^2 f(x) \notin L^1(\R^d)$). 
We introduce two operators  $G_h^{+}$ and $G_h^{-}$ defined by 
 \begin{eqnarray}
  G_h^{+} \Psi(x) &=& \sup\Set{\lambda \in \R}{S_h \1_{\Psi \geq \lambda}(x) \geq \theta_h}  \\
  G_h^{-} \Psi(x) &=&   \inf\Set{\lambda \in \R}{S_h \1_{\Psi \geq \lambda}(x) <\theta_h} 
 \end{eqnarray}
which are not necessarly equal  as our kernel is not being  nonnegative. 
%%----------------------------------------------------------------------------§
\subsection{A consistency result in the case where $f = K_{\phi}$}
To adapt these results to our context we modify the assumptions $(H_1)$, $(H_4)$ and $(H_5)$
as follows
\begin{itemize}  
\item[$(H^\prime_1)$]
$%%\label{hyp:positivity}
\int_{p^{\perp}} f(x) d \H^{d-1} > 0 \;\text{for all}\; p \in S^{d}, 
\quad f(-x) = f(x) 
\quad \text{and}\; \int_{\R^d} f(x) dx = 1,
$
\item[$(H^\prime_4)$]
$%\label{hyp:decrease2}
\int_{\R^d} |x|^{2-\mu} |f(x)| dx < \infty \quad \text{for}\; 0 < \mu < 2, 
$
\item[$(H^\prime_5)$]
Assume that $\mu \in ]0,1/2]$. Then  for all collections $\set{R(\rho)}_{0<\rho<1} \subset \R$ 
such that $R(\rho) \to \infty$ and $\rho R(\rho)^{2-\mu} \to 0$ as $\rho \to 0$, and for all functions $g: \R^{d-1} \to \R$ of the form $g(\xi) = a + \Scal{A \xi}{\xi}$ with $a \in \R$ and $A \in \S^{d-1}$, 

\begin{eqnarray*} 
\label{hyp:limite_rho}
\lim_{\rho \to 0}  \sup_{U \in O(d)} \sup_{0<r<\rho} 
\left| 
\int_{B(0,R(\rho))} f_{U}(\xi,rg(\xi)) g(\xi) d\xi 
-  \int_{\R^{d-1}}f_{U}(\xi,0) g(\xi) d\xi   
\right| = 0,
\\
\label{hyp:limite_rho_abs}
\lim_{\rho \to 0}  \sup_{U \in O(d)} \sup_{0<r<\rho} 
\left|
\int_{B(0,R(\rho))} \left|f_{U}(\xi,rg(\xi))\right|  g(\xi)  d\xi 
-  \int_{\R^{d-1}} \left| f_{U}(\xi,0) \right| g(\xi)  d\xi   
\right|
= 0,
\end{eqnarray*}
\end{itemize}

In this last statement, $B(0,R(\rho))$ denotes the $(n-1)$-dimensional ball, centered at $0$
and of radius $R(\rho)$. 

%%----------------------------------------------------------------------------§

\subsubsection{ $K_{\phi}$ satisfies $(H_2, H_3)$ and $(H^\prime_1, H^\prime_4, H^\prime_5)$}
We remark that  $\hat{K}_{\phi}(\xi) = \hat{K}_{\phi}(-\xi)$ and $\F(K_{\phi})(0)=1$, 
so that
\begin{eqnarray*}
 K_{\phi}(-x) = K_{\phi}(x) \quad \text{for all } x \in \R^d, \quad \text{and} \quad \int_{\R^d} K_{\phi}(x) dx = 1.  
\end{eqnarray*}
Moreover, proposition~\eqref{proposition:positivity} shows that 
\begin{eqnarray*} 
 \int_{p^{\perp}} K_\phi(x) d \H^{d-1} \geq  \frac{1}{ (4 \pi)^{d/2} \Lambda^d } > 0    
 \quad \text{for all} \quad p \in S^{d},
\end{eqnarray*}
so that $(H^\prime_1)$ is satisfied.
Propositions  \eqref{proposition:positivity} and \eqref{prop:moment_order_2} also imply
that $K_\phi$ satisfies $(H_2)$, i.e., 
\begin{eqnarray}
\int_{p^{\perp}} (1 + |x|^2 ) |K_\phi(x)| d \H^{d-1} 
&<& \infty \quad \text{for all} \quad p \in S^{d}.
\end{eqnarray}
Concerning $(H_3)$, we note that
\begin{eqnarray*}
\frac{1}{2} \int_{p^{\perp}} x \otimes x K_{\phi}(x) d\H^{d-1}  
&=& 
\ds\frac{1}{2\sqrt{\pi}} \phi^o_{\xi\xi}(p),
\end{eqnarray*}
and that
\begin{eqnarray*}
\int_{p^{\perp}} K_{\phi}  d\H^{d-1}  
&=& \frac{1}{ 2 \sqrt{\pi} \phi^o(p) }.
\end{eqnarray*}
Since $\phi^o$ is smooth on $\R^d \setminus \{0\}$ and positive 
(in particular $\phi^o \geq \lambda$ on $\mathbf{S}^{d}$ ) we see that the functions 
\begin{eqnarray*} 
p \to  \int_{p^{\perp}} K_{\phi}(x) d\H^{d-1} 
&\quad& 
p  \to \int_{p^{\perp}} x_i x_j K_{\phi}(x) d\H^{d-1},
\quad 1 \leq i,j \leq d,
\end{eqnarray*}
are continuous on $\mathbf{S}^{d}$.

We next prove  that if $0 < \mu < 2$, then
\begin{eqnarray*}
 \int_{\R^d} |x|^{2-\mu} |f(x)| dx < \infty. 
\end{eqnarray*}
Indeed, proposition~\ref{prop_decrease} with $s = 1- \mu/2$ shows that 
\begin{eqnarray*}
\int_{\R^d} |x|^{2-\mu} |f(x)| \,dx  
&\leq&   \int_{\R^d}  \frac{C_{\phi^o,s} |x|^{2-\mu} }{1 + |x|^{d+1+ (1- \mu/2})} dx 
\;\leq\;   \int_{\R^d}  \frac{C}{1 + |x|^{d + \mu/2}} \,dx
\\
&\leq& 
C|\mathbf{S}^d| \int_{0}^{\infty} \frac{1}{(1 + r^{1 + \mu/2})} dr \;<\; \infty,
\end{eqnarray*}
for some generic constant $C$.
It remains to prove $(H^\prime_5)$:
Let $0 < \mu < 1/2$ and let $R~:~\R^+ \longrightarrow \R^+$ such that, as $\rho \to 0$,
$R(\rho) \to \infty$ and $\rho R(\rho)^{2-\mu} \to 0$.
Setting $f_U(x) = K_\phi(U^*x)$, we consider

\begin{eqnarray}\label{int_fU}
\int_{B(0,R(\rho))} f_{U}(\tilde{x},rg(\tilde{x})) g(\tilde{x}) d\tilde{x} =   \int_{B\left(0,R(\rho)^{\frac{2-\mu}{2}}\right)} f_{U}(\tilde{x},rg(\tilde{x})) g(\tilde{x}) d\tilde{x} 
\\
+ \int_{ B(0,R(\rho)) \setminus B\left(0,R(\rho)^{\frac{2-\mu}{2}}\right)} f_{U}(\tilde{x},rg(\tilde{x})) g(\tilde{x}) d\xi. 
\end{eqnarray}

Let $h_{\rho}, h  : \R^{d-1} \to \R$ denote the functions
$$ 
\begin{cases}
h_{\rho,r}(\tilde{x}) = f_{U}(\tilde{x},r g(\tilde{x})) g(\tilde{x}) 
\chi_{B(0,R(\rho))^{\frac{2-\mu}{2}}} 
\\
h(\tilde{x}) = f_{U}(\tilde{x},0) g(\tilde{x}). 
\end{cases}
$$
When $r < \rho$, $h_{\rho,r}(\tilde{x})$ converge to $h(\tilde{x}) $ pointwise as $\rho \to 0$, 
and 
$$ |h_{\rho,r}(\tilde{x})| \leq \frac{C}{ 1 + |x|^{d-1+s} } \in L^1(\R^{d-1}),$$
for some constant $C$ independent of $\rho$, $r$ and $U$.  
Invoking the Lebesque dominated convergence theorem, we conclude that 
$$ 
\lim_{\rho \to 0, \; r < \rho  }\int_{\R^{d-1}} h_{\rho,r}(\tilde{x}) \,d\tilde{x} 
\to  
\int_{\R^{d-1}} h(\tilde{x})  \,d\tilde{x},
$$
uniformly with respect to $U$  and $r$.  
The second term in~(\ref{int_fU}) converges to 0 uniformly with respect to 
$U$ and $r$ as $\rho \to 0$, since 
\begin{eqnarray*}
\lefteqn{
\int_{ B(0,R(\rho)) \setminus B\left(0,R(\rho)^{\frac{2-\mu}{2}}\right)} 
|f_{U}(\tilde{x},rg(\tilde{x})) g(\tilde{x})| \,d\tilde{x}}
\\ 
&\leq& 
C\, \int_{ B(0,R(\rho)) \setminus B\left(0,R(\rho)^{\frac{2-\mu}{2}}\right)} 
\frac{1}{ 1 + |\tilde{x}|^{d-1 + s} } d\tilde{x}   
%%\\
%%&\leq&
\;\leq\;
C\, |\mathbf{S}^{d-1}| \,\int_{R(\rho)^{\frac{2-\mu}{2}}}^{R(\rho)}  
\frac{1}{ 1 + |r|^{1+ s} } dr 
\\
&\leq&
C\, |\mathbf{S}^{d-1}| \left( R(\rho)^{\frac{-(2-\mu)s}{2}} - R(\rho)^{-s} \right),  \\
\end{eqnarray*}
for some generic constant $C$. 
We conclude that 
\begin{eqnarray*}
\lim_{\rho \to 0} \sup_{U \in O(d)} \sup_{0<r<\rho} 
\left| 
\int_{B(0,R(\rho))} f_{U}(\tilde{x},rg(\tilde{x})) g(\tilde{x}) \,d\tilde{x} 
- \int_{\R^{d-1}}f_{U}(\tilde{x},0) g(\tilde{x}) \,d\tilde{x}   
\right| 
&=& 0.
\end{eqnarray*}
The second statement in $(H_5^\prime)$ is established similarly.
%%-------------------------------------------------------------------------------------

\subsubsection{The consistency proof}

%%-------------------------------------------------------------------------------------
\begin{prop}{\label{lemme_3}} 
Let  $\varphi \in C^2(\R^d)$. For all $z \in \R^d$ and $\epsilon > 0$,  there exists $\delta > 0$ such that for all  $x \in B(z,\delta)$ and $h \in (0,\delta]$, 
if $\nabla\phi(x) \neq 0$ we have
\begin{eqnarray*} 
G_h^{-} \varphi(x) &\leq& \varphi(x) + (-F(D^2 \varphi(z) , D\varphi(z)) + \epsilon)h 
\\
\text{and}\; G_h^{+} \varphi(x) &\geq& \varphi(x) + (-F( D^2\varphi,D\varphi(z)) - \epsilon)h.
\end{eqnarray*}
\end{prop}
%%-------------------------------------------------------------------------------------

\begin{proof}
We closely follow the argument in~\cite{Ishhi_pires_souganidis}.

1. We only prove the first inequality. The other one is obtained similarly.
%%---------------------------------------------------

2. Without loss of generality, we can assume that $z=0$. Let us fix $a \in \R$, such that
$$ 
a> -F(D^2\varphi(0),D \varphi(0)).
$$
The inequality is proved if we can exhibit a $\delta>0$ 
such that, for all $x \in B(0,\delta)$ and $h \in (0,\delta]$,
$$ 
S_h \1_{\varphi\geq \varphi(x) + ah}(x)  < \theta_h. 
$$ 
%%---------------------------------------------------
3. Fix $\delta_1 > 0$, such that $D \varphi \neq 0$ on $B(0,\delta_1)$ 
and choose a continuous family  
$\set{U(x)}_{x \in B(0,\delta_1)} \subset O(d)$, 
such that for all $x \in B(0,\delta_1)$,
$$ 
U(x) \left( \ds\frac{D \varphi(x)}{|D \varphi(x)|} \right) = e_d,
$$
where $e_d$ denotes the unit vector with components $(0,0,...,0,1) \in \R^d$. 
Note that if $x \in B(0,\delta_1)$, then
$$
S_h \1_{\varphi \geq \varphi(x) + a h} 
= 
\int_{\R^d} f_{U(x)}(y) \1_{\varphi \geq \varphi(x) + a h}(x - \sqrt{h} U(x)^{*}y) dy.
$$
%%---------------------------------------------------
4. Choosing $\delta$ smaller if necessary, $(H^\prime_1)$ implies the inequality
$$  
a> -F(D^2\varphi,D \varphi) \quad \text{in} \quad B(0,\delta_1),
$$
or in other words, 
\begin{eqnarray} \label{eqn:reads}
\frac{1}{2} \int_{\R^{d-1}} \Scal{P^* U(x) D^2 \varphi(x) U(x)^{*} P \xi}{\xi} f_{U(x)}(\xi,0) d\xi 
&-& a \int_{\R^{d-1}} f_{U(x)}(\xi,0) d\xi \nonumber \\
 &<&  -c |D\varphi(x)|,
\end{eqnarray}
where $P$ denotes the $d \times (d-1)$ matrix with components $P_{ij} = \delta_{ij}$. 

%%---------------------------------------------------
5. We next fix $\epsilon>0$, and $\delta_2 \in (0,\delta_1[$, 
such that for all $x \in B(0,\delta_2)$,
\begin{eqnarray}
\nonumber
\frac{1}{2}  \int_{\R^{d-1}} \Scal{P^* U(0)(D^2 \varphi(0) + 3\epsilon^2 I ) U(0)^{*} P \xi}{\xi} f_{U(x)}(\xi,0) d\xi  \\
 - (a - \epsilon^2) \int_{\R^{d-1}} f_{U(x)}(\xi,0) d\xi < \; -(\xi +  \epsilon) |D\varphi(0)|. 
\end{eqnarray}
%%---------------------------------------------------
6. The Taylor theorem yields a $\gamma>0$ such that 
for all $h>0$, $y \in \R^d$, and $x \in B(0,\delta_2)$, 
if $\sqrt{h} |y| \leq \gamma$, then 
\begin{eqnarray*}
\varphi(x - \sqrt{h}U(x)^* y) 
&\leq& 
\varphi(x) - \sqrt{h} \Scal{D \varphi(x)}{U(x)^* y} \\
&& \quad \quad  + \frac{h}{2} 
\Scal{U(x) (D^2 \varphi(x) + \epsilon^2 I) U(x)^* y}{y} \\
&\leq&  
\varphi(x) - \sqrt{h}|D\varphi(x)| y_d + Ch y_d^2 
\\
&&  \quad \quad + \frac{h}{2}\Scal{P^* U(x) (D^2 \varphi(x) + 2\epsilon^2 I) U(x)^* P y'}{y'},
\end{eqnarray*}
and 
\begin{eqnarray*}
\varphi(x - \sqrt{h}U(x)^* y) 
&\geq& 
\varphi(x) - \sqrt{h} \Scal{D \varphi(x)}{U(x)^* y} \\
&& \quad \quad   + \frac{h}{2} 
\Scal{U(x) (D^2 \varphi(x) - \epsilon^2 I) U(x)^* y}{y} 
\\
&\geq& 
\varphi(x) - \sqrt{h}|D\varphi(x)| y_d - Ch y_d^2 
\\
&& \quad \quad
+ \frac{h}{2}\Scal{ P^* U(x) (D^2 \varphi(x) - 2\epsilon^2 I) U(x)^* P y'}{y'},
\end{eqnarray*}
where we write $y = (y',y_d) \in \R^{d-1} \times \R$, and where $C$ is a positive constant. 

%%---------------------------------------------------
7. Reducing $\gamma$ and $\delta_2$ if necessary, the previous inequalities imply 
that for $y \in B(0,\gamma/\sqrt{h})$ and $x \in B(0,\delta_2)$, 
\begin{itemize}
\item  
if  $ \varphi(x - \sqrt{h} U(x)^{*} y) \geq \varphi(x) + ah$,  then
\begin{eqnarray*}
y_d &\leq& 
\frac{\sqrt{h}}{ |D \varphi(x)| - C \sqrt{h} y_d} 
\left( -a + \frac{1}{2} \Scal{P^*U(x) (D^2 \varphi(x) + 2\epsilon^2 I) U(x)^* P y'}{y'} \right) 
\\
&\leq& 
\frac{\sqrt{h}}{ |D \varphi(0)|} 
\left( -a + \epsilon^2 + \frac{1}{2} 
\Scal{P^* U(0) (D^2 \varphi(0) + 3\epsilon^2 I) U(0)^*  P y'}{y'} \right)
\end{eqnarray*}
\item 
if 
$$ 
y_d \leq \frac{\sqrt{h}}{ |D \varphi(0)|} 
\left( 
-a - \epsilon^2 + \frac{1}{2} \Scal{P^* U(0) (D^2 \varphi(0) - 3\epsilon^2 I) U(0)^* P y'}{y'}    \right),
$$
then
$$ 
\varphi(x - \sqrt{h} U(x)^{*} y) \geq \varphi(x) + ah.
$$
\end{itemize}
We define
$$ \begin{cases}
  a^{\epsilon} = (a-\epsilon^2)|D\varphi(0)|^{-1} \\
  a_{\epsilon} = (a+\epsilon^2) |D\varphi(0) |^{-1} \\
A^{\epsilon} = |D\varphi(0)|^{-1} P^{*} U(0) \left( D^2 \varphi(0) + 3\epsilon^2 I \right) U(0)^{*}P   \\
A_{\epsilon} = |D\varphi(0)|^{-1} P^{*} U(0) \left( D^2 \varphi(0) - 3\epsilon^2 I \right) U(0)^{*}P,
   \end{cases}
$$
and for $y^\prime \in \R^{d-1}$  
\begin{eqnarray*}
g^{\epsilon}(y^\prime) = \left( - a^{\epsilon} + \frac{1}{2} \Scal{A^{\epsilon} y^\prime}{y^\prime} \right) 
&\quad&
g_{\epsilon}(y^\prime) = \left( - a_{\epsilon} + \frac{1}{2} \Scal{A_{\epsilon} y^\prime}{y^\prime}\right).
\end{eqnarray*}
We also set 
$$ 
V_{h,x} = \Set{y \in R^d}{\varphi(x - \sqrt{h} U(x)^{*} y) \geq \varphi(x) + ah }, 
$$
and
$$
\begin{cases}
E^{+}_{\epsilon,h,x} = 
\Set{y \in \R^d}{ y_d \leq \sqrt{h} g_\epsilon(y^\prime)} 
\\
E^{-}_{\epsilon,h,x} = 
\Set{y \in \R^d}{ y_d \leq \sqrt{h} g^\epsilon(y^\prime)}.
\end{cases}
$$
We check that for all $x \in B(0, \delta_2)$, 
$$ 
\begin{cases}
   \left(   V_{h,x}\cap B(0,\gamma/\sqrt{h}) \right) \quad  \subset \quad \left(  E^{+}_{\epsilon,h,x}\cap B(0,\gamma/\sqrt{h}) \right)  \\
    \left(  E^{-}_{\epsilon,h,x}\cap B(0,\gamma/\sqrt{h})  \right) \quad \subset \quad   \left(  V_{h,x}\cap B(0,\gamma/\sqrt{h}) \right) 
\end{cases}
$$
%%---------------------------------------------------
8. The assumption $(H_4)$ yields the existence  of a decreasing function  $\omega \in C([0,\infty),[0,\infty)) $ such that
$\omega(R) \to 0$ as $R \to \infty $, and

$$ \int_{B(0,R)^{c}} |f(y)| |y|^{2-\mu} dy \leq \omega(R)^2, \text{ for all } R \geq 0.$$
For each $0<t<1$, we define the family of sets $R(t) \in (0,\infty)$ by
\begin{eqnarray}
 \omega(R(t)) &=& t R(t)^{2-\mu},
\end{eqnarray}
which satify ($H^\prime_5$). We then choose $\tau \in (0,1)$ such that
\begin{eqnarray}
 R(t) \leq \gamma /t, \text{ for all } t \in (0,\tau]
\end{eqnarray}
%%---------------------------------------------------
9. Let  
$$ 
\rho = \sqrt{h}, \quad T(\rho) = B_{n-1}(0,R(\rho)) \times \R \subset \R^d.
$$
For all $h \in ]0,\tau^2)$ and for all $x \in B(0,\delta_2)$, we estimate
\begin{eqnarray*}
\int_{V_{h,x}} f_{U(x)}(y) dy  
&=&
\int_{\R^d} f_{U(x)}(y) \1_{\varphi \geq \varphi(x) + a h}(x - \sqrt{h}U^*(x)y) dy 
\\
&\leq&  
\int_{V_{h,x}\cap B(0,R(\rho))} f_{U(x)}(y) dy +  \int_{B(0,R(\rho))^c} |f_{U(x)}(y)| dy 
\\
&\leq& 
\int_{E^{+}_{\epsilon,h,x}\cap B(0,R(\rho))} f_{U(x)}(x) dx
+  \int_{B(0,R(\rho))^c} |f_{U(x)}(y)| dy  
\\
&& \quad \quad \quad +  \int_{\left(E^{+}_{\epsilon,h,x} \setminus E^{-}_{\epsilon,h,x}  \right)\cap B(0,R(\rho))} 
|f_{U(x)}(y)| dy 
\\
&\leq&
\int_{E^{+}_{\epsilon,h,x}\cap T(\rho)} f_{U(x)}(y) dy 
+  \int_{\left(E^{+}_{\epsilon,h,x} \setminus E^{-}_{\epsilon,h,x}\right)\cap T(\rho)} 
|f_{U(x)}(y)| dy 
\\
&& 
\quad \quad \quad +  3\int_{B(0,R(\rho))^c} |f_{U(x)}(y)| dy 
\\
\end{eqnarray*}
%%---------------------------------------------------
10. For the last integral above, we have 
\begin{eqnarray*}
\int_{B(0,R(\rho))^c} |f_{U(x)}|(y) dy 
&\leq&  
\frac{1}{R(\rho)^{2-\mu}} \int_{B(0,R(\rho))^c} |y|^{2-\mu} |f_{U(x)}|(y) dy   
\;\leq\; \omega(R(\rho)) \rho,
\end{eqnarray*}
and moreover, since $K_\phi$ is symmetric,
\begin{eqnarray*}
\frac{1}{2} &=& 
\int_{y_d \leq 0}f_{U(x)}(y) dy 
\;\leq\;
\int_{T(\rho) \cap \set{y_d \leq 0}}f_{U(x)}|(y) dy + \omega(R(\rho)) \rho.
\end{eqnarray*}
We note that
\begin{eqnarray*}
&&\int_{T(\rho) \cap E^{+}_{\epsilon,h,x}} f_{U(x)}(y) dy 
=
\int_{T(\rho) \cap \set{y_d \leq \rho g^{\epsilon}(y') } } f_{U(x)}(y) \,dy  
\\
&& \quad \quad \quad \quad  =   
\int_{T(\rho) \cap \set{y_d \leq 0}}f_{U(x)}(y) dy 
+ \int_{B_{n-1}(0,R(\rho))} d\xi \int_0^{\rho g^{\epsilon}(y')} f_{U(x)}(\xi,r) \,dr 
\\
&& \quad \quad \quad \quad  =  
\int_{T(\rho) \cap \set{y_d \leq 0}}f_{U(x)}(y) \,dy  
+ \int_0^{\rho} dr \int_{B_{n-1}(0,R(\rho))}  f_{U(x)}(\xi,rg(\xi)) g^{\epsilon}(\xi) \,d\xi.  
\end{eqnarray*}
It follows from  $(H^\prime_5)$ that as $\rho \to 0$,
\begin{eqnarray*}
\frac{1}{\rho} \left\{  \int_{T(\rho) \cap   E^{+}_{\epsilon,h,x}} f_{U(x)}(y) dy   
-  \int_{T(\rho) \cap \set{y_d \leq 0}}f_{U(x)}(y) dy \right\}  
&\to&  
\int_{\R^{d-1}}f_{U(x)}(\xi,0) g^{\epsilon}(\xi) d\xi,
\end{eqnarray*}
uniformly with respect to $x$.
Possibly reducing $\tau$  we may assume that for $x \in B(0, \delta_2)$,
\begin{eqnarray*}
\frac{1}{\rho} \left\{  \int_{T(\rho) \cap   E^{+}_{\epsilon,h,x}} f_{U(x)}(y) dy   
-  \int_{T(\rho) \cap \set{y_d \leq 0}}f_{U(x)}(y) dy \right\} 
&\leq& 
\int_{\R^{d-1}}f_{U(x)}(\xi,0) g^{\epsilon}(\xi) d\xi + \epsilon^2. 
\end{eqnarray*}
%%%%%%%%%%%%%%%%stopped here%%%%%%%%%%%%%%%%%%

Using same argument, we also conclude that
 \begin{eqnarray*}
    \int_{T(\rho) \cap \left(E^{+}_{\epsilon,h,x} \setminus E^{-}_{\epsilon,h,x}  \right)}|f_{U(x)}(y)|dy &=& 
   \left\{  \int_{T(\rho) \cap \set{0 \leq y_d \leq \rho g^{\epsilon}(y') } }|f_{U(x)}(y)|dy \right\}  \\ 
   && \quad \quad - \left\{  \int_{T(\rho) \cap \set{0 \leq y_d \leq \rho g_{\epsilon}(y') } }|f_{U(x)}(y)|dy \right\} \\
   &\leq&  \rho \int_{\R^{d-1}} |f_{U(x)}|(\xi,0) (g^{\epsilon}(\xi) - g_{\epsilon}(\xi)) d\xi + \rho\epsilon^2 \\
   &\leq& \rho\epsilon^2 \left( 1 +  \int_{\R^{d-1}} \left( 2 + 3 |\xi|^2 \right) |f_{U(x)}|(\xi,0) d\xi    \right) \\
   &\leq&  C_0 \rho\epsilon^2, 
 \end{eqnarray*}
 where
 $$C_0 =  \sup_{x \in B(0,\delta_2)}\left\{ 1 +  \int_{\R^{d-1}} \left( 2 + 3 |\xi|^2 \right) |f_{U(x)}|(\xi,0) d\xi    \right\}.$$
%%---------------------------------------------------
11. Finally, noting that from \eqref{eqn:reads},
 $$ \int_{\R^{d-1}} f_{U(x)}(\xi,0) g^{\epsilon}(\xi) d\xi \leq -c -\epsilon,$$
 we get
 \begin{eqnarray*}
  \int_{\R^d} f(x) \1_{\varphi \geq \varphi(x) + a h}(x - \sqrt{h}z) dz  &\leq&  \frac{1}{2} + \int_{\R^{d-1}} f_{U(x)}(\xi,0) g(\xi) d\xi \\
&& \quad \quad \quad \quad + \rho\left(\epsilon^2 + 4\omega(R(\rho)) + C_0 \epsilon^2 \right) \\
  &\leq&  \frac{1}{2} + \rho \left( -c -\epsilon + \epsilon^2 +  4\omega(R(\rho))   + C_0 \epsilon^2  \right) \\
 &<&  \theta_h, 
\end{eqnarray*}
for $\eps$ sufficiently small.
 \end{proof}

%%-------------------------------------------------------------------------------------------
Even if the function $\phi$ is regular, $G_h^+ \varphi$ and $G_h^- \varphi$ need not be equal and continuous.
However, it is easy to check that if $\varphi = \1_{\Omega}$ is a characteristic function then
 $G_h^+ \1_{\Omega} = G_h^- \1_{\Omega}$. The next proposition shows that if $\varphi$ is smooth,  $G_h^- \varphi(x) = G_h^+(x) \varphi + o(h)$, so that one could conceivably build a Bence Merriman
Osher type scheme using either $G_h^+$ or $G_h^-$.

\begin{prop}{\label{lemme_4}} 
Let  $\varphi \in C^2(\R^d)$. Let $x \in \R^d$ such as  $\nabla \varphi(x) \neq 0$, then 
$$ G_h^{-} \varphi(x) =   G_h^{+} \varphi(x) + o(h).$$
\end{prop}

\begin{proof}[proof]
Let $x \in \R^d$  such as  $\nabla\varphi(x) \neq 0$ and  for all $h>0$ let  
$$ \epsilon(h) =   G_h^{+} \varphi(x) - G_h^{-} \varphi(x).$$
Introduce  also  $g_{h}(\lambda):\R \to \R$ defined by 
$$ 
g_{h}(\lambda) = S_h \chi_{\varphi \geq \lambda}(x) = \int_{\R^d} K_{\phi,h}(y) \chi_{\varphi \geq \lambda}(x-y) dy.
$$
This function may not be continuous. We claim that  its jumps are bounded by $o(\sqrt{h})$.
Indeed, for all $\lambda \in \R$,  one can express  $g_{h}(\lambda)$  as 
\begin{eqnarray*}
  g_{h}(\lambda) &=& \int_{ B(0,\sigma)} K_{\phi,h}(y) \chi_{\set{\varphi \geq \lambda}}(x-y) dy +  \int_{ \R^d \setminus B(0,\sigma)} K_{\phi,h}(y) \chi_{\set{\varphi \geq \lambda}}(x-y) dy \\
&=&  \tilde{g}_{h}(\lambda) + R_h(\lambda),\\ 
\end{eqnarray*}
where $\sigma$ is chosen sufficiently small so that  $|\nabla \varphi(y)|>0$  for all $y \in  B(x,\sigma)$. 
%%%%%%%%%%%%%%%%%%%%%%%%%%%%%%%%%%%%%%%%%%%%%%%%%ù
%%%%%%%%%%%%%%%%%%%%%%%%%%%%%%%%%%%%%%%%%%%%%%%%%
Let $0 < \mu < 1$, let
$$
\omega(R) =  \int_{ B(0,R)^c } |y|^{2-\mu} |K_{\phi}(y)| dy, 
$$
and let $R(t)$ be defined by the equality  $\omega(R(t)) = t R(t)^{2-\mu}$. 
Note that  $(H^\prime_4)$ implies that $\sqrt{h} R(h)^{2-\mu} \to 0$ as $h \to 0$, so that 
$\sqrt{h}R(h)^{1-\nu/2} < \sigma$ for $h$ sufficiently small, it follows that
\begin{eqnarray*}
 |R_h(\lambda)|  &\leq&   \int_{\R^d \setminus B(0, \sqrt{h} R(h)^{1-\mu/2})} |K_{\phi,h}(y)| dy.
\end{eqnarray*}
Moreover, changing variables, we see that
\begin{eqnarray*}
\int_{\R^d \setminus B(0, \sqrt{h} R(h)^{1-\mu/2})} |K_{\phi,h}(y)| dy
&\leq&
\int_{\R^d \setminus B(0, R(h)^{1-\mu/2})} |K_{\phi}(y)| dy
\\
&\leq&
\ds\frac{1}{R(h)^{(2-\mu)^2/2}}
\int_{\R^d \setminus B(0, R(h)^{1-\mu/2})} |y|^{2-\mu}\,|K_{\phi}(y)| dy
\\
&\leq&
\ds\frac{\omega(R(h))}{R(h)^{(2-\mu)^2/2}}.
\end{eqnarray*}
Since $0 < (2-\mu)/2 < 1$, it follows that
\begin{eqnarray*}
|R_{h}(\lambda)| &\leq& \left( \ds\frac{\omega(R(h))}{R(h)^{2 - \mu}} \right)^{1- \mu/2}
\;=\; h^{1-\mu/2}
\;=\; o(\sqrt{h}).
\end{eqnarray*}

Further,  the fact that $|\nabla \varphi(y)| > 0$ on $B(x,\sigma)$ show that 
$\tilde{g}_h$ is continuous in $\lambda$, which proves the claim.

%%%%%%%%%%%%%%%%%%%%%%%%%%%%%%%%%%%%%%%%%%%%
%%%%%%%%%%%%%%%%%%%%%%%%%%%%%%%%%%%%%%%%%%%%%%

Recall that  
$$ \begin{cases}
    G_h^{-} \varphi(x) = \inf \Set{ s \in \R}{ S_h \chi_{\varphi \geq s}(x) < \theta_h } \\
   G_h^{+} \varphi(x) = \sup \Set{ s \in \R}{ S_h \chi_{\varphi \geq s}(x) \geq \theta_h },
   \end{cases}
$$
it follows  from the claim above  that
$$  S_h \chi_{\varphi \geq  G_h^{-} \varphi(x)}(x) = \theta_h + o(\sqrt{h}), \text{ and } \quad S_h \chi_{\varphi \geq  G_h^{+} \varphi(x)}(x) = \theta_h + o(\sqrt{h}), $$

and consequently
$$ \int_{\R^d} K_{\phi,h}(y) \chi_{ G_h^{-} \varphi(x) \leq  \varphi \leq G_h^{-} \varphi(x) + \epsilon(h) }(x-y) dy = o(\sqrt{h}).$$ 

One can use the same argument as in the consistency proof, (in particular see point 7) to show that 
asymptotically, the above integral behaves like 
$$  \int_{\R^d} K_{\phi,h}(y) \chi_{ G_h^{-} \varphi(x) \leq  \varphi \leq G_h^{-} \varphi(x) + \epsilon(h) }(x-y) dy = \frac{\epsilon(h)}{|\nabla \varphi(x)| \sqrt{h}} \int_{p^{\perp}} K_{\phi}(x) d\H^{d-1}(x) + o(\sqrt{h}), $$
where, $p = \frac{\nabla \phi(x)}{|\nabla \phi(x)|}$.
In conclusion, as $\int_{p^{\perp}} K_{\phi}(x) d\H^{d-1}(x)> 0$, we deduce that 
$$ \epsilon(h) =   \frac{|\nabla \varphi(x)| }{\int_{p^{\perp}} K_{\phi}(x) d\H^{d-1}(x)} o(h),$$
which proves the proposition.
\end{proof}

\subsection{Discution}
Our consistency result sheds light on the relationship between the kernel $K_{\phi}$
and the Hamilton Jacobi equation \eqref{eq:hamilton_jacobi_phi}. Proving convergence of a Bence Merriman
Osher type algorithm in our context seems to be very difficult (if true at all). The
argument of \cite{Ishhi_pires_souganidis} does not apply here. 
The main difficulty is that $G^{\pm}_h \varphi$ may not be continuous, even if $\varphi$ is regular.
 Further, we can only show monotonicity of the operators $G_h^{\pm}$ up to $o(h)$ for smooth functions whose gradients do not vanish. The source of these difficulties is really the thresholding in the definition of $G_h^{\pm}$.  \\

Thus, rather than advocating for a BMO algorithm, we have considered in the
next section a numerical scheme where instead of this thresholding, we modify
the convolution product $K_{\phi,h}*\varphi$ using a reaction operator, in the spirit of a phase
field algorithm. More precisely, given a small parameter $\eps > 0$, we may define
$$ G_{h,\eps} \varphi(x) = T_{h,\eps}(K_{\phi,h}*\varphi),$$
where $T_{h,\eps}$ is defined as follows: Given $\lambda \in \R$, $T_{h,\eps}(\lambda) = \psi(\lambda)$  where $\psi$ 
is the solution of the ODE
$$ \begin{cases}
    \psi_t  &= - \frac{1}{\eps^2 } W'(\psi) \\
    \psi(0) &= \lambda,
   \end{cases}
$$
and $W$ a double well potential with wells located at $\psi = 0$ and $\psi =1 $. Note that
if $\varphi = \1_{\Omega}$ is a characteristic function, then
  $$\lim_{\eps \to 0 } G_{h,\eps}  \1_{\Omega} = G^{+}_h \1_{\Omega} = G^{-}_h \1_{\Omega} $$

\section{Numerical simulations}

In the previous section, we proved a consistency result for a Bence Merriman
Osher-type algorithm. Here we numerically investigate the convergence proper-
ties of a related scheme, based on a phase-field discretization. We explained above
why we did not directly implement a BMA algorithm. In the next paragraph, we 
describe the phase-field algorithm based on the operator $\tilde{\Delta}_{\phi}$.

%%-------------------------------------------------------------------------------

\subsection{The $\tilde{\Delta}_\phi$-phase field model and its discretisation}

As an approximation to the anisotropic Allen-Cahn equation~(\ref{eq:ac_ani}), we consider
the following phase-field model
\begin{eqnarray} \label{eqn:phase_field_model}
 \begin{cases}
    u_t = \tilde{\Delta}_{\phi} u - \frac{1}{\eps^2} W^{'}(u) \\
    u(x,0) = q \left( \frac{\textup{dist}(x,\partial E)}{\eps} \right)
   \end{cases}
\end{eqnarray}
We also report tests, where we estimate the $L^1$-error on anisotropic Wulff sets 
(the sets which minimize the anisotropic perimeter under a volume constraint).
To impose volume conservation, we consider a conserved phase-field model,
of the form 
 \begin{eqnarray}  \label{eqn:phase_field_model_conserved}
 \begin{cases}
    u_t(x,t) = \tilde{\Delta}_{\phi} u(x,t) - \frac{1}{\eps^2} W^{'}(u(x,t)) + \frac{1}{\eps} \lambda(t) \sqrt{2 W(u(x,t))},\\ 
    u(x,0) = q \left( \frac{\textup{dist}(x,\partial E)}{\eps} \right).
   \end{cases}
\end{eqnarray}
The parameter
$$\lambda(t) =  \frac{\int_{\R^d}W^{'}(u(x,t))dx }{ \eps \int_{\R^d}\sqrt{2 W(u(x,t))}dx},$$
can be seen as a Langange multiplier, which preserves the mass of $u$.
See~ \cite{BrasselBretin} where schemes of this form have been studied for 
isotropic mean curvature with a volume constraint.

%---------------------------------------------------------------------------
We now describe the numerical method we use for solving the PDE's \eqref{eqn:phase_field_model} and  \eqref{eqn:phase_field_model_conserved}. 
Several studies of classical numerical schemes for the Allen--Cahn equation have already been conducted in the past: see for instance,
\cite{Deckelnick2005,Paolini1997,Chen1998,Chen1998a,VerdiPaolini,Garcke1,Garcke2}.
Here, the computational domain is the fixed box $Q = [-1/2,1/2]^d \subset \R^d$, $d = 2,3$.
The initial datum is $u_0 = q(\frac{\textup{dist}(x, \partial \Omega_0}{\eps})$,
where $\Omega_0$ is a smooth bounded set strictly contained $Q$.
We assume that during the evolution, the set 
$\Omega_{\eps}(t) := \{u_\eps(x,t) = 1/2\}$ remains strictly inside $Q$,
so that we may impose periodic boundary conditions on $\partial Q$.

Our strategy consists in representing $u$ as a Fourier series 
in $Q$, and in using a splitting method. First, one applies the
diffusion operator, which given the form of $\tilde{\Delta}_\phi$,
merely amounts to a multiplication in the Fourier space.
The interesting feature of our approach is that this step is fast
and very accurate. Next, the reaction term is applied.

More precisely, $u_\eps(x,t_n)$ at time $t_n = t_0 + n \delta t$
is approximated by
\begin{eqnarray*}
u_\e^P(x,t_n) &=& \sum_{\max_{1 \leq i \leq d}|p_i|  \leq P}
u_{\e,p}(t_n) e^{2i\pi p\cdot x}.
\end{eqnarray*}
In the diffusion step, we set
\begin{eqnarray*}
u_\e^P(x, t_n +1/2) = \sum_{\max_{1 \leq i \leq d}|p_i|  \leq P}
u_{\e,p}(t_n) e^{-4 \pi^2 \delta t \, \phi^o(p)^2}
e^{2i\pi p\cdot x}.
\end{eqnarray*}
We then integrate the reaction terms
\begin{eqnarray*}
u_\e^P(x,t_n+1) &=& u_\e^P(x,t_n+1)
- \displaystyle{\delta t}{\e^2} W^\prime_{i,\e}(u_\e^P(x, t_n +1/2)).
\end{eqnarray*}
In practice, the first step is performed via a fast Fourier transform,
with a computational cost $O(P^d \ln(P))$.

The corresponding numerical scheme turns out to be stable 
when solving~\eqref{eqn:phase_field_model}, under the condition
$\delta t \leq M\e^2$,
where $M =  \left[ \sup_{t \in [0,1]} \left\{  W^{''}(t)\right\} \right]^{-1}$.
Numerically, we observed that this condition is also sufficient for the
conserved potential in~(\ref{eqn:phase_field_model_conserved}).
In the simulations, we used $W(s) = \frac{1}{2} s^2(1-s)^2$.

The isotropic version of our splitting scheme has been studied in~\cite{BrasselBretin}.
It is shown there that this scheme converges with the same rate as phase-field approximations
based on a spatial discretization by finite differences or by finite elements. 
Its advantages are greater precision, and unconditionnal stability.

%%-------------------------------------------------------------------------

\subsection{Test of convergence in dimension $2$}

We consider following anisotropic densities
 $$ \begin{array}{l} 
     \phi^o_1(\xi)=  \| \xi \|_{\ell^4} = \left( |\xi_1|^4 +  |\xi_2|^4 \right)^{\frac{1}{4}}
     \\
     \phi^o_2(\xi)= \| \xi \|_{\ell^{\frac{4}{3}}} = \left( |\xi_1|^{\frac{4}{3}} +  |\xi_2|^{\frac{4}{3}} \right)^{\frac{3}{4}}
     \\
          \phi^o_3(\xi)= \left( |\xi_1|^{1,001} + |\frac{1}{2}\xi_1 + \frac{\sqrt{3}}{2} \xi_2|^{1,001} + |\frac{1}{2}\xi_1 - \frac{\sqrt{3}}{2} \xi_2|^{1,001} \right)^{\frac{1}{1,001}}. 
    \end{array}
 $$
See figure~\eqref{fig:forme_wulff_frank_ani1} for a representation of their
Wulff sets $B_{\phi_i}$ and Frank diagrams $B_{\phi^o_i}$. 

\begin{figure}
\begin{center}  \label{fig:forme_wulff_frank_ani1}
\includegraphics[width=4cm]{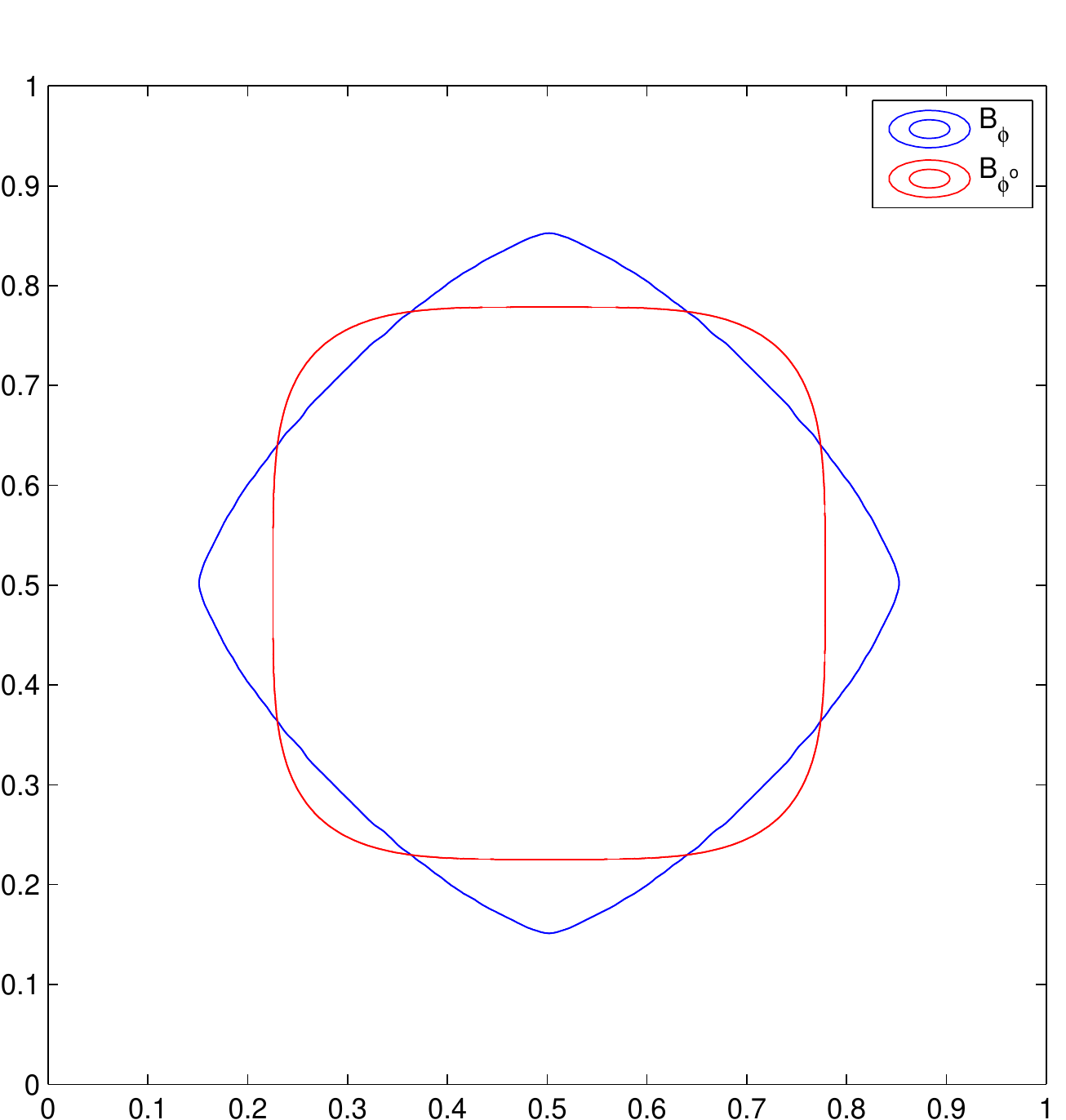}
\includegraphics[width=4cm]{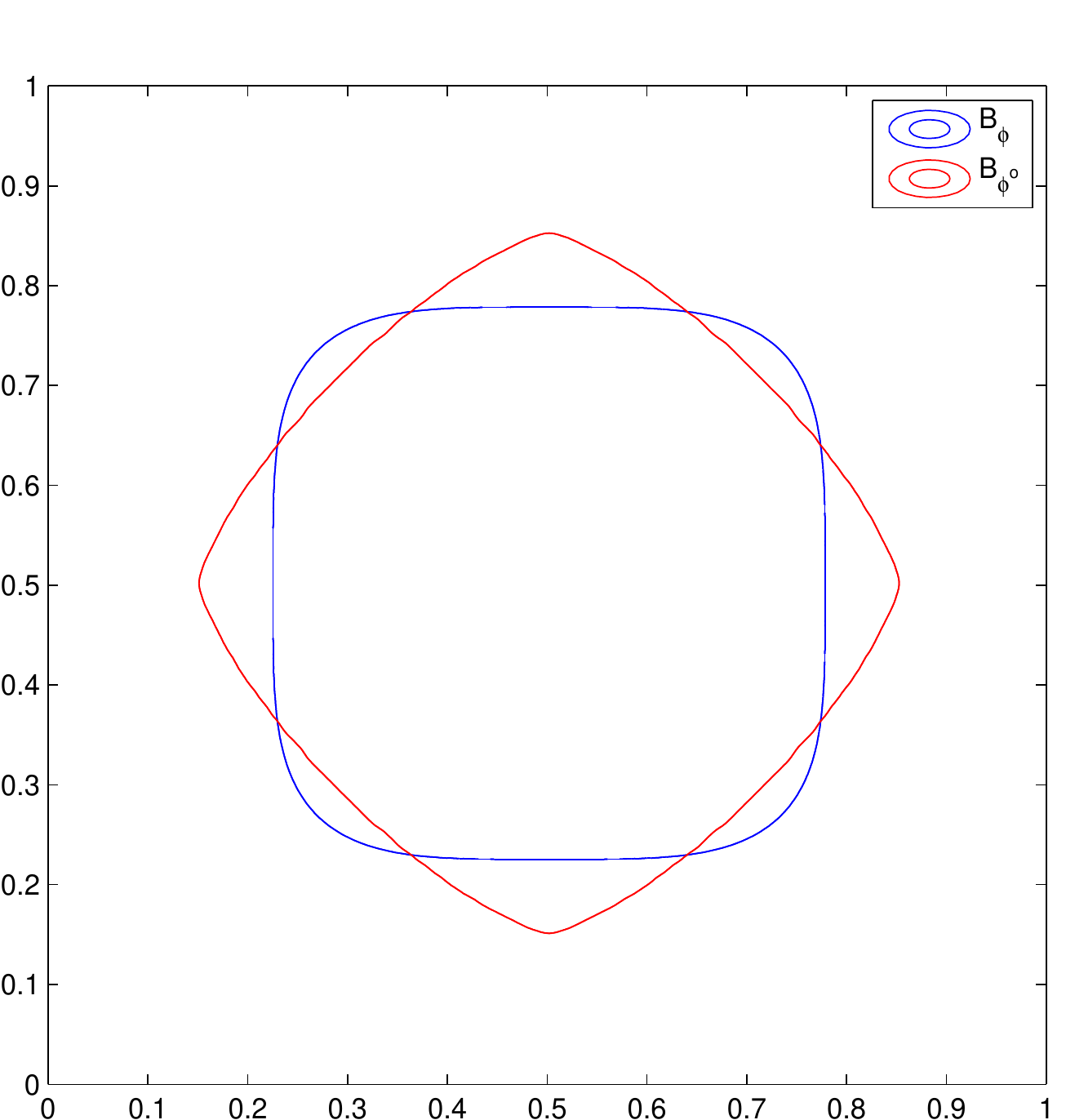}
\includegraphics[width=4cm]{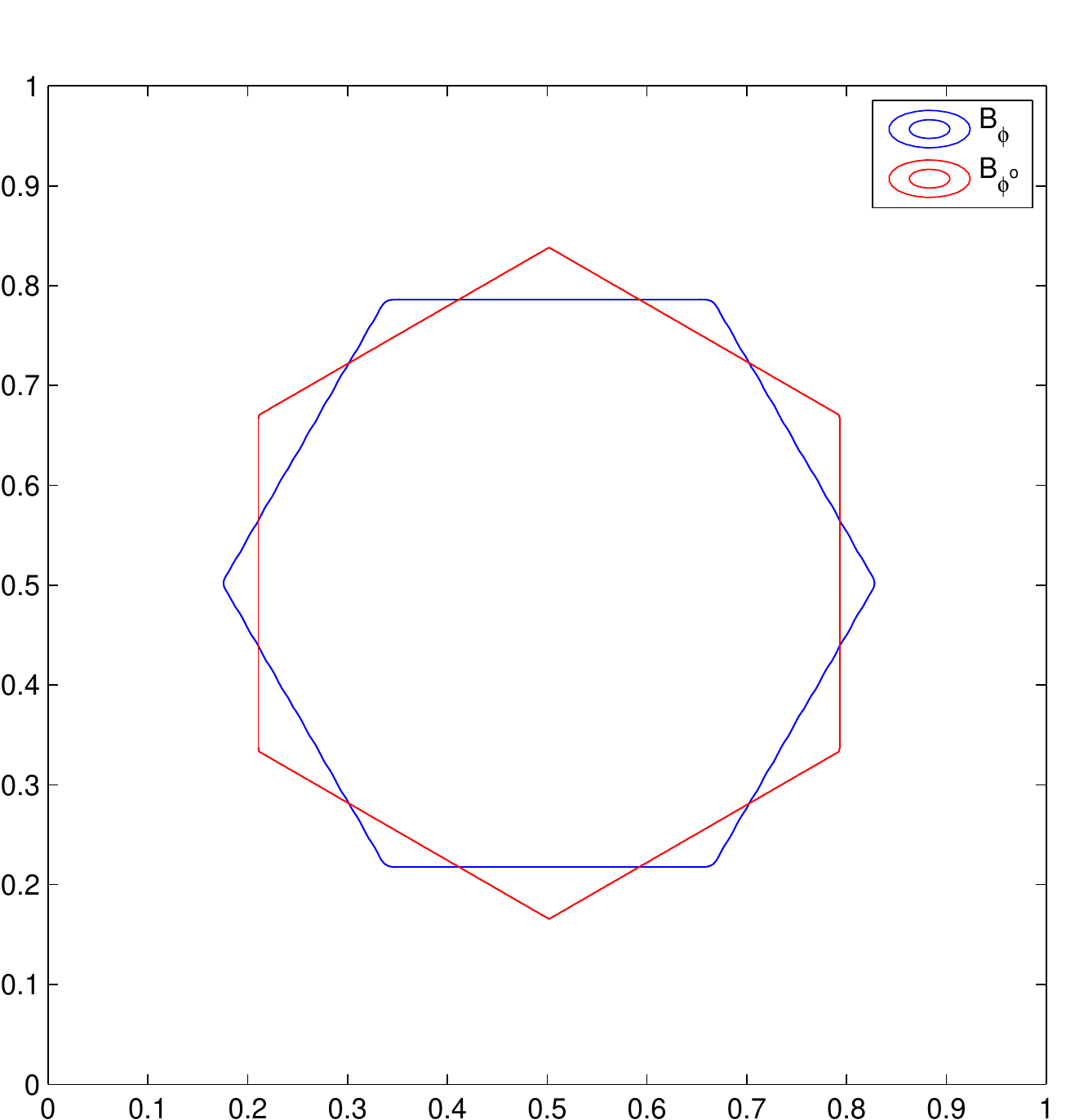}
\caption{ Wulff Set (blue) and Frank diagram (red) for the anisotropic densities 
$(\phi_1,\phi^o_1)$, $(\phi_2,\phi^o_2)$ and $(\phi_3,\phi^o_3)$    }
\end{center}
\end{figure}

%%--------------------------------------------------------------------

{1. \it Evolution from a Wulff set.}~\\
We consider the equation 
 $$ \begin{cases}
      \partial_t u &= \tilde{\Delta}_{\phi} u - \frac{1}{\eps^2} W'(u) \\
      u(0,x) &= q \left( dist(x, \Omega_0) /\eps^2 \right),  
    \end{cases}
$$ 

where the initial set $\Omega_0$ is a Wulff set of radius $R_0 = 0.25$ 
$$ 
\Omega_0 = \Set{x \in \R^2}{ \phi(x) \leq R_0}.
$$ 
It is well known that the set $\Omega(t)$
obtained from $\Omega_0$ through evolution by anisotropic mean curvature
is a  Wulff set with radius $R(t) = \sqrt{R_0^2 - 2t}$, 
which decreases to a point at the extinction time $t_{ext} = \frac{R_0^2}{2}$.
In these simulations, the number of Fourier modes is $P = 2^8$, 
and the time step and phase-field parameter are chosen to be 
$\delta_t=1/P^2$ and $\epsilon = 1/P$. 
On figure~(\ref{evol_ani}) the interface $\Omega(t)$ is plotted at different times. 
We observe a good agreement between the theoretical and computed curves, in spite of
the smoothening of the corners of the latter.

\begin{figure} 
 \begin{center} \label{evol_ani}
 \includegraphics[width=4cm]{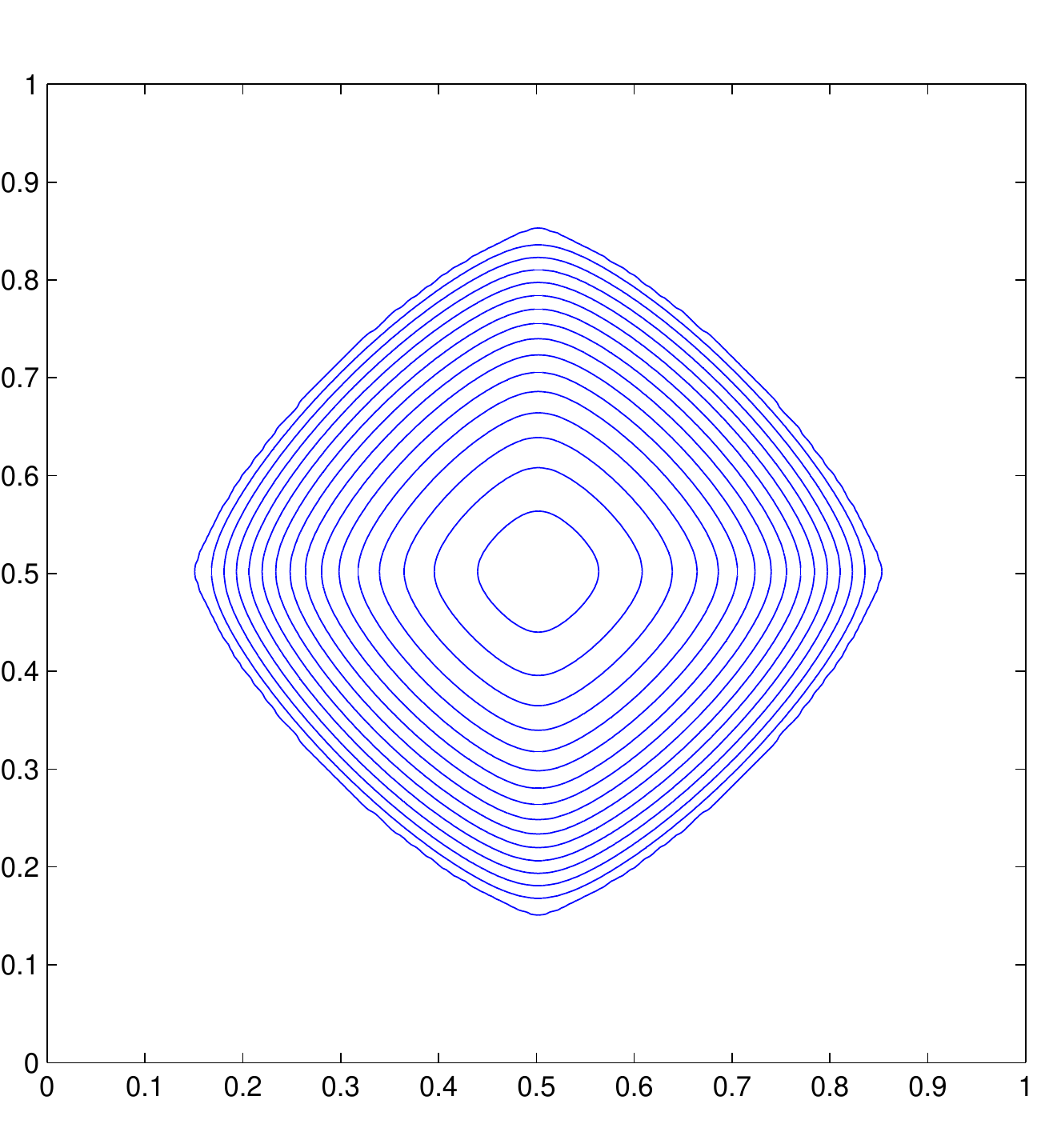}
 \includegraphics[width=4cm]{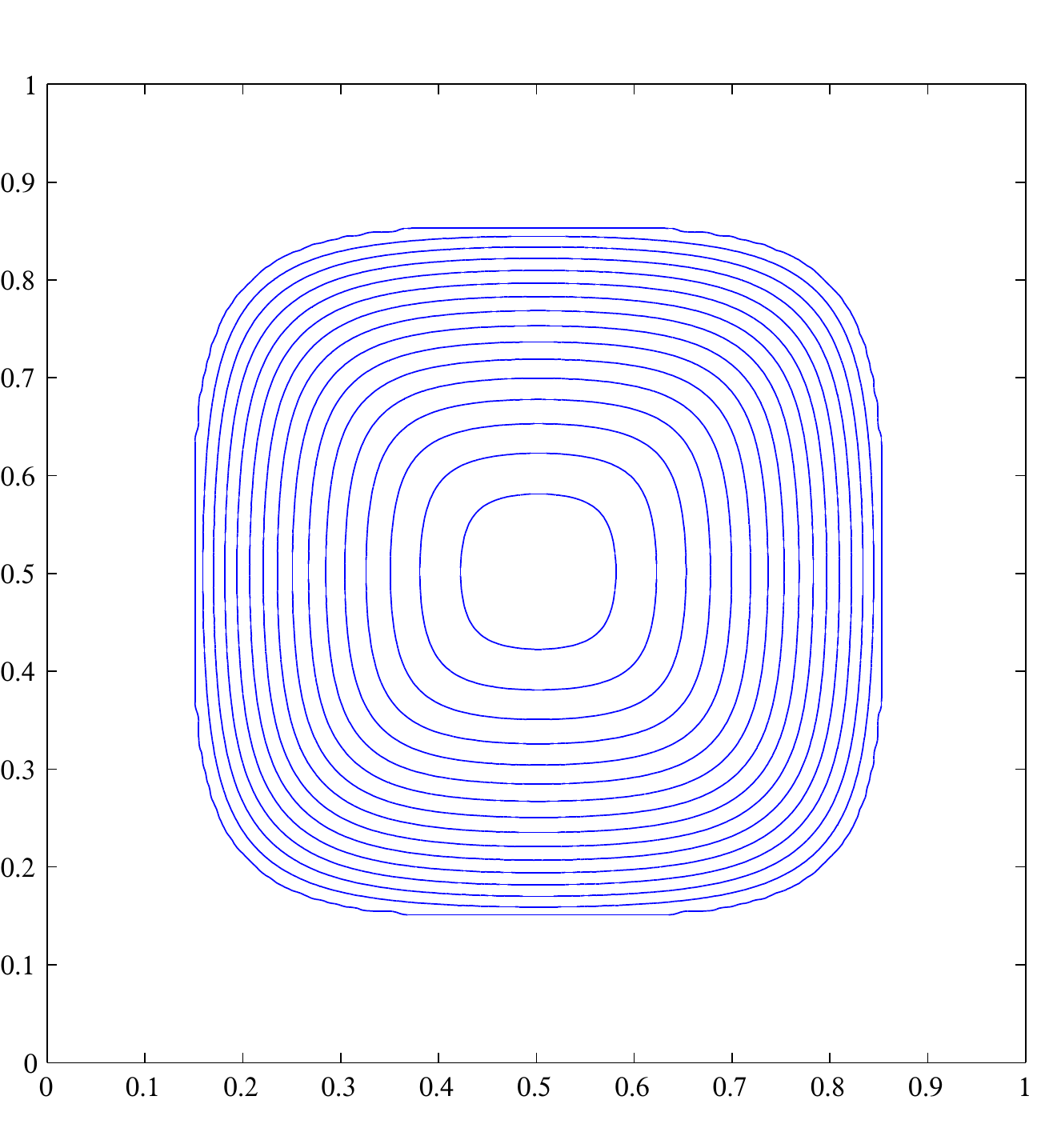}
 \includegraphics[width=4cm]{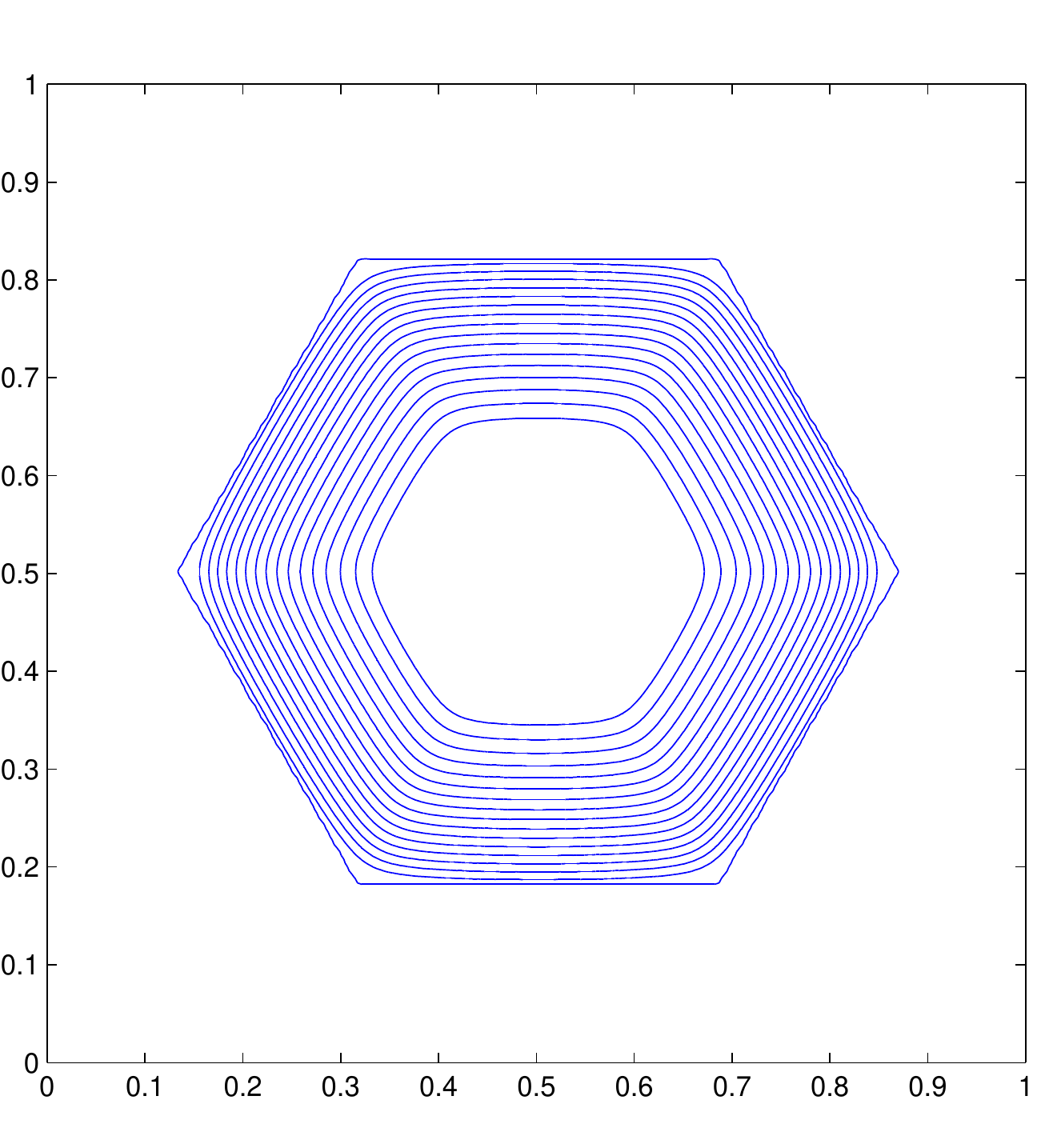}
  \caption{ $\Omega(t)$ at different times for the anisotropic densities $\phi_1, \phi_2, \phi_3$}
  \label{fig:evol_ani_eps_1}
 \end{center}
\end{figure}

%%--------------------------------------------------------------------

{2.\it Convergence to the Wulff set }~\\

This smoothening of corners actually depends on the thickness $\eps$ of the diffuse interface,
as evidenced in the next series of tests, 
of evolution by anisotropic mean curvature under a volume constraint 
according to~(\ref{eqn:phase_field_model_conserved}).
%%, for different values of $\eps$. We observe that 
%%the smaller $\eps$, the sharper the corners.
The initial set $\Omega_0$ is a circle centered at $0$, of the same volume as
$\Omega^{*} = \Set{x \in \R^d}{ \phi(x) < R_0}$. 
The evolution $\Omega(t)$ from $\Omega_0$ is expected to converge 
to the Wulff set $\Omega^{*}$.

Figures~\ref{evol_ani2}-a,b represent the final sets $\Omega^{*}_{\epsilon}$  obtained from the resolution of anisotropic Allen-Cahn equation, with respective
anisotropic densities $\phi_1$ and $\phi_2$, and for differents value of $\epsilon$. 
We observe that the smaller $\eps$, the better the approximation of the Wulff set.
In figure~\ref{evol_ani2}-c, the $L^1$ error
$$ \eps \to  \| \1_{\Omega^{*}} - \1_{\Omega^{*}_{\epsilon}} \|_{L_1(\R^d)},$$
is plotted in a logarithmic scale.
This graph indicates that this error is of order $\eps$.

\begin{figure} 
 \begin{center} \label{evol_ani2}
 \includegraphics[width=4cm]{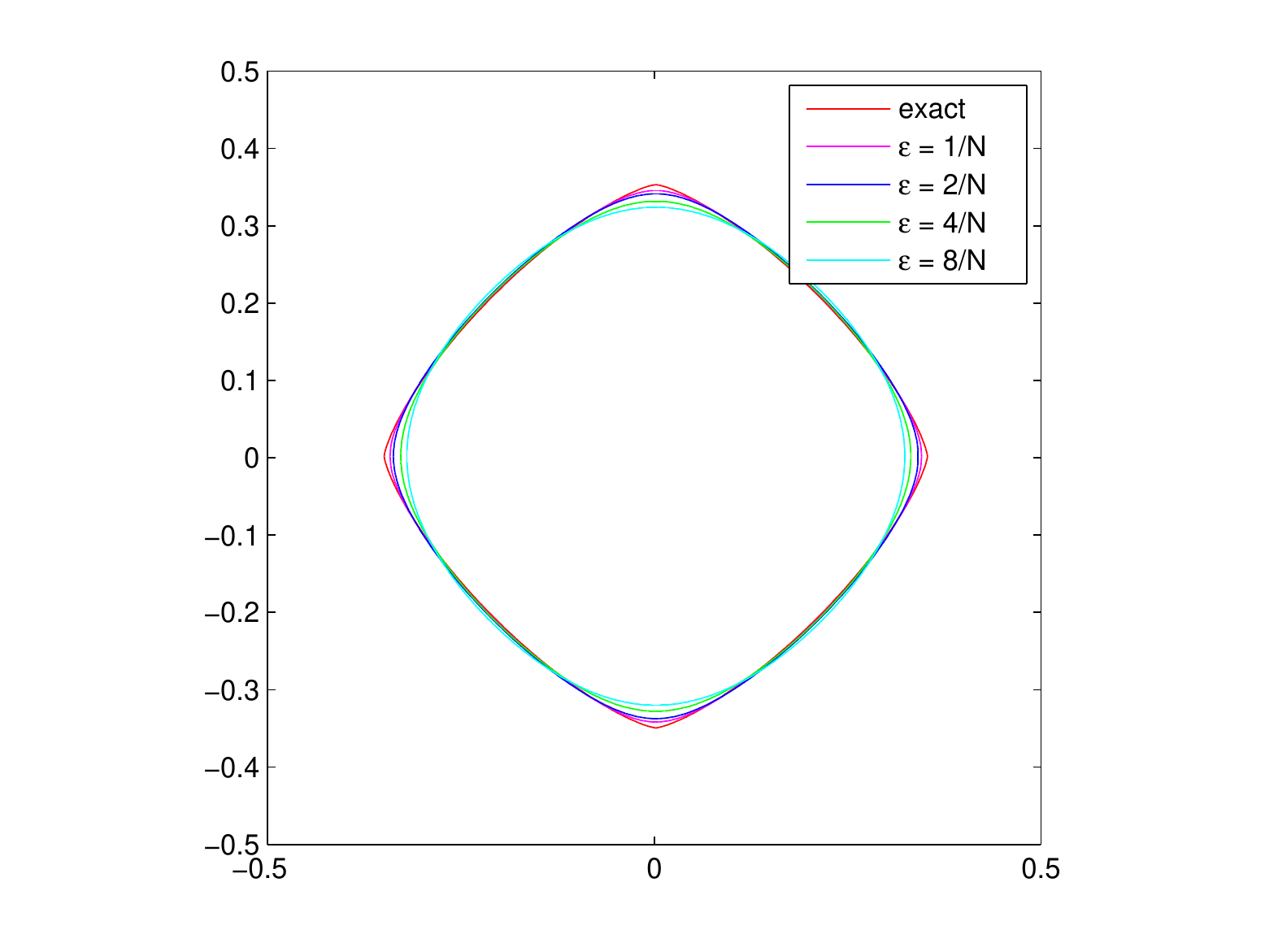}
 \includegraphics[width=4cm]{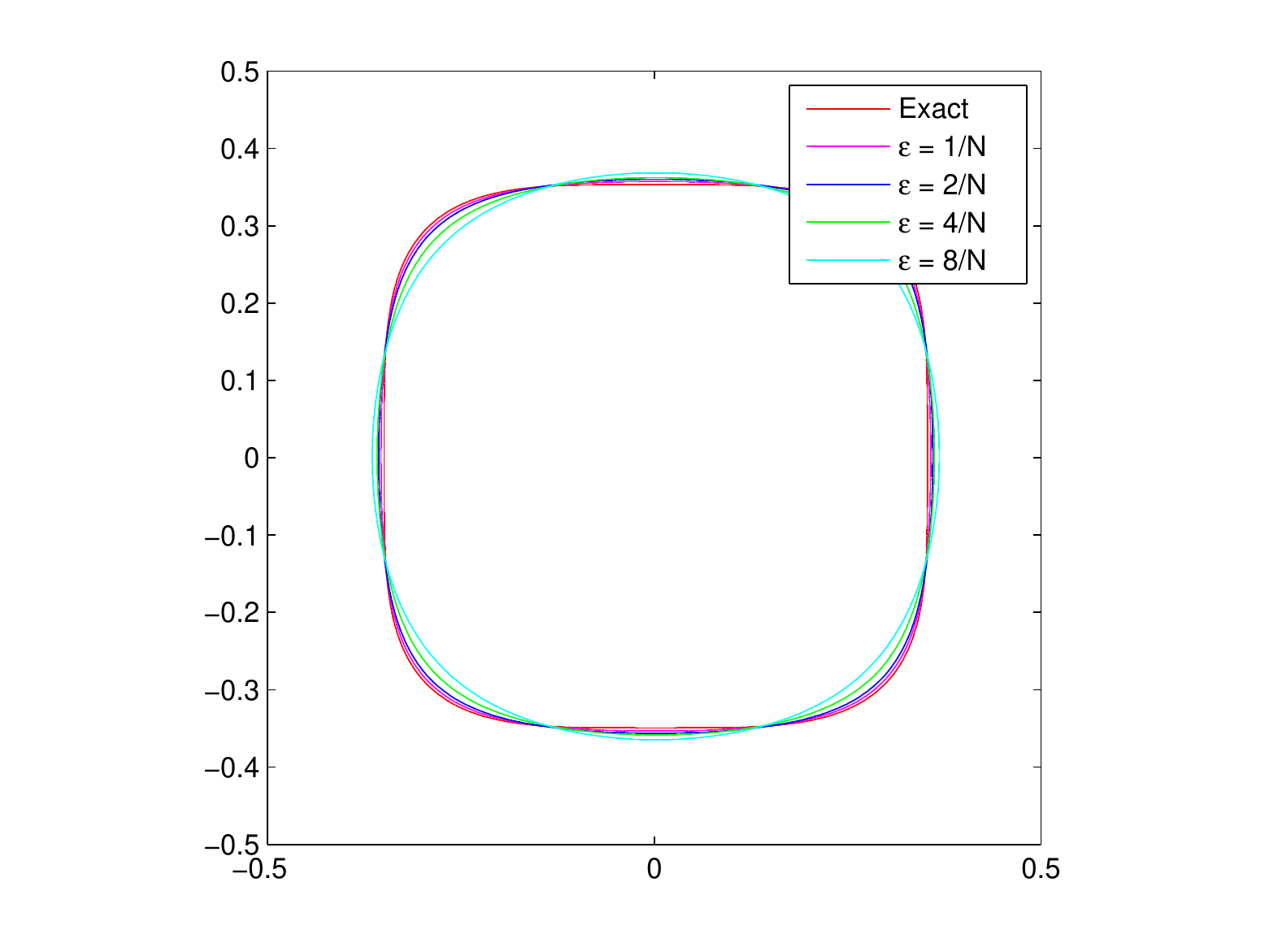}
 \includegraphics[width=4cm]{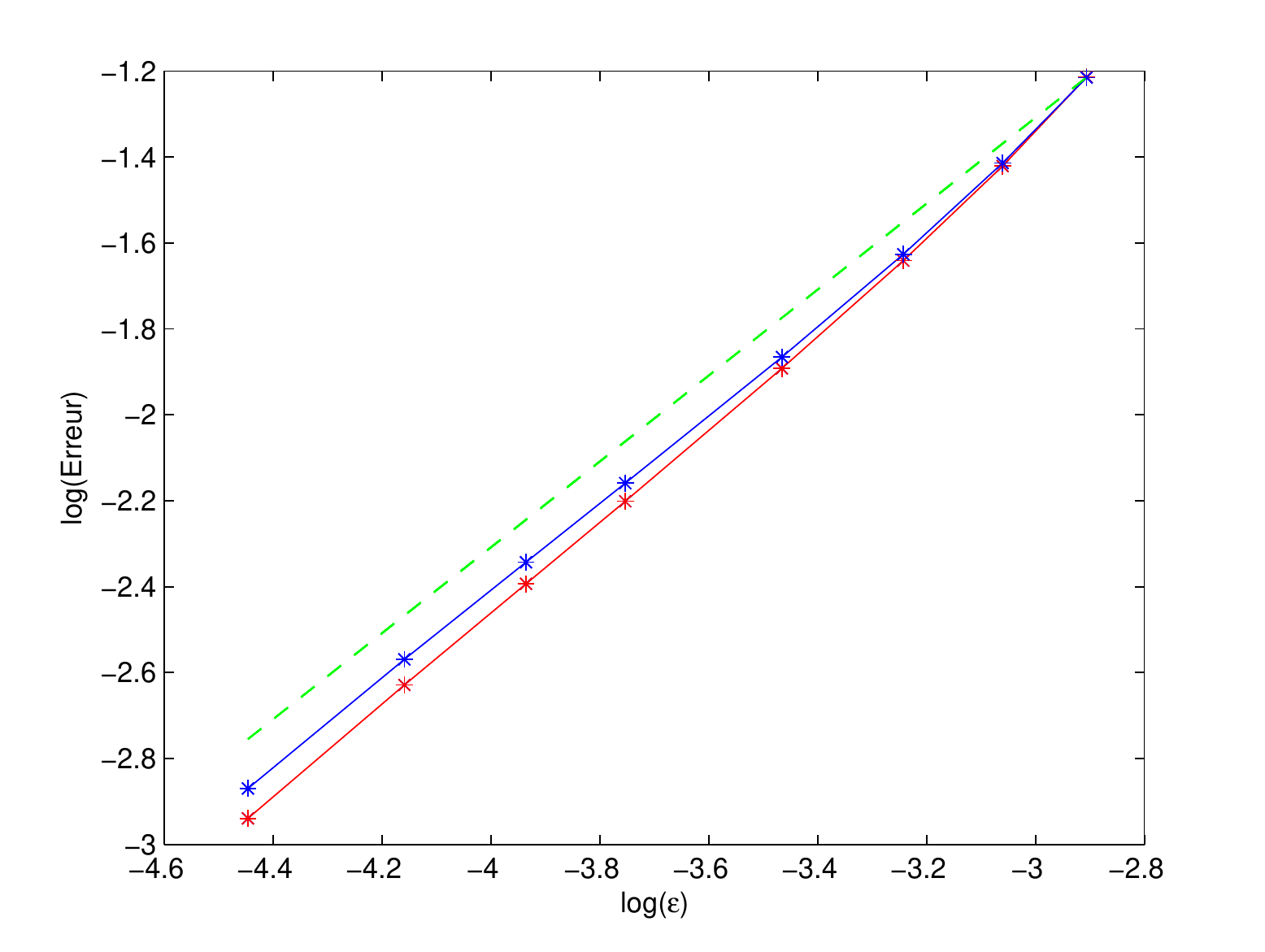}
  \caption{ From left to right :  $\Omega(t)$ at different times with anisotropy $\phi^o_1$,  
  $\Omega(t)$ at different times with anisotropy $\phi^o_2$, error estimate $\eps \to \| \1_{B_{\phi}^{\epsilon}} - \1_{B_{\phi,R_0}} \|_{L_1(\R^d)}$ in logarithmic scale ( $\phi^o_1$ in red and $\phi^o_2$ in blue)}
  \label{fig:evol_ani_eps_1}
 \end{center}
\end{figure}
  
%%--------------------------------------------------------------------

\subsection{Some $3$D simulations}
As final illustrations, we consider the anisotropic densities  
$$ \begin{cases}
\phi^o_4(\xi) &=  \sqrt{\xi_1^2 + \xi_2^2} + |\xi_3|  \\
\phi^o_5(\xi) &=   |\xi_1| + |\xi_2| + |\xi_3|\\
   \end{cases}
$$ 
The corresponding Wulff sets and Frank diagrams are
plotted in figure~\eqref{fig:test3d_ani_phio}.
  
\begin{figure}[htbp]
 \begin{center}
 \includegraphics[width=3cm]{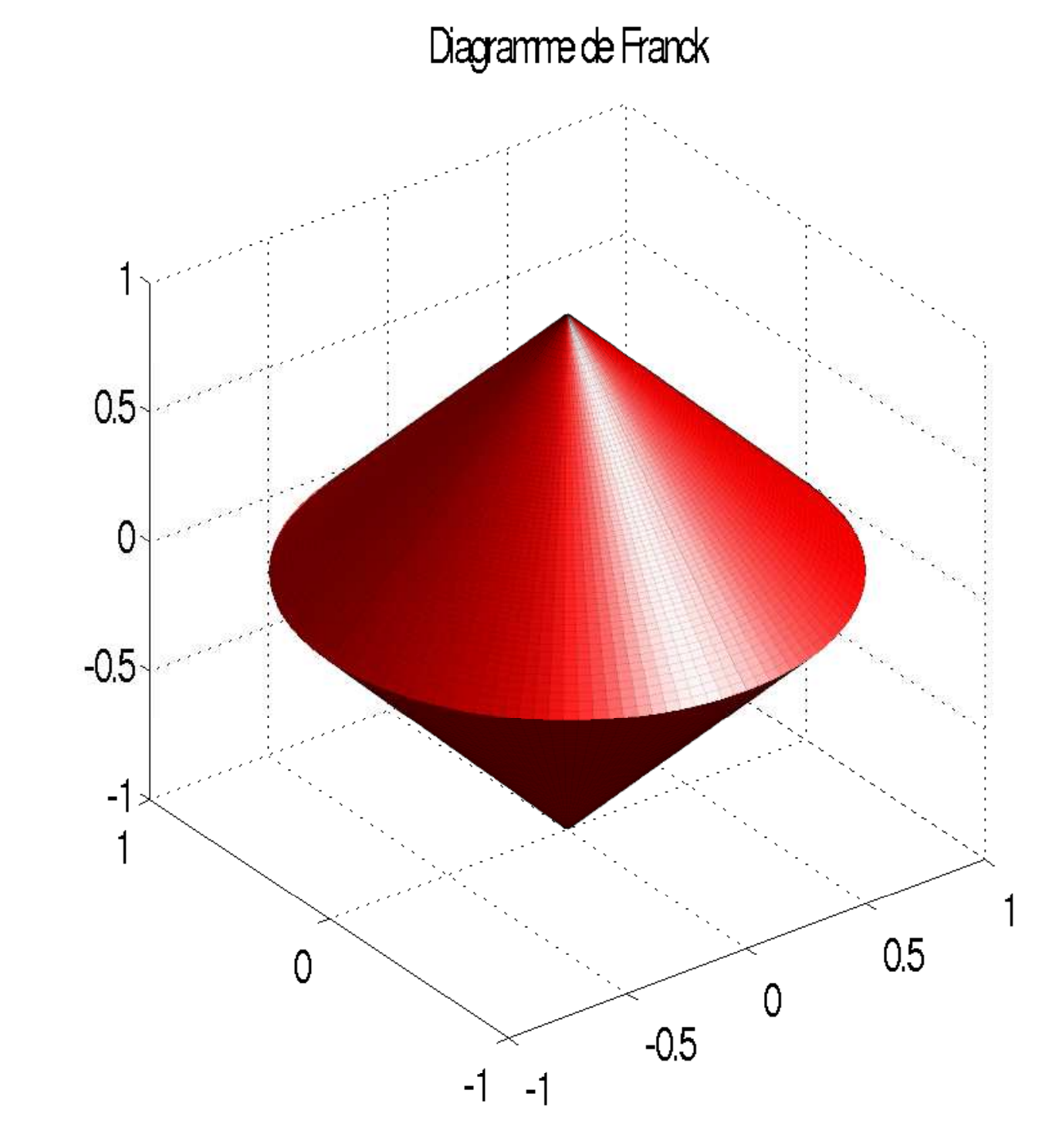}
 \includegraphics[width=3cm]{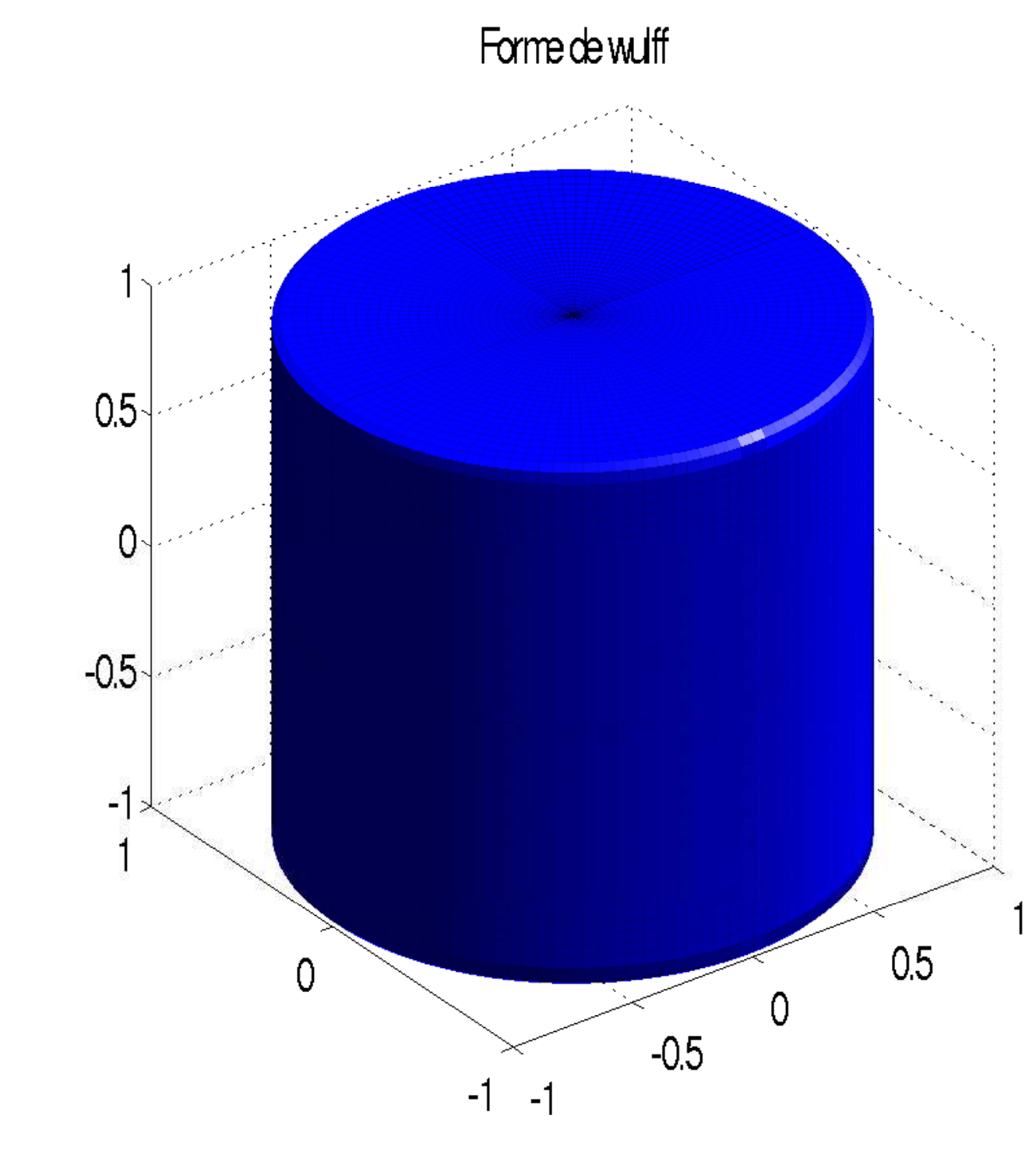}
  \includegraphics[width=3cm]{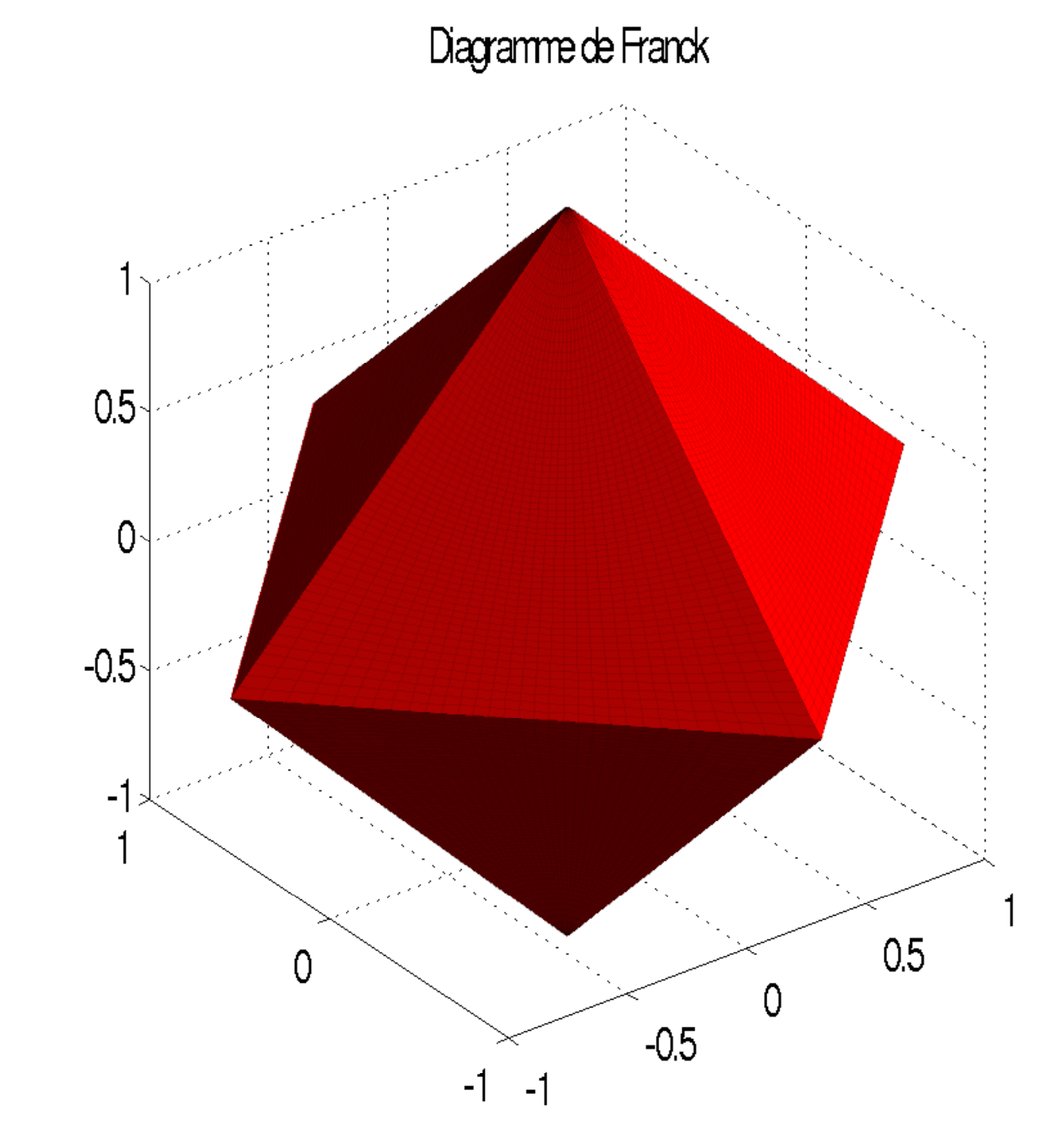}
  \includegraphics[width=3cm]{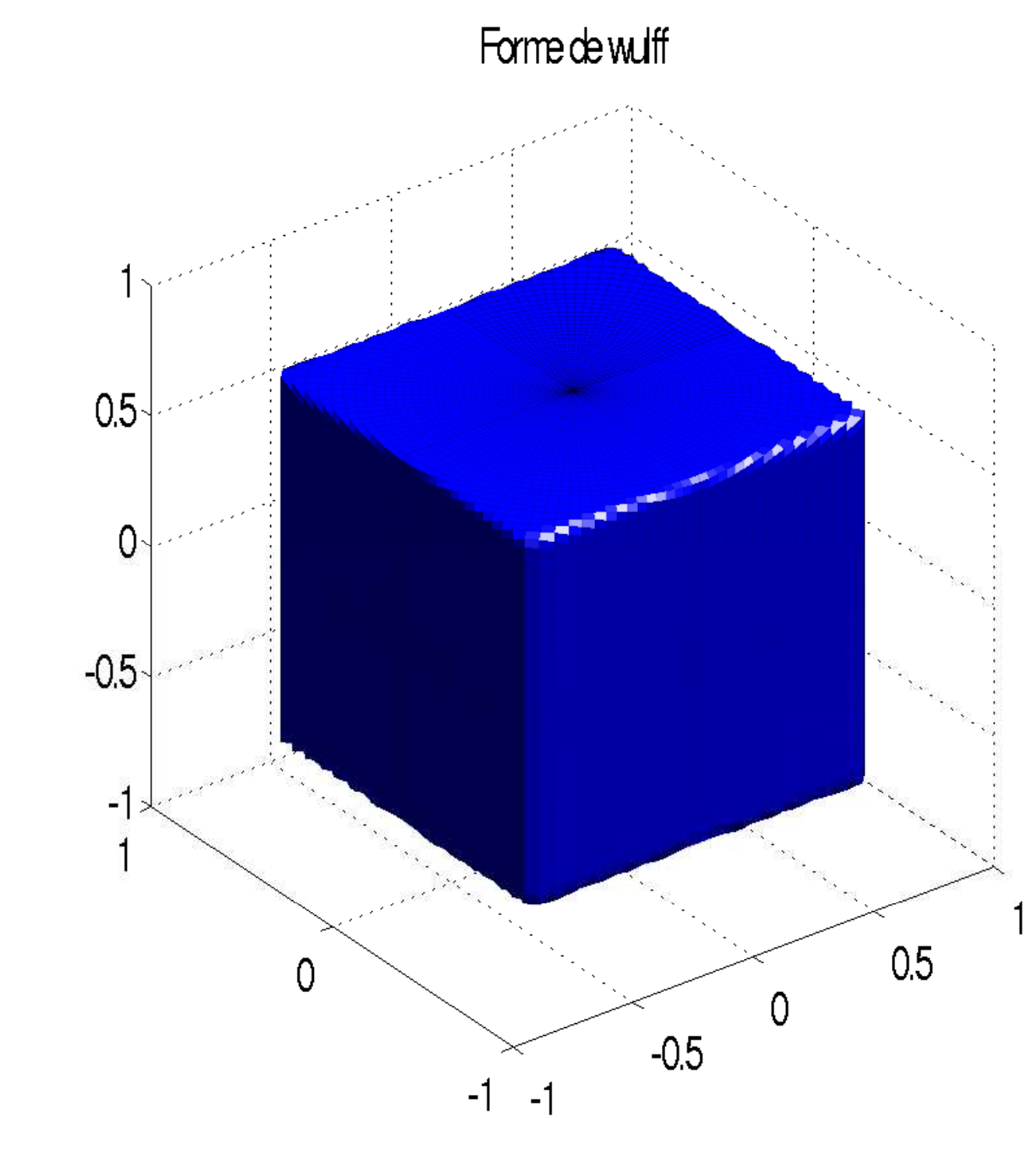}
  \caption{ Franck diagramm and  Wulff set :  $B_{\phi^o_4}$, $B_{\phi_4}$ , $B_{\phi^o_5}$, $B_{\phi_5}$ }
  \label{fig:test3d_ani_phio}
 \end{center}
\end{figure}

We report in figure~\eqref{fig:test3d_ani_tore2} (respectively in figure~\eqref{fig:test3d_ani_tore3}) the evolution by $\phi^o_4$ (resp. $\phi^o_5$) anisotropic mean curvature 
from an initial torus. 
The number of Fourier modes is $P = 2^7$, the time step and diffuse interface thickness
are $\delta_t=1/P^2$ and $\epsilon = 1/P$.

\begin{figure}[htbp]
 \begin{center}
 \includegraphics[width=4cm]{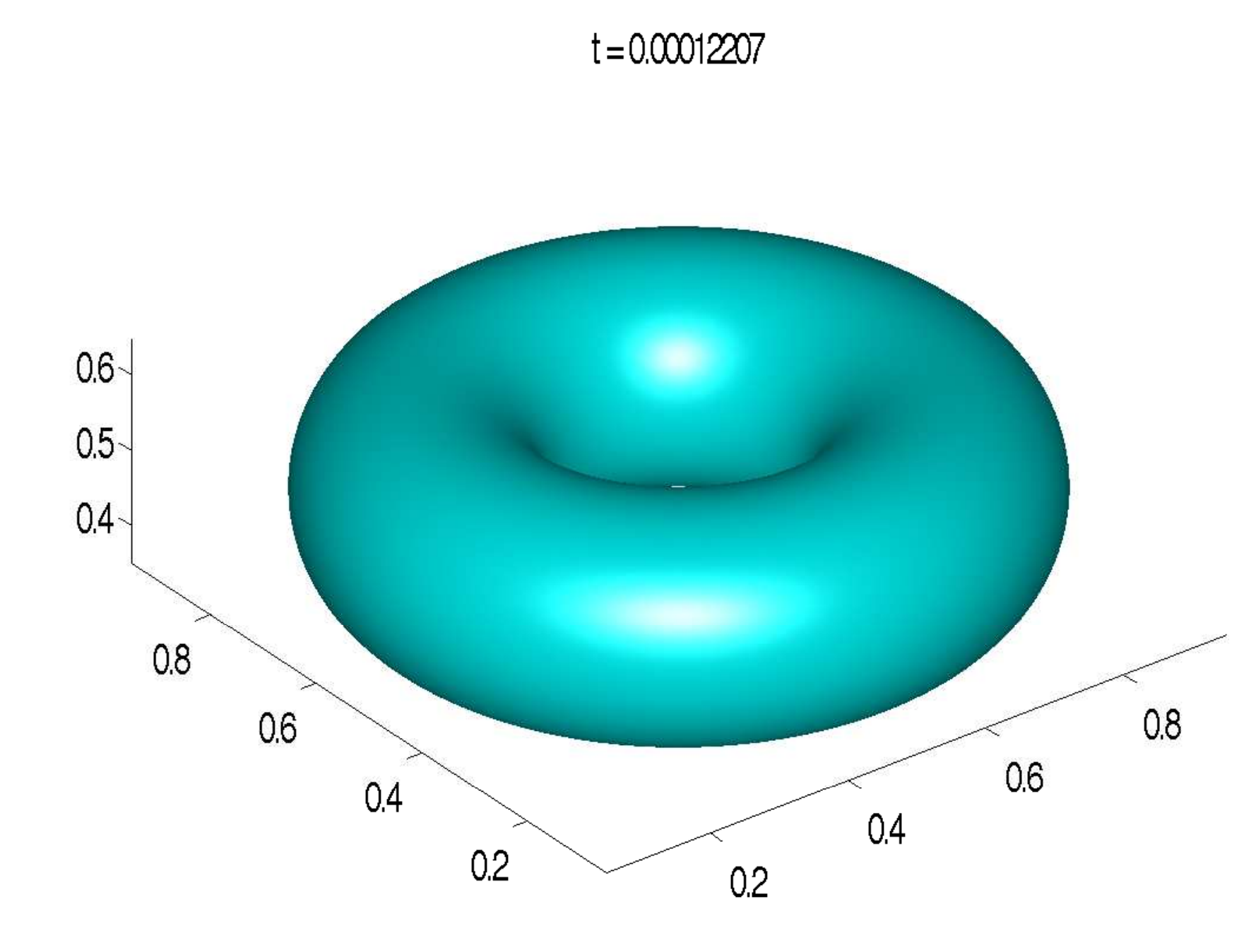}
\includegraphics[width=4cm]{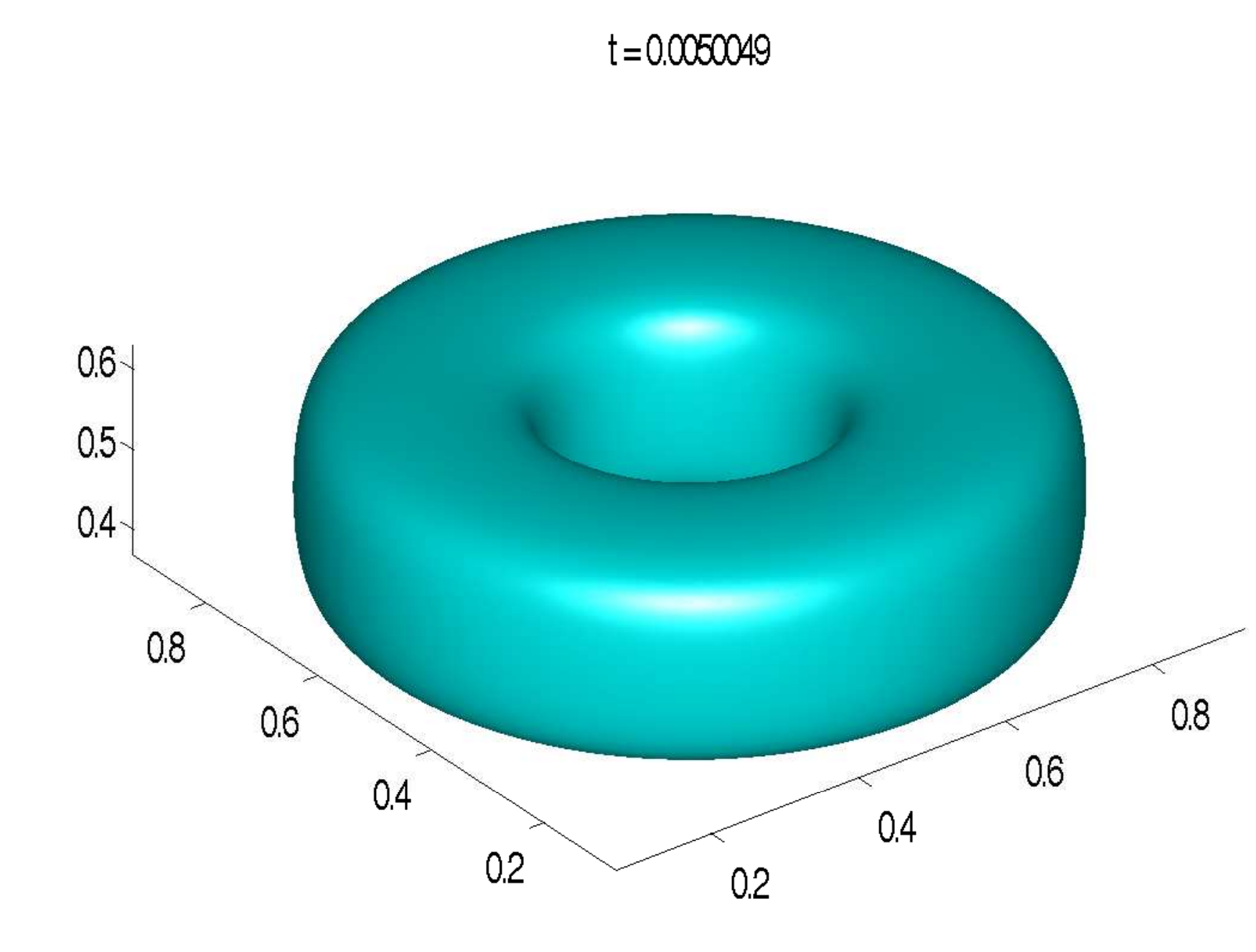}
\includegraphics[width=4cm]{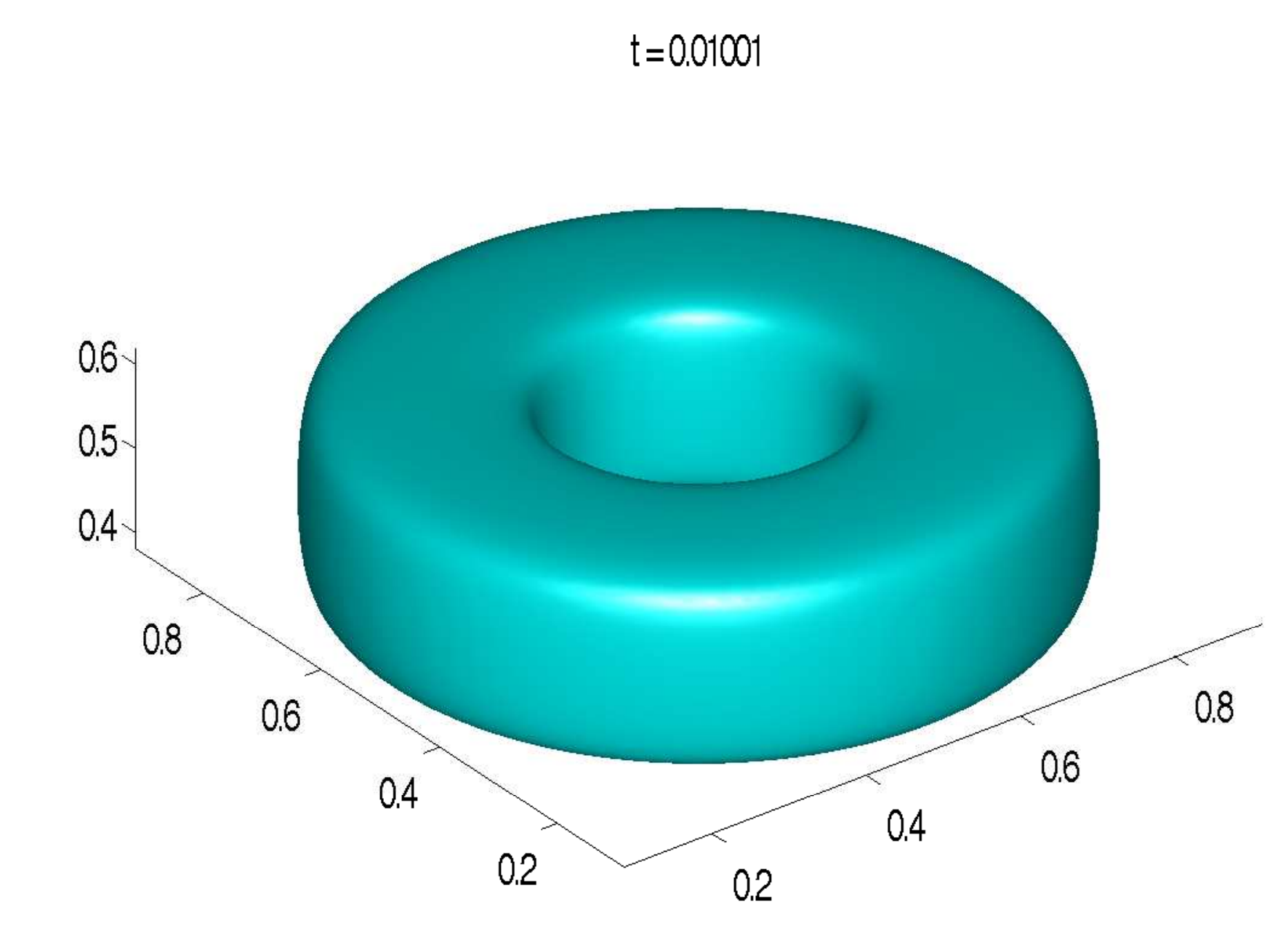}
\includegraphics[width=4cm]{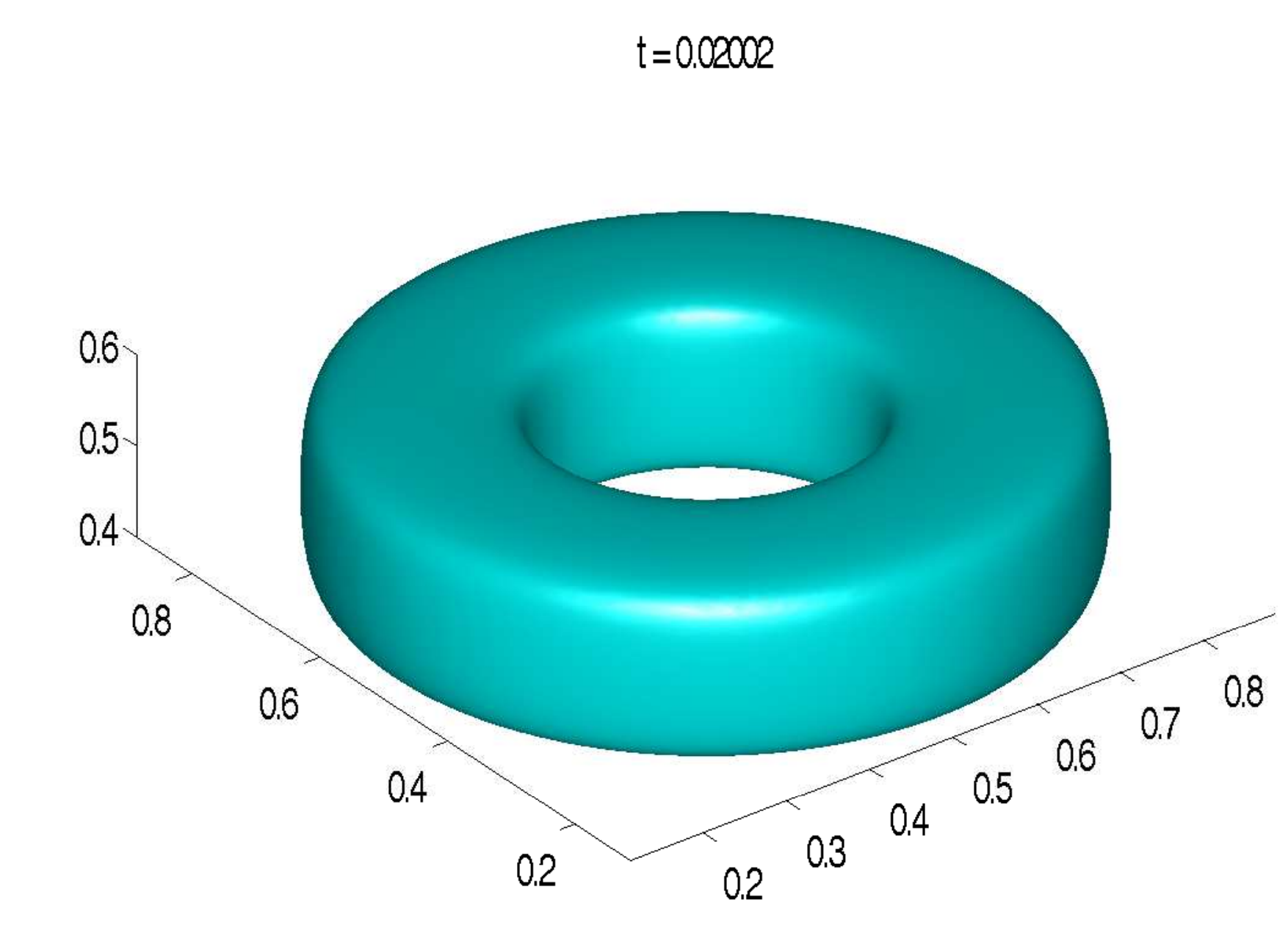}
\includegraphics[width=4cm]{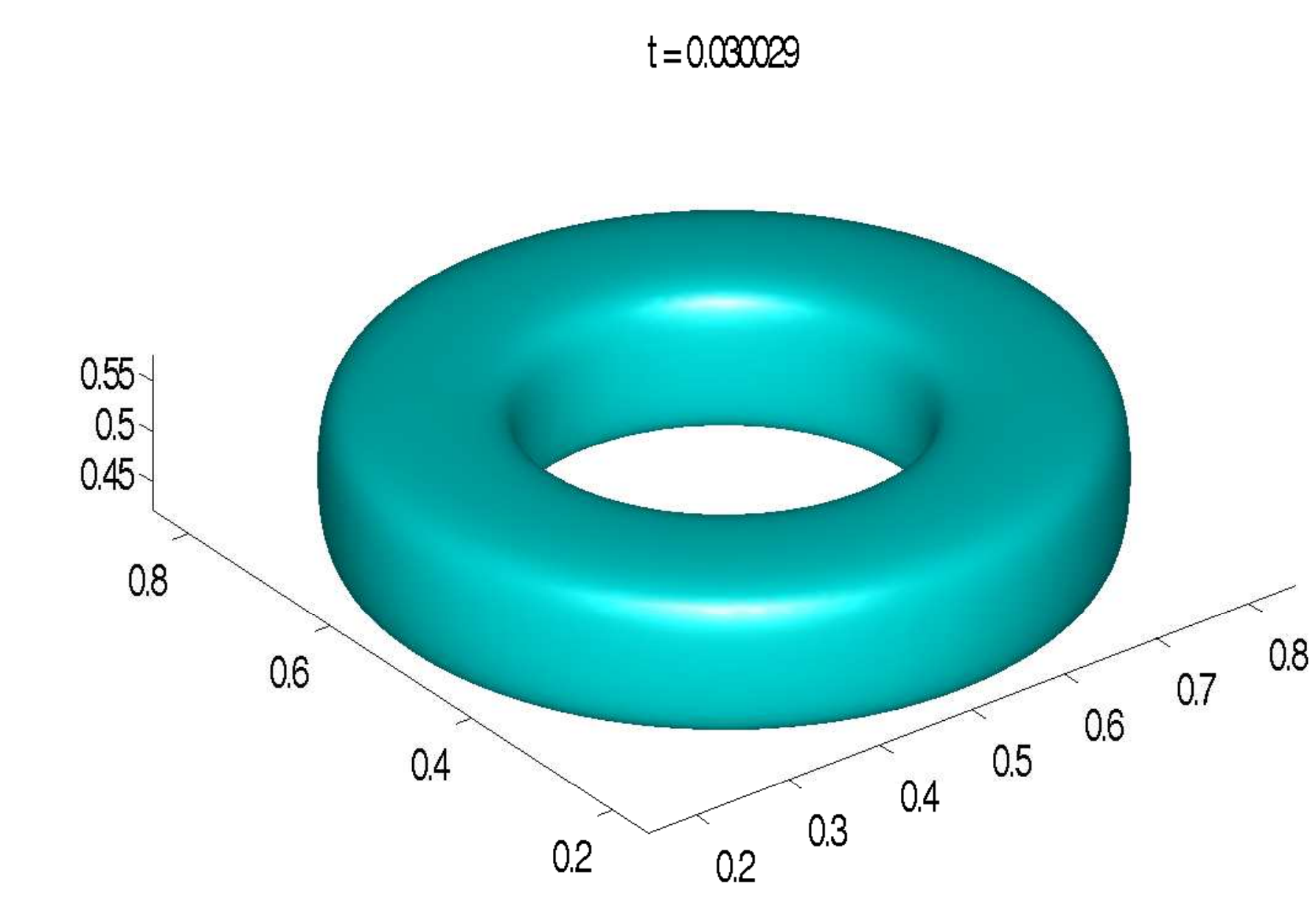}
\includegraphics[width=4cm]{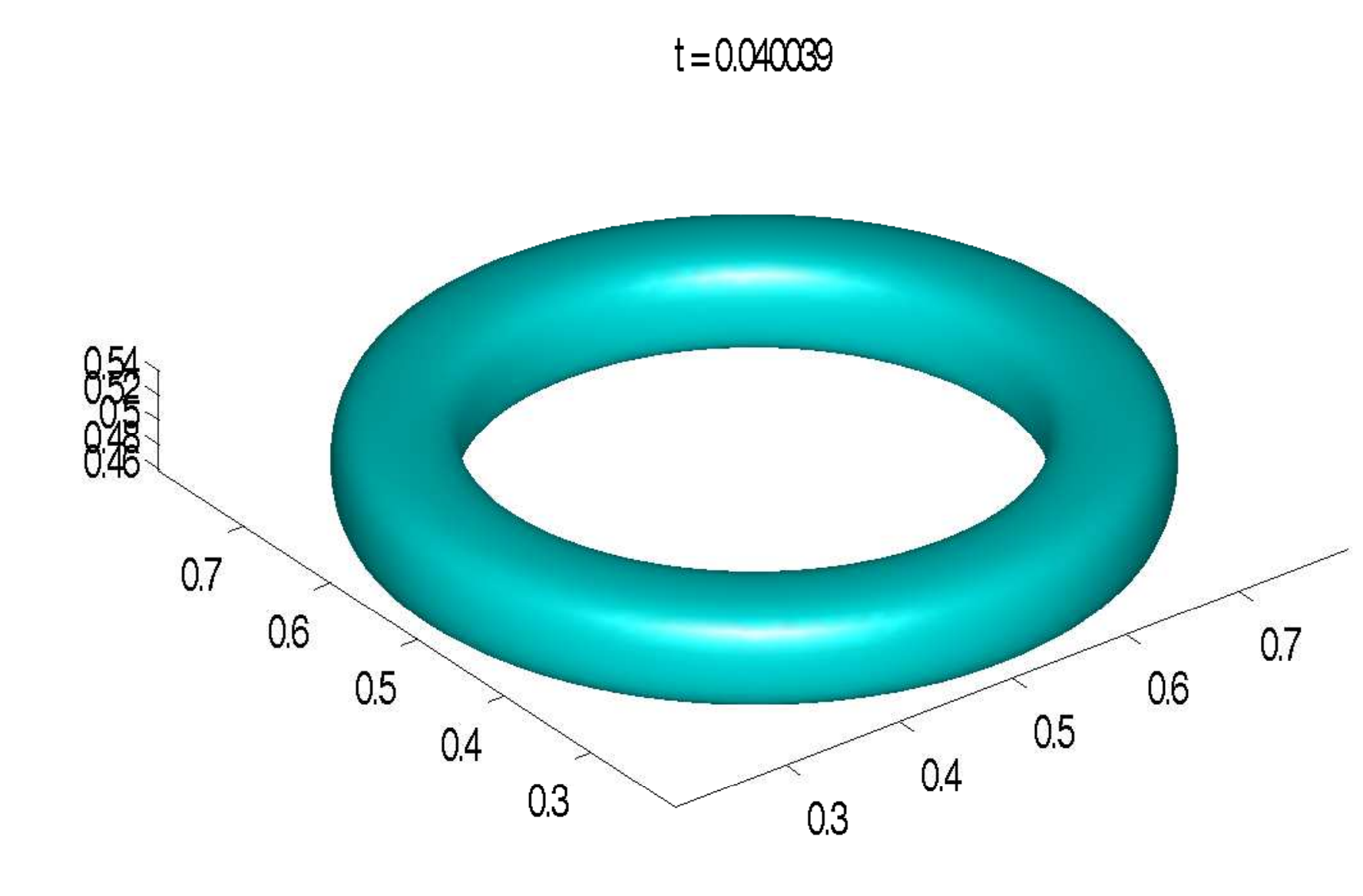}
  \caption{$\phi_4^o(\xi)$-evolution from an initial torus, at different times}
  \label{fig:test3d_ani_tore2}
 \end{center}
  \end{figure}
  
    \begin{figure}[htbp]
 \begin{center}
 \includegraphics[width=4cm]{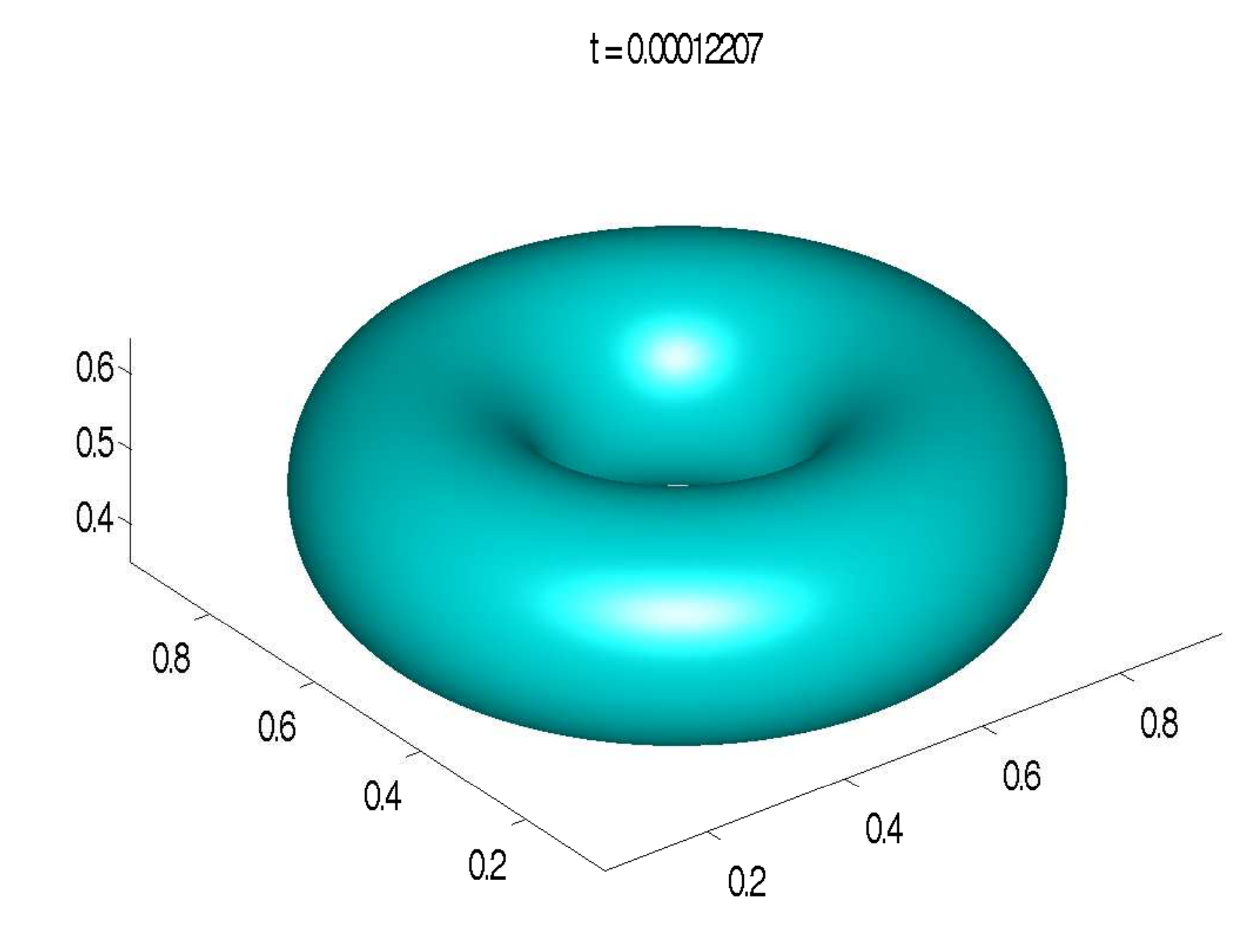}
\includegraphics[width=4cm]{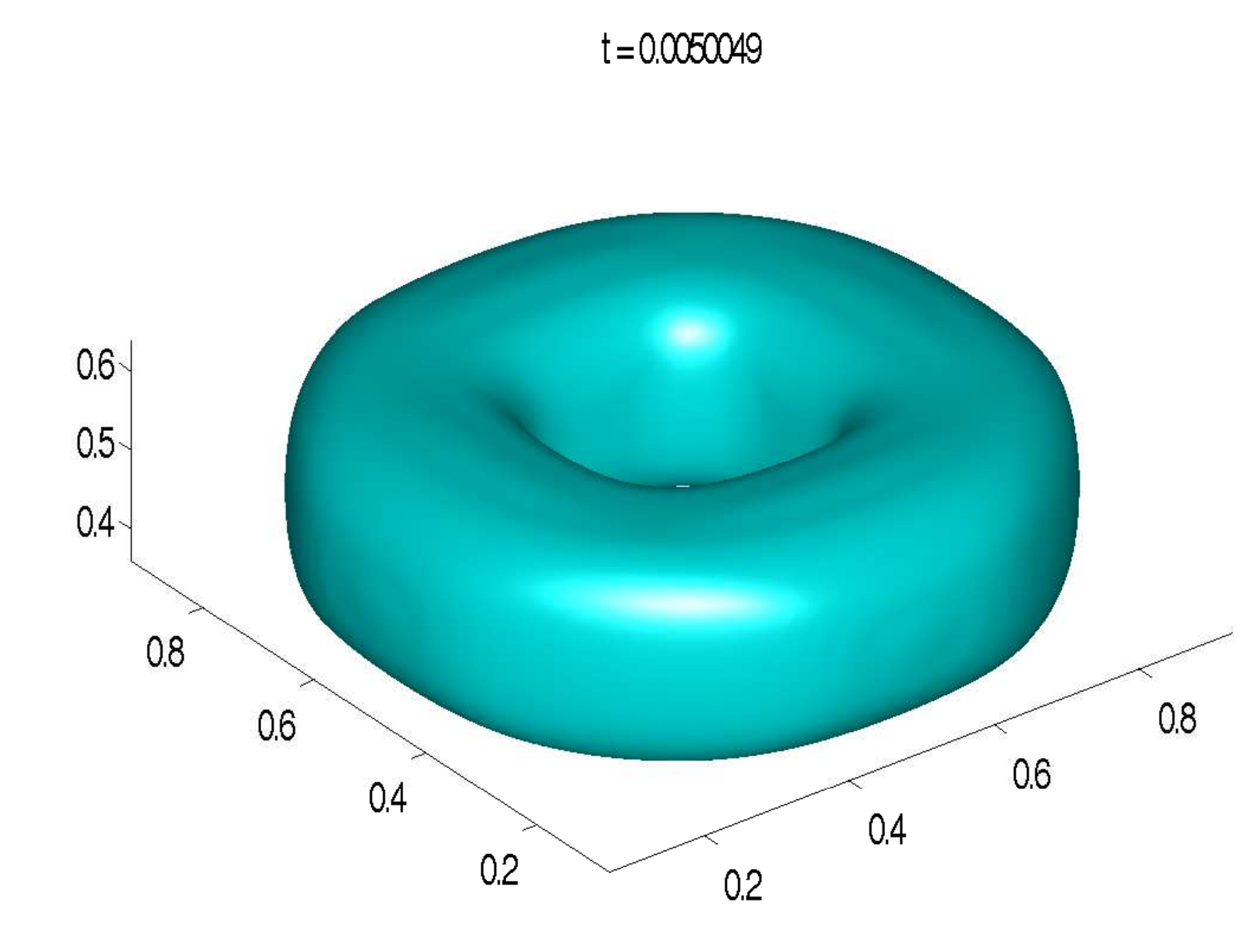}
\includegraphics[width=4cm]{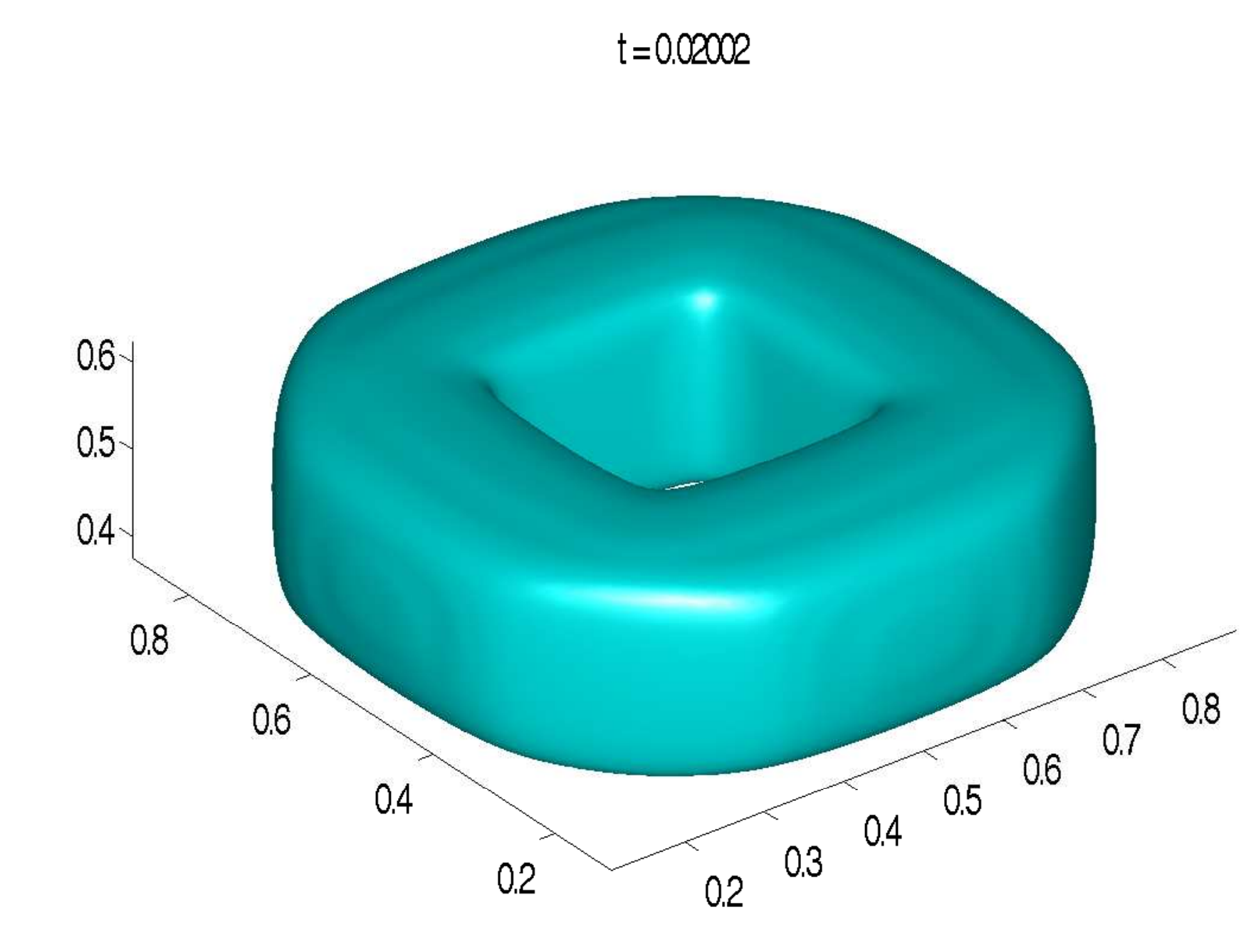}
\includegraphics[width=4cm]{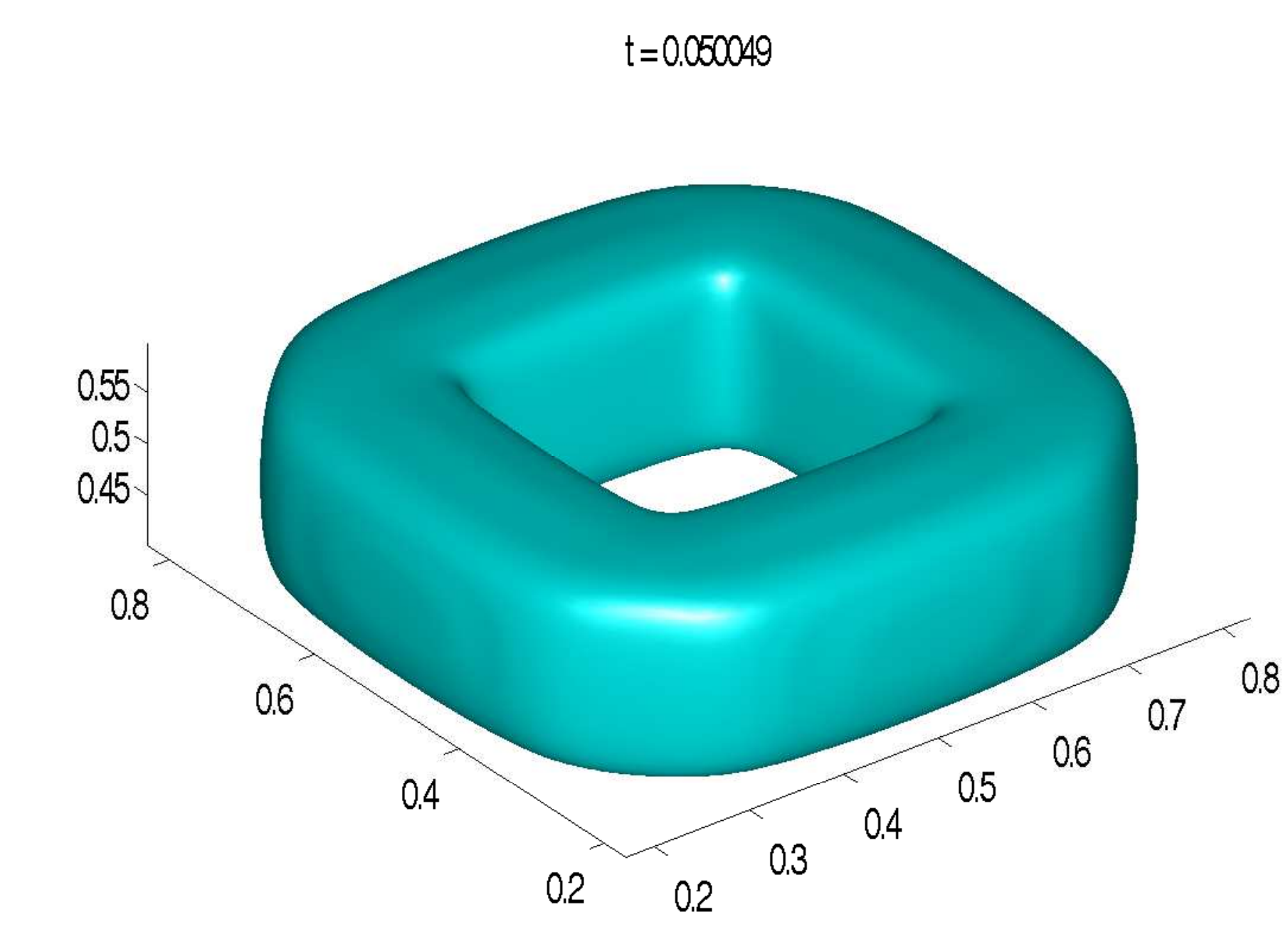}
\includegraphics[width=4cm]{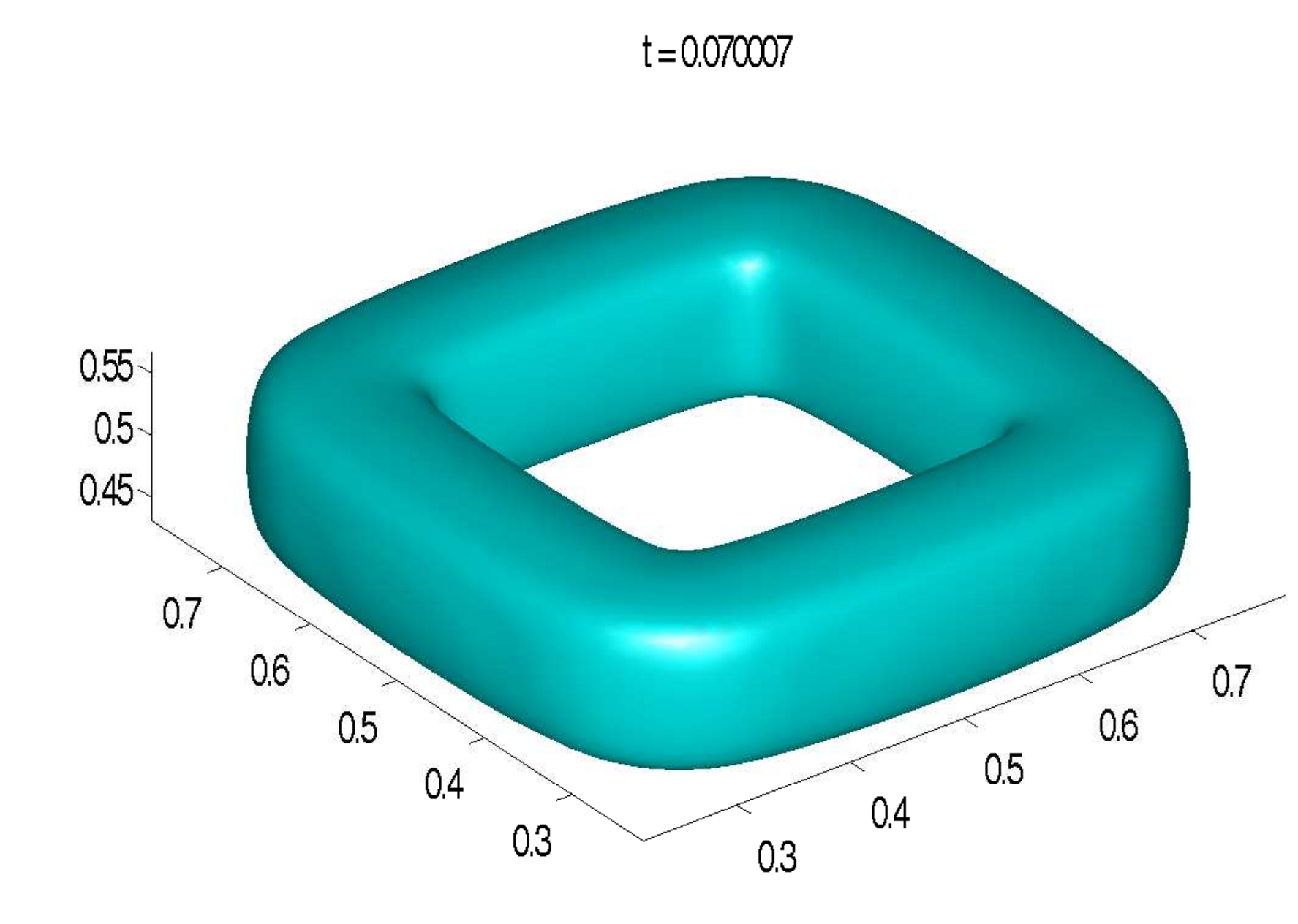}
\includegraphics[width=4cm]{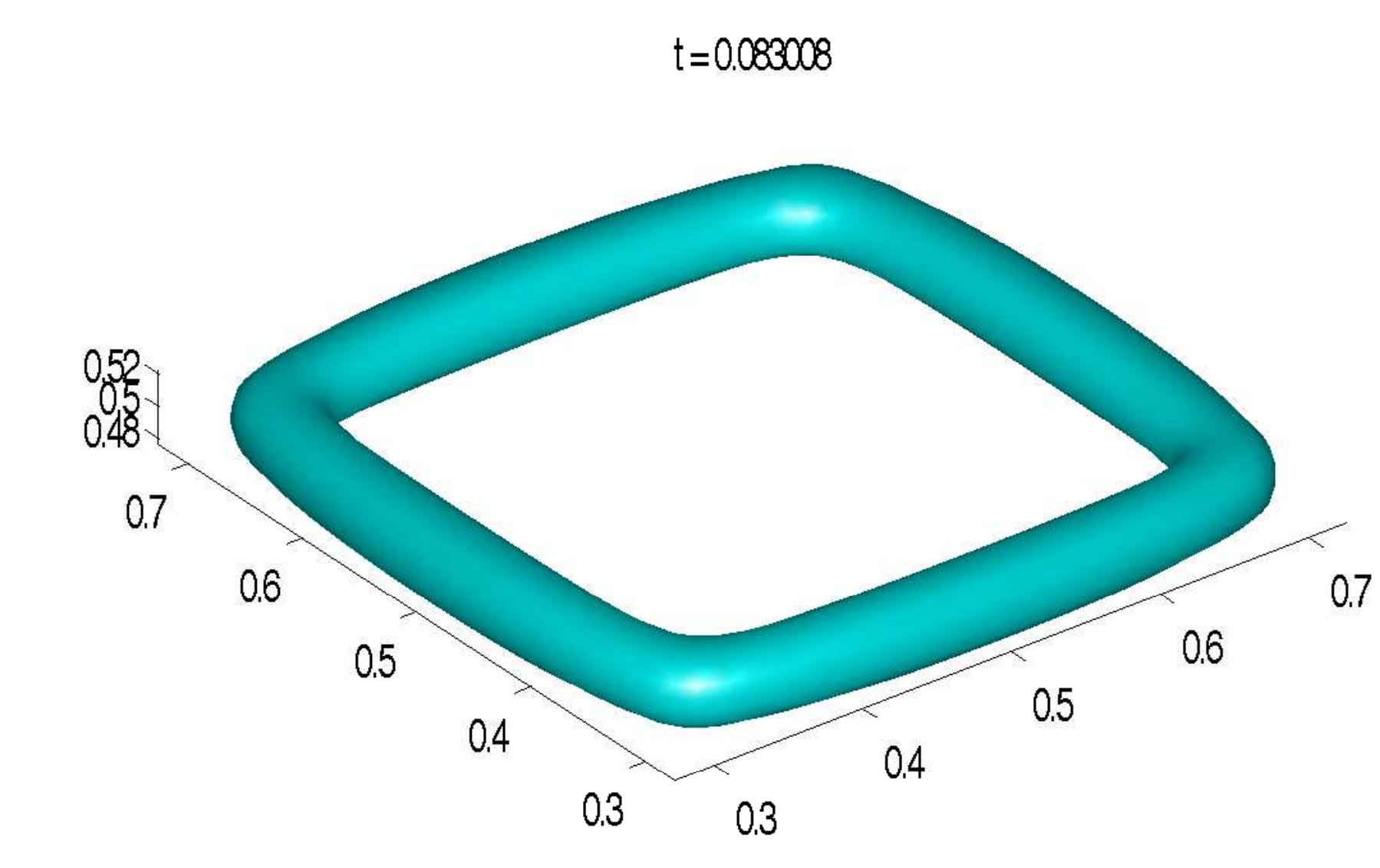}
  \caption{$\phi_5^o(\xi)$-evolution from an initial torus, at different times}
  \label{fig:test3d_ani_tore3} 
 \end{center}
  \end{figure}

\bibliographystyle{abbrv}
\bibliography{biblio}

\end{document}